\documentclass[11pt, reqno]{amsart}
\usepackage{amsmath,amsthm,amsfonts,amssymb,amscd}
\usepackage{graphicx}
\usepackage[T1]{fontenc}
\usepackage{epsfig}
\input{epsf.tex}
\newtheorem{theorem}{Theorem}[section]
\newtheorem{proposition}[theorem]{Proposition}
\newtheorem{lemma}[theorem]{Lemma}
\newtheorem{corollary}[theorem]{Corollary}
\newtheorem{define}[theorem]{Definition}
\newtheorem{remark}[theorem]{Remark}

\def\Empty{}

\catcode`\@=11

\def\section{\@startsection {section}{1}{\z@}{-3.5ex plus -1ex minus
-.2ex}{2.3ex plus .2ex}{\large\bf}}

\def\fnum@figure{{\small Figure \thefigure}}
\def\fakefigure{\def\@captype{figure}}

\long\def\@makecaption#1#2{
    \vskip 10pt
    \def\FCap{#2} \def\NoCap{\ignorespaces}
    \ifx \FCap\NoCap
       \setbox\@tempboxa\hbox{#1}  
      \else
       \setbox\@tempboxa\hbox{#1: \small \it #2}
    \fi
    \ifdim \wd\@tempboxa >\hsize   
        \unhbox\@tempboxa\par      
      \else                        
        \hbox to\hsize{\hfil\box\@tempboxa\hfil}
    \fi}

\def\@oddhead{\hbox{}\rightmark \hfil \rm\thepage}
\def\sectionmark#1{\markright {\sc{\ifnum \c@secnumdepth >\z@
      \S\thesection.\hskip 1em\relax \fi #1}}}

\catcode`\@=12

\def\oplabel#1{
  \def\OpArg{#1} \ifx \OpArg\Empty {} \else
  	\label{#1}
  \fi}

\newlength{\saveu}

\catcode`\@=11

\newcommand{\aas}{\mbox{${\mathcal A}^s$}}
\newcommand{\aau}{\mbox{${\mathcal A}^u$}}

\newcommand{\cH}{\mbox{${\mathcal H}$}}

\newcommand{\hhu}{\mbox{${\mathcal H} ^u$}}
\newcommand{\hhs}{\mbox{${\mathcal H} ^s$}}

\newcommand{\cc}{\mbox{$\mathcal C$}}

\newcommand{\oo}{\mbox{$\mathcal O$}}
\newcommand{\oos}{\mbox{${\mathcal O}^s$}}
\newcommand{\oou}{\mbox{${\mathcal O}^u$}}

\newcommand{\rrrr}{\mbox{${\mathbb R}$}}

\newcommand{\mi}{\mbox{$\widetilde M$}}

\newcommand{\ls}{\mbox{$\Lambda^s$}}
\newcommand{\lu}{\mbox{$\Lambda^u$}}

\newcommand{\wls}{\mbox{$\widetilde \Lambda^s$}}
\newcommand{\wlu}{\mbox{$\widetilde \Lambda^u$}}

\newcommand{\wwp}{\mbox{$\widetilde \Phi$}}

\newcommand{\wt}{\mbox{$\widetilde T$}}

\newcommand{\fol}{\mbox{$\mathcal F$}}

\newcommand{\fn}{\mbox{$\widetilde {\mathcal F}$}}

\newcommand{\ws}{\mbox{$\widetilde  W^s$}}
\newcommand{\wu}{\mbox{$\widetilde  W^u$}}

\newcommand{\hatp}{\mbox{$\hat P$}}
\newcommand{\hatflo}{\mbox{$\hat \Phi$}}
\newcommand{\hatwls}{\mbox{$\hat \Lambda^s$}}
\newcommand{\hatwlu}{\mbox{$\hat \Lambda^u$}}

\catcode`\@=12

\begin{document}

\title[\centerline{Free Seifert pieces of pseudo-Anosov flows}]{Free Seifert pieces of pseudo-Anosov flows}

\author{Thierry Barbot}

\author{S\'{e}rgio R. Fenley}
\thanks{Research of the second author partially
supported by a grant of the Simons foundation.}

\address{Thierry Barbot\\ Universit\'e d'Avignon et des pays de Vaucluse\\
LANLG, Facult\'e des Sciences\\
33 rue Louis Pasteur\\
84000 Avignon, France.}

\email{thierry.barbot@univ-avignon.fr}

\address{S\'ergio Fenley\\Florida State University\\
Tallahassee\\FL 32306-4510, USA \ \ and \ \
Princeton University\\Princeton\\NJ 08544-1000, USA}

\email{fenley@math.princeton.edu}

 \email{}
\maketitle

\vskip .2in

\small
\noindent
{{\underline {Abstract}} $-$ We prove a structure theorem for pseudo-Anosov flows restricted
to Seifert fibered pieces of three manifolds. The piece is called periodic if there is a
Seifert fibration so that a regular fiber is freely homotopic, up to powers, to a closed
orbit of the flow. A non periodic Seifert fibered piece is called free. In a previous paper \cite{bafe1} we
described the
structure of a pseudo-Anosov flow restricted to a
periodic piece up to isotopy along the flow.
In the present paper we consider free Seifert pieces.
We show that, in a carefully defined neighborhood of the free piece, the
pseudo-Anosov
flow is orbitally equivalent to a hyperbolic blow up of a geodesic flow piece.
A geodesic flow piece is a finite cover of the geodesic flow on
a compact hyperbolic
surface, usually with boundary. In the proof we introduce almost $k$-convergence groups and prove a convergence
theorem. We also introduce an alternative model for the geodesic flow of
a hyperbolic surface that is suitable
to prove these results, and we carefully define
what is a hyperbolic blow up.
}

\tableofcontents

\section{Introduction}

The purpose of this article is to prove a structure theorem for pseudo-Anosov flows
restricted to Seifert fibered pieces that are called free. This is part of a very broad program
to classify all pseudo-Anosov flows in $3$-manifolds.

A pseudo-Anosov flow is a flow without stationary orbits that is roughly transversely hyperbolic.
This means it has stable and unstable $2$-dimensional foliations, that may have singularities along
finitely many closed orbits. The singularities are of $p$-prong type.

These were introduced by Anosov \cite{An} who studied amongst other things geodesic flows in
the unit tangent bundle of negatively curved closed manifolds. These flows are now called Anosov flows,
and they are smooth. In dimension $3$ Thurston \cite{Th1,Th2} introduced suspension pseudo-Anosov flows,
that are suspensions of pseudo-Anosov homeomorphisms of
surfaces. These were used in an essential way to
prove the hyperbolization theorem of atoroidal $3$-manifolds that fiber over the circle.
Mosher \cite{Mo1,Mo2} then generalized this to define general pseudo-Anosov flows in $3$-manifolds, that
is, flows that
are locally like the suspension flows of pseudo-Anosov homeomorphisms.
Pseudo-Anosov flows are extremely common due to works of Thurston \cite{Th1,Th3,Th4},
Gabai and Mosher \cite{Mo3},
Calegari \cite{Cal1,Cal2} and the second author \cite{Fe1,Fe4}.
More recently B\'eguin, Bonatti, and Yu \cite{BBY} introduced extremely general new constructions
of Anosov flows in $3$-manifolds. These greatly expanded the class of Anosov flows in $3$-manifolds.
In addition, pseudo-Anosov flows have been used to analyze and understand the topology and geometry
of $3$-manifolds \cite{Cal1,Fe3,Fe4,Ga-Ka1,Ga-Ka2}.

Pseudo-Anosov flows are strongly connected with the topology of $3$-manifolds.
The existence of a pseudo-Anosov flow implies that the manifold is {\em irreducible} \cite{Fe3}, that
is, every embedded sphere bounds a ball \cite{He}. The other important property in $3$-manifolds
concerns $\pi_1$-injective tori. If a closed manifold admits this, one says that the
manifold is {\em toroidal}, otherwise it is called {\em atoroidal}.
The existence of a pseudo-Anosov flow {\underline {does not}} imply that the manifold
is atoroidal $-$ the simplest counterexample is the unit tangent bundle of a closed
surface of negative curvature, in which the geodesic flow is Anosov.
But in fact pseudo-Anosov flows in toroidal $3$-manifolds are extremely common,
see
for example
\cite{Fr-Wi,Ha-Th,Ba3,bafe1,BBY}.
One very important problem, that will not be addressed in this article is to determine
exactly which $3$-manifolds admit pseudo-Anosov flows.

The goal of this article is to advance the understanding of pseudo-Anosov flows
in relation with the topology of the $3$-manifold. A compact, irreducible $3$-manifold has
a canonical decomposition into Seifert fibered
and atoroidal pieces. This is the
JSJ decomposition of the manifold \cite{Ja-Sh,Jo}.
{\em Seifert fibered} means that it has a one dimensional foliation by circles.
We want to understand how a pseudo-Anosov flow
interacts with this decomposition. The atoroidal case is by far the most mysterious
and unknown and will not be addressed in this article. We will consider Seifert fibered pieces.
Here a lot is already known.
First, in a seminal work, Ghys \cite{Gh} proved that an Anosov flow in an $\mathbb S^1$ bundle
is orbitally equivalent to a finite cover of a geodesic flow.
{\em Orbitally equivalent} means that there is a homeomorphism that sends orbits to orbits.
Without explicitly defining it, Ghys introduced the notion of an $\mathbb R$-covered Anosov flow:
this means that (say) the stable foliation is $\mathbb R$-covered,
that is, when lifted to the universal cover, it is a foliation
with leaf space homeomorphic to the real numbers $\mathbb R$.
Later the first author \cite{Ba1} extended this result to any Seifert fibered space.

Previously the first author also started the analysis of the structure of an Anosov flow
restricted to a Seifert fibered space. He proved \cite{Ba3} that if the flow is $\mathbb R$-covered,
then the flow restricted to the Seifert
piece in an appropriate manner is orbitally equivalent
to a finite cover of a geodesic flow of a compact hyperbolic surface with boundary.
This is called a {\em geodesic flow piece}.
He also started the study of more general Anosov flows restricted to Seifert fibered pieces.
We emphasize that even for Anosov flows, the $\rrrr$-covered property
is extremely restrictive, and it does not allow blow ups.

In this article we consider the much more general case of pseudo-Anosov flows and study
the relationship with Seifert fibered pieces. Let $P$ be such a piece of the torus decomposition
of the manifold $M$.
The fundamental dichotomy here is the following: we say that the piece is {\em periodic} if
there is a Seifert fibration of $P$ so that a regular fiber of the fibration is up to finite
powers freely homotopic to a periodic orbit of the flow. Otherwise the piece is called
{\em free}. Geodesic flows have only one piece and it is free. The Handel-Thurston examples
are $\mathbb R$-covered with two Seifert pieces, both of which are free.
The Bonatti-Langevin examples \cite{Bo-La} have periodic Seifert pieces. We previously
constructed a very large class of examples in graph manifolds, so that all pieces are
periodic \cite{bafe1}, and in \cite{bafe2}, we proved that every pseudo Anosov flow on a graph manifold such that all Seifert pieces
are periodic is orbitally equivalent to one of the examples constructed in \cite{bafe1}.

The structure of a pseudo-Anosov flow in a periodic piece is fairly simple and was completely
determined in \cite{bafe1}. A {\em Birkhoff annulus} is an annulus tangent to the flow in
the boundary and transverse to the flow in the interior.
For a periodic piece there is a
$2$-dimensional  spine that is a union of a mostly embedded finite collection of Birkhoff annuli.
An arbitrarily small neighborhood of this spine is a representative for the Seifert fibered piece.
In that way the dynamics of the flow in the piece is extremely simple: there are finitely
many closed orbits entirely contained in the piece, their local stable and unstable
manifolds and every other orbit piece enters and exits the piece. In particular there
are no full non closed orbits contained in the piece.

As we mentioned above there are non-trivial free Seifert pieces that are orbitally equivalent
to a geodesic flow piece. In those cases it follows that in the piece the dynamics is extremely rich:
there are countably many periodic orbits contained in the piece and uncountably
many non periodic full orbits  entirely contained in the piece.
Given the results mentioned above  the natural conjecture is that in a free Seifert
piece the flow is orbitally equivalent to a geodesic flow piece.
However, the situation is not nearly so simple. We briefly explain one example
to illustrate what could happen.

B\'eguin, Bonatti and Yu introduced new, powerful and vast constructions
of Anosov flows. Start with a geodesic flow in a closed hyperbolic surface
and consider an orbit that projects to a simple, separating geodesic.
The unit tangent bundles $M_1, M_2$ of the complements
$S_1, S_2$  of the geodesic in the surface
are Seifert fibered manifolds.
Blow up the orbit using a DA operation, derived from Anosov, and remove a torus
neighborhood of this orbit to create a manifold with boundary and an incoming
semi-flow in it.
Roughly the DA operation transforms a hyperbolic orbit into (in this
case) a repelling orbit. This can be achieved by keeping the
expansion along the unstable directions and adding an expansion
stronger than the contraction in the stable direction.
The local unstable leaf splits into two leaves (or one
if the leaf is a M\"{o}bius band) with a solid torus
in between. See the details in \cite{Fr-Wi,BBY} or in
section \ref{examples}.
Take the manifold with the boundary a torus and
an incoming semiflow.
Glue a copy of this manifold with a semiflow
which is a time reversal of the flow, so the flow is
outgoing from the other manifold. The
results of B\'eguin, Bonatti and Yu \cite{BBY} show that this flow is Anosov. The Seifert manifolds
$M_1, M_2$ survive in the final manifold and are Seifert fibered pieces of the JSJ decomposition
of this manifold.
They are also free Seifert pieces $-$ see detailed explanation
in section \ref{examples}.
But because of the blow up operation, the final flow restricted to $M_1$ or $M_2$ cannot
be orbitally equivalent to a geodesic flow piece.
For example consider $M_1$: the only possible geodesic flow
here would be the geodesic flow of $S_1$. The problem is that
we blew up one orbit corresponding to a boundary
geodesic of $S_1$. In terms of the stable foliation,
this means that the stable leaf of this geodesic is
blown into an interval of leaves. This behavior
is exactly what necessitates the hyperbolic blow
up operation in the statement of the Main theorem.

A {\em hyperbolic blow up} of a geodesic flow piece is obtained by essentially blowing
up a geodesic flow piece as above. We explain more in the Sketch of the proof subsection
below.
The primary goal of this article is to prove the following:

\vskip .15in
\noindent
{\bf {Main theorem}} $-$ Let $(M, \Phi)$ be a pseudo-Anosov flow. Let $P$ be a free Seifert piece in $M$. Assume that $P$ is not elementary, i.e. that $\pi_1(P)$ does
not contain a free abelian group of finite index. Then, in the intermediate cover $M_P$ associated to $\pi_1(P)$ there is a compact submanifold $\hat P$
bounded by embedded Birkhoff tori, such that the restriction of the lifted flow $\hat \Phi$ to $\hat P$ is orbitally equivalent to a hyperbolic blow up of a
geodesic flow.  This orbital equivalence preserves the restrictions of the weak stable and unstable foliations.
Moreover, $\hat P$ is almost unique up to isotopy along the lifted flow $\hat \Phi$:
if  $\hat P'$ is another compact submanifold bounded by embedded Birkhoff tori, and if $\hat P_*$, $\hat P'_*$ are the complements in $\hat P$, $\hat P'$ of the (finitely many) periodic orbits contained in $\partial \hat P$,
there is a continuous
map $t: \hat P_* \to \mathbb R$ such that the map from $\hat P_*$ into $M_P$ mapping $x$ on $\hat{\Phi}^{t(x)}(x)$ is a homeomorphism, with image $\hat P'_*$.
\vskip .15in

See section \ref{conclusion} and Theorem \ref{main} for a more detailed statement.
The isotopy along the flow does not extend in general
to the tangent periodic orbits (see Remark \ref{rk:isotopybirkhofftori}).
A Birkhoff torus is one that is a union of Birkhoff annuli. We need
the cover $M_P$ because in $M$ the projection of the embedded Birkhoff
tori in $M_P$ may have tangent orbits that collapse together.
This is quite common. To get $\hat{P}$ embedded and the structure theorem
above, we need to lift to the cover $M_P$.

The main theorem substantially adds to the understanding of the relationship
of pseudo-Anosov flows with the topology of the manifold.
We emphasize that to show that the structure of periodic Seifert pieces
is given by a spine of Birkhoff annuli is relatively simple given
the understanding of the topological structure of the stable
and unstable foliations in the universal cover and the analysis
of periodic orbits and lozenges (see Background section).
Unlike the case of periodic Seifert pieces, the proof of the Main
theorem about free Seifert pieces is quite
complex and involves several new objects or constructions.

To put the Main theorem in perspective notice that pseudo-Anosov
flows are extremely common in $3$-manifolds. We already remarked
the very general recent constructions of Beguin-Bonatti-Yu of
Anosov flows \cite{BBY}. Very roughly they consider ``blocks"
with smooth semiflows where the non wandering set is hyperbolic
and the boundary is a union of transverse tori. Under extremely general
conditions they can glue these manifolds to produce Anosov flows.
In addition pseudo-Anosov flows are extremely common because
of Dehn surgery: they generate pseudo-Anosov flows in the surgered
manifolds for
the vast majority of Dehn surgeries on a closed orbit
of an initial pseudo-Anosov flow. Most of the time the resulting
flow is truly pseudo-Anosov, meaning it has $p$-prong singular
orbits.
Given this enormous flexibility it is quite remarkable that the
structure of a pseudo-Anosov flow in a Seifert fibered piece $P$ is
either described by a finite union of Birkhoff annuli (when
$P$ is periodic), or by the Main theorem (if $P$ is free).
In the course of the proof of the Main theorem, we will prove that
if $P$ is free, there is a representative $W$ for $P$, bounded
by Birkhoff annuli, so that it does not have $p$-prong singularities
in the interior. This is in contrast with the case that $P$ is
periodic, and this also highlights the remarkable fact that
possible singularities do not essentially affect the structure
of the pseudo-Anosov flow in a free Seifert piece.
The Main theorem implies that a great part of the enormous
flexibility of pseudo-Anosov flows resides in the cases that
either  $M$
is hyperbolic, or in the atoroidal pieces of the JSJ
decomposition of $M$.
The study of these is still in its infancy.

\vskip .1in
\noindent
{\bf {Sketch of the proof the main theorem}}

In order to prove the main theorem we will use a very important result that
under very general circumstances a flow is determined up to orbital equivalence by its action on the orbit space \cite{Hae}. The orbit space
is the quotient space of the flow in the universal cover.
In our situation the orbit space of the flow or a subset of the
flow will always be a subset of the plane and it will have
induced stable and unstable (possibly singular) foliations.

We introduce a new model to describe geodesic flows in compact hyperbolic
surfaces, usually with boundary,
 that involves the projectivized tangent bundle of an associated
orbit space. The fundamental group of the manifold acts in a properly
discontinuous cocompact way providing a new model for geodesic flows.

We now consider blow ups.
Instead of doing the blow up of the geodesic flow on the
$3$-manifold level we will do it
on the level of actions of the fundamental group on the circle or the line as follows.
For the geodesic flow, the stable/unstable foliations are $\mathbb R$-covered.
The quotient of the stable or unstable
leaf space (homeomorphic to  $\mathbb R$)
  by the representative $h$ of the regular fiber
is a circle $\mathbb S^1$ and the orbifold quotient fundamental group
$\pi_1(P)/<h>$ acts on these circles. These are convergence group actions.
When we lift the flow to a finite cover, the associated actions are not
convergence group actions, but are what we call $k$-convergence group
actions. If the unrolling of the fiber direction has order $k$, then
an element of $\pi_1(P)/<h>$ with fixed points in $\mathbb S^1$ has
$2k$ fixed points that are alternatively attracting and repelling.
Recall that for a convergence group action in $\mathbb S^1$
we can only have $2$ fixed points.
We prove a $k$-convergence group action theorem, extending the convergence group
theorem. This is fairly straightforward. Then we define almost $k$-convergence
groups, allowing modifications of the actions in some periodic
intervals.

To prove our theorem we will only allow hyperbolic blow ups, where the
modifications in the intervals introduce an arbitrary finite number of
points that are attracting or repelling.
The reason is that the stable and unstable foliations of
pseudo-Anosov flows have attracting or repelling holonomy
along periodic orbits.
We show that such an action is semiconjugate to a $k$-convergence action.

We build model flows via such actions. We start with
two hyperbolic blow ups of Fuchsian actions. Here a Fuchsian action is one associated
with a finite cover of the geodesic flow.
Using a combination of the two actions
we then construct  an ``orbit'' space which is a subset
of the plane, and then
a model flow with a compact
``core'' associated with this blow up.
In particular the hyperbolic blow up is a new technique to
construct flows that are later proved to be a part of
pseudo-Anosov flows.

Now consider an {\underline {arbitrary}} pseudo-Anosov $\Phi$,
and $P$ a free Seifert piece of $\Phi$.
Then the fiber $h$ in $\pi_1(P)$ acts on the leaf space of the stable/unstable foliations
freely, generating two axes $\mathcal A^s, \mathcal A^u$ that are homeomorphic to the reals.
The fundamental group of the piece $\pi_1(P)$ acts on $\mathcal A^s, \mathcal A^u$ and their
quotients by $h$, producing two actions on the circle $\mathbb S^1$. After a lot
of work we show that these two actions are hyperbolic blow ups of Fuchsian
actions. Using the hyperbolic blow ups of the two Fuchsian actions,
we construct
the associated model flow
as described above.
Finally we show that in the cover $M_P$ and in
the carefully defined subset $\hat{P}$ of $M_P$, the restriction of the
pseudo-Anosov flow $\Phi$
is orbitally equivalent to the model flow we constructed,
finishing the proof
of the main theorem.

The reader may feel at first glance that our construction of a hyperbolic
blow up through actions on the line is unnecessarily sophisticated,
and may consider that a definition involving DA operations
in dimension three
as explained in this introduction would have been more natural.
But if one uses the DA operations, then it
is for instance extremely difficult to analyze the flow or
establish the structure of the flow up to orbital equivalence.
In other words, we have to prove that {\underline {every}}
free piece of an arbitrary pseudo-Anosov flow
has the structure we are proposing and that is
very difficult if one considers just DA operations on the flow level.
Our approach
makes a much more direct connection between the construction of examples and the proof that every free Seifert piece has this form.
In the last section of this article (called Examples)
we provide several ways to exhibit free Seifert pieces of pseudo-Anosov flows with relatively simple
constructions.

\vskip .1in
\noindent
{\bf {What was previously known}}

As we mentioned before,  the first author \cite{Ba3} proved
a similar result in the case that $P$ is a free Seifert piece
of an $\mathbb R$-covered Anosov flow. In \cite{Ba3} the first author
also started the study of free Seifert pieces of general Anosov flows.
In particular in this article we use some of the constructions and proofs
of \cite{Ba3} or \cite{bafe1}.
However, even in the case of $\rrrr$-covered Anosov
flows, our Main Theorem is a refinement of the
main result in \cite{Ba3}: there, the
fact that the orbital equivalence preserves also the weak stable and unstable foliations was established only
outside the periodic orbits tangent to the boundary.
Also, as we explained before, in the case of $\rrrr$-covered Anosov flows
there are no blow ups and the analysis is much simpler.

The general strategy of the proof here is quite different. The use of almost
$k$-convergence groups is completely new.
The hyperbolic blow ups of actions is
also completely  new.
The alternative model of the geodesic flow is also new.
We expect that these objects introduced here will be useful in
other contexts.

\vskip .1in
\noindent
{\bf {The future}}

The results of this article can be used to understand/classify
pseudo-Anosov flows. For example suppose that $\Phi$ is
a pseudo-Anosov flow in $M$ that is a {\em {graph manifold}}.
This means that all pieces of the JSJ decomposition are
Seifert fibered. This is a large and very important class of
$3$-manifolds. The results of \cite{bafe1} and of this article
show that one can understand the flow $\Phi$ up to orbital
equivalence in carefully
defined neighborhoods of each Seifert piece $P$ of $M$.
There are infinitely many different gluing maps, but the next
goal is to show that up to orbit equivalence there are only
finitely many pseudo-Anosov flows in a fixed graph manifold.
This would be a substantial result.

\section{Background}

\noindent
{\bf {$p$-prong}}

Let $\mathbb R^2_{1/2}$ be the quotient of $\mathbb R^2$ by $-id$, and let $\beta: \mathbb R^2_p \to \mathbb R^2_{1/2}$ be the
finite $p$-covering over $\mathbb R^2_{1/2}$ branched over the origin $0:$
in complex coordinates, one can simply define $\beta$ as the map $z \mapsto z^{p/2}.$ Let $\tau: \mathbb R^2_p \to \mathbb R^2_p$
be the rotation by $2\pi/p:$ it is a generator of the Galois group of $\beta.$ We denote by $\mathfrak P_0^s$, $\mathfrak P_0^u$
the preimage in $\mathbb R^2_p$ of the ``vertical line" and the
``horizontal line" through $0$, respectively. Let $\lambda$ be a real number
(non necessarily positive) of modulus $> 1.$ Let $f_\lambda: \mathbb R^2_p \to \mathbb R^2_p$ be the only lift of the linear map
$(x,y) \mapsto (\lambda x, \lambda^{-1}y)$ such that:

-- (if $\lambda > 1$) $f_\lambda$ preserves every component of $\mathfrak P_0^s  -  0,$

-- (if $\lambda < -1$) $f_\lambda$ maps every component of $\mathfrak P_0^s  -  0$ onto its image by the rotation by $\pi/p.$

\begin{define}{(local model near a $p$-prong periodic orbit)}{}\label{def:modelprong}
For any integer $0 \leq k < p,$ the composition $\tau^k \circ f_\lambda$ is the model $p$-prong map of index $k$ and parameter $\lambda.$
We denote it by $f_{\lambda, k}.$
The suspension of $f_{\lambda, k},$ i.e. the quotient $M_{\lambda, k}$ of $\mathbb R^2_p \times \mathbb R$ by the transformation $(z, t) \mapsto (f_{\lambda, k}(z), t-1)$
equipped with the projection of the horizontal vector field $\frac\partial{\partial t},$ is the model $p$-prong vector field of index $k.$
The periodic orbit, projection of the line $\{ z=0 \},$ is the model $p$-prong periodic orbit of index $k.$
The projections of $\mathfrak P_0^s \times \mathbb R$, $\mathfrak P_0^u \times \mathbb R$ are respectively denoted by $\Lambda_0^s$, $\Lambda_0^u,$ and
called the stable (unstable) leaf of the model periodic orbit.
\end{define}

Observe that, up to topological conjugacy, $f_{\lambda, k}$ does not depend on $\lambda$, just on its sign. Similarly, a model $p$-prong only depends, up to orbital equivalence,
on $p,$ the index $k,$ and the sign of $\lambda.$

\vskip .1in
\noindent
{\bf {Pseudo-Anosov flows $-$ definitions}}

\begin{define}{(pseudo-Anosov flow)}
Let $\Phi$ be a flow on a closed 3-manifold $M$. We say
that $\Phi$ is a pseudo-Anosov flow if the following conditions are
satisfied:


- For each $x \in M$, the flow line $t \to \Phi(x,t)$ is $C^1$,
it is not a single point,
and the tangent vector bundle $D_t \Phi$ is $C^0$ in $M$.

- There are two (possibly) singular transverse
foliations $\ls, \lu$ which are two dimensional, with leaves saturated
by the flow and so that $\ls, \lu$ intersect
exactly along the flow lines of $\Phi$.

- There is a finite number (possibly zero) of periodic orbits,
called singular orbits such that in the neighborhood of each of them the flow is locally orbit equivalent to a model $p$-prong flow as defined in definition \ref{def:modelprong}, with $p \geq 3.$


- In a stable leaf all orbits are forward asymptotic,
in an unstable leaf all orbits are backwards asymptotic.
\end{define}

Basic references for pseudo-Anosov flows are \cite{Mo1,Mo2} and
\cite{An} for Anosov flows. A fundamental Remark is that the ambient manifold
supporting a pseudo-Anosov flow is necessarily irreducible - the
universal covering is homeomorphic to ${\mathbb R}^3$ (\cite{Fe-Mo}).
We stress that in our definition one prongs are not allowed.
Even if they will not appear in the present paper, we mention that there are however ``tranversely hyperbolic" flows with one prongs:

\begin{define}{(one prong pseudo-Anosov flows)}{}
A flow $\Phi$ is a one prong pseudo-Anosov flow in $M^3$ if it satisfies
all the conditions of the definition of pseudo-Anosov flows except
that the $p$-prong singularities  can also be
$1$-prong ($p = 1$).
\end{define}

\vskip .05in
\noindent
{\bf {Torus decomposition}}

Let $M$ be an irreducible closed $3$--manifold. If $M$ is orientable, it  has a unique (up to isotopy)
minimal collection of disjointly embedded incompressible tori such that each component of $M$
obtained by cutting along the tori is either atoroidal or Seifert-fibered \cite{Ja,Ja-Sh}
and the pieces are isotopically maximal with this property.
If $M$ is not orientable,
a similar conclusion holds; the decomposition has to be performed along tori, but also along
some incompressible embedded Klein bottles.

Hence the notion of maximal Seifert pieces in $M$ is well-defined up to isotopy. If $M$ admits
a pseudo-Anosov flow, we say that a Seifert piece $P$  is {\em periodic} if there is a
Seifert fibration on $P$ for which, up to finite powers, a regular
fiber is freely homotopic to a periodic orbit of $\Phi$. If not,
the piece is called {\em free.}

\vskip .05in
\noindent{\bf {Remark.}}
In a few circumstances, the Seifert fibration is not unique: it happens for example
when $P$ is homeomorphic to a twisted line bundle over the Klein bottle or
$P$ is $T^2 \times I$.
We stress out that our convention is to say that the Seifert piece is free
if
{\underline {no}} Seifert fibration in $P$ has fibers homotopic to a periodic orbit.

\vskip .2in
\noindent
{\bf {Birkhoff annuli}}

\begin{define}{(Birkhoff annulus)}
Let $A$ be an immersed annulus in $M$. We say that $A$ is a Birkhoff annulus
if the boundary $\partial A$ is a union of closed orbits of $\Phi$ and the
interior of $A$ is transverse to $\Phi$.
\end{define}

Observe that Birkhoff annuli are essentially topological objects: we may require differentiability at the (transverse) interior, but not
at the tangent periodic orbits. Since the interior of $A$ is transverse to $\Phi$, it is transverse to the
stable and unstable foliations, which then induce one dimensional foliations
in $A$ - each boundary component of $A$ is a leaf of each of these foliations.
We say that the Birkhoff annulus $A$ is {\em elementary} if each of these
foliations does not have a closed leaf in the interior of $A$.
We stress that $A$ need not be embedded in $M$. The easiest non embedded
example occurs for geodesic flows in the unit tangent bundle
$M = T^1 S$ where $S$ is a hyperbolic surface. If $\gamma$ is a non embedded
closed geodesic in $S$, consider the annulus $A$ obtained by turning the angles
along $\gamma$ from $0$ to $\pi$. Since $\gamma$ is not embedded,
this generates
a Birkhoff annulus that is not homotopic rel boundary to an embedded one.

\vskip .1in
\noindent
{\bf {Orbit space and leaf spaces of pseudo-Anosov flows}}

\vskip .05in
\noindent
\underline {Notation/definition:} \
We denote by $\pi: \mi \to M$ the universal covering of $M$, and by $\pi_1(M)$ the fundamental group of $M$,
considered as the group of deck transformations on $\mi$.
The singular
foliations lifted to $\mi$ are
denoted by $\wls, \wlu$.
If $x \in M$ let $W^s(x)$ denote the leaf of $\ls$ containing
$x$.  Similarly one defines $W^u(x)$
and in the
universal cover $\ws(x), \wu(x)$.
Similarly if $\tilde \theta$ is an orbit of $\Phi$ define
$W^s(\tilde \theta)$, etc...
Let also $\wwp$ be the lifted flow to $\mi$.

\vskip .05in

We review the results about the topology of
$\wls, \wlu$ that we will need.
We refer to \cite{Fe2,Fe3} for detailed definitions, explanations and
proofs.
The orbit space of $\wwp$ in
$\mi$ is homeomorphic to the plane $\rrrr^2$ \cite{Fe-Mo}
and is denoted by $\oo \cong \mi/\wwp$. There is an induced action of $\pi_1(M)$ on $\oo$.
Let

$$\Theta: \ \mi \ \rightarrow \ \oo \ \cong \ \rrrr^2$$

\noindent
be the projection map: it is naturally $\pi_1(M)$-equivariant.
If $L$ is a
leaf of $\wls$ or $\wlu$,
then $\Theta(L) \subset \oo$ is a tree which is either homeomorphic
to $\rrrr$ if $L$ is regular,
or is a union of $p$-rays all with the same starting point
if $L$ has a singular $p$-prong orbit.
The foliations $\wls, \wlu$ induce $\pi_1(M)$-invariant singular $1$-dimensional foliations
$\oos, \oou$ in $\oo$. Its leaves are $\Theta(L)$ as
above.
If $L$ is a leaf of $\wls$ or $\wlu$, then
a {\em sector} is a component of $\mi - L$.
Similarly for $\oos, \oou$.
If $B$ is any subset of $\oo$, we denote by $B \times \rrrr$
the set $\Theta^{-1}(B)$.
The same notation $B \times \rrrr$ will be used for
any subset $B$ of $\mi$: it will just be the union
of all flow lines through points of $B$.
We stress that for pseudo-Anosov flows there are at least
$3$-prongs in any singular orbit ($p \geq 3$).
For example, the fact that the orbit space in $\mi$ is
a $2$-manifold is not true in general if one allows
$1$-prongs.

\begin{define}
Let $L$ be a leaf of $\wls$ or $\wlu$. A slice of $L$ is
$l \times \rrrr$ where $l$ is a properly embedded
copy of the reals in $\Theta(L)$. For instance if $L$
is regular then $L$ is its only slice. If a slice
is the boundary of a sector of $L$ then it is called
a line leaf of $L$.
If $a$ is a ray in $\Theta(L)$ then $A = a \times \rrrr$
is called a half leaf of $L$.
If $\zeta$ is an open segment in $\Theta(L)$
it defines a {\em flow band} $L_1$ of $L$
by $L_1 = \zeta \times \rrrr$.
We use the same terminology of slices and line leaves
for the foliations $\oos, \oou$ of $\oo$.
\end{define}

If $F \in \wls$ and $G \in \wlu$
then $F$ and $G$ intersect in at most one
orbit.

We abuse convention and call
a leaf $L$ of $\wls$ or $\wlu$ {\em periodic}
if there is a non-trivial covering translation
$\gamma$ of $\mi$ with $\gamma(L) = L$. This is equivalent
to $\pi(L)$ containing a periodic orbit of $\Phi$.
In the same way an orbit
$\tilde \theta$ of $\wwp$
is {\em periodic} if $\pi(\tilde \theta)$ is a periodic orbit
of $\Phi$. Observe that in general, the stabilizer of an element $\tilde \theta$
of $\oo$ is either trivial, or a cyclic subgroup of $\pi_1(M)$.

\vskip .2in
\noindent
{\bf {Leaf spaces of $\wls, \wlu$}}

Let $\hhs, \hhu$ be the leaf spaces of $\wls, \wlu$
respectively, with the respective quotient topology.
These are the same as the leaf spaces of $\oos, \oou$
respectively.
The spaces $\hhs, \hhu$ are inherently one dimensional
and they are simply connected.
It is essential
that there are no $1$-prongs.
For simplicity consider $\hhs$. Through each point
passes either a germ of an interval, if the stable leaf
is non singular; or a $p$-prong, if the stable leaf
has a $p$-prong singular orbit. In this way the space
$\hhs$ is ``treelike". In addition it may not be
Hausdorff, and this is extremely common, even for
Anosov flows \cite{Fe2,bafe1}. These spaces are what
is called a non Hausdorff tree. This was defined
in \cite{Fe5,Ro-St}. In those articles, group actions
on non Hausdorff trees are carefully studied, in
particular the case of free actions. Even more general
than non Hausdorff trees is the concept of ``order trees"
introduced by Gabai and Kazez \cite{Ga-Ka2}, where the local
model is any space with a total order.

\vskip .2in
\noindent
{\bf {Product regions}}

Suppose that a leaf $F \in \wls$ intersects two leaves
$G, H \in \wlu$ and so does $L \in \wls$.
Then $F, L, G, H$ form a {\em rectangle} in $\mi$, ie. every stable leaf between $F$ and $L$
intersects every unstable leaf between $G$ and $H$. In particular,
there is  no singularity
in the interior of the rectangle \cite{Fe3}.

\begin{define}{}{}
Suppose $A$ is a flow band in a leaf of $\wls$.
Suppose that for each orbit $\tilde \theta$ of $\wwp$ in $A$ there is a
half leaf $B_{\tilde \theta}$ of $\wu(\tilde \theta)$ defined by $\tilde \theta$ so that:
for any two orbits $\tilde \theta', \tilde \theta''$ in $A$ then
a stable leaf intersects $B_{\tilde \theta'}$ if and only if
it intersects $B_{\tilde \theta''}$.
%
%
This defines a stable product region which is the union
of the $B_{\tilde \theta}$.
Similarly define unstable product regions.
\label{defsta}
\end{define}

The main property of product regions is the following:
for any product region $P$, and
for any $F \in \wls$, $G \in \wlu$ so that
$(i) \ F \cap P  \
\not = \ \emptyset \ \ {\rm  and} \ \
 (ii) \ G \cap P \ \not = \ \emptyset,
\ \ \ {\rm then} \ \
F \cap G \ \not = \ \emptyset$.
There are no singular orbits of
$\wwp$ in $P$.

\begin{theorem}{(\cite{Fe3})}{}
Let $\Phi$ be a pseudo-Anosov flow. Suppose that there is
a stable or unstable product region. Then $\Phi$ is
topologically equivalent to a suspension Anosov flow.
In particular $\Phi$ is non singular.
\label{prod}
\end{theorem}

We will occasionally use {\em product pseudo-Anosov flow}
as an abbreviation for {\em pseudo-Anosov flow
topologically equivalent to a suspension.}

\vskip .2in
\noindent
{\bf {Perfect fits, lozenges and scalloped chains}}

Recall that a foliation $\fol$ in $M$ is $\rrrr$-covered if the
leaf space of $\fn$ in $\mi$ is homeomorphic to the real line $\rrrr$
\cite{Fe1}.

\begin{figure}
  \centering
   \includegraphics[scale=0.8]{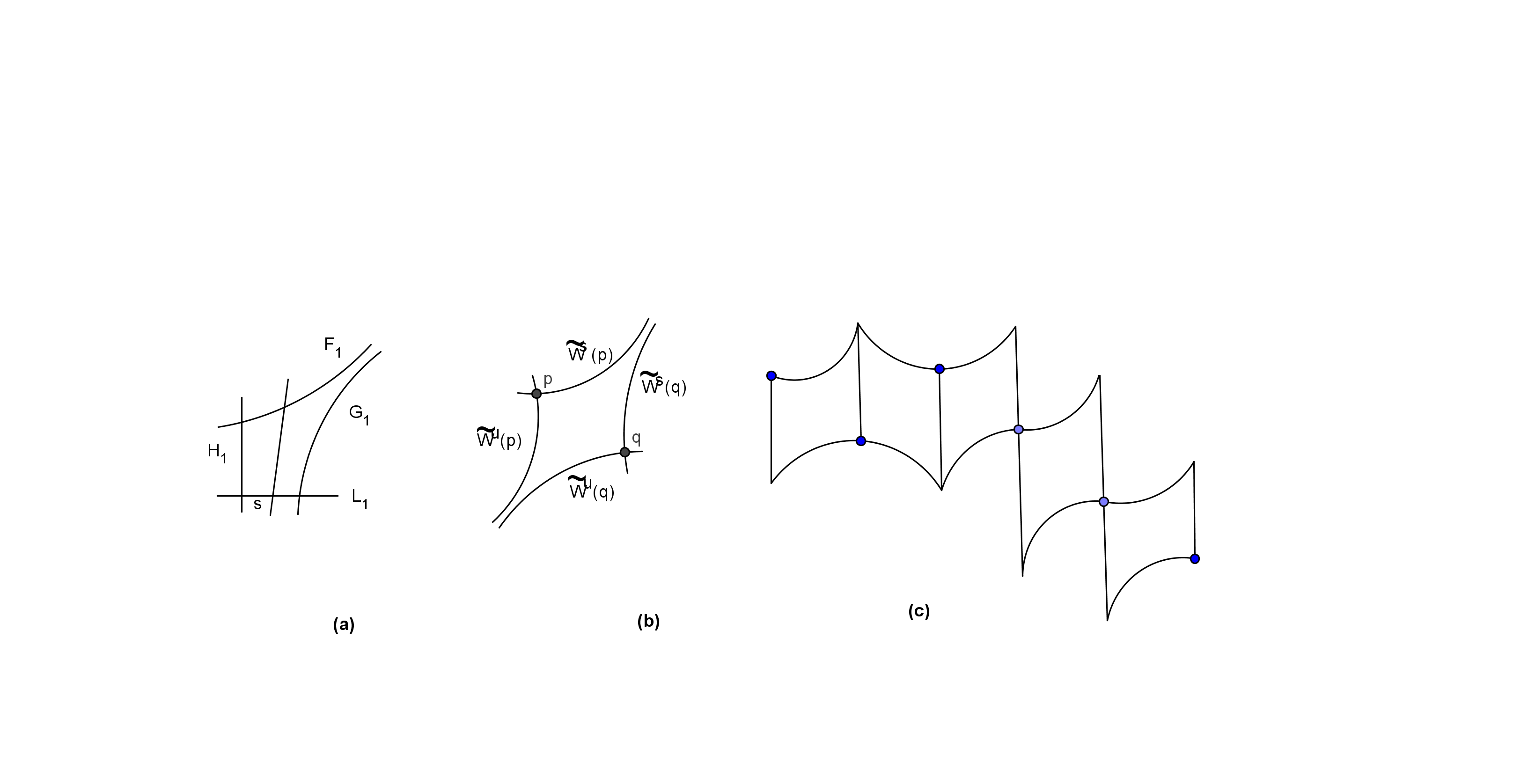}
\caption{a. Perfect fits in $\mi$,
b. A lozenge, c. A chain of lozenges.}
\label{loz}
\end{figure}

\begin{define}{(\cite{Fe2,Fe3})}\label{def:pfits}
Perfect fits -
Two leaves $F \in \wls$ and $G \in \wlu$, form
a perfect fit if $F \cap G = \emptyset$ and there
are half leaves $F_1$ of $F$ and $G_1$ of $G$
and also flow bands $L_1 \subset L \in \wls$ and
$H_1 \subset H \in \wlu$,
so that
%
the set

$$\overline F_1 \cup \overline H_1 \cup
\overline L_1 \cup \overline G_1$$

\noindent
separates $M$ and forms an a rectangle $R$ with a corner removed:
The joint structure of $\wls, \wlu$ in $R$ is that of
a rectangle with a corner orbit removed. The removed corner
corresponds to the perfect of $F$ and $G$ which
do not intersect.
\end{define}

We refer to fig. \ref{loz}, a for perfect fits.
There is a product structure in the interior of $R$: there are
two stable boundary sides and two unstable boundary
sides in $R$. An unstable
leaf intersects one stable boundary side (not in the corner) if
and only if it intersects the other stable boundary side
(not in the corner).
We also say that the leaves $F, G$ are {\em asymptotic}.

%
%
%
%

\begin{define}{(\cite{Fe2,Fe3})}\label{def:lozenge}
Lozenges - A lozenge $R$ is a region of $\mi$ whose closure
is homeomorphic to a rectangle with two corners removed.
More specifically two points $p, q$ define the corners
of a lozenge if there are half leaves $A, B$ of
$\ws(p), \wu(p)$ defined by $p$
and  $C, D$ half leaves of $\ws(q), \wu(q)$ defined by $p, q$, so
that $A$ and $D$ form a perfect fit and so do
$B$ and $C$. The region bounded by the lozenge
$R$ does not have any singularities.
%
%
%
%
The sides of $R$ are $A, B, C, D$.
The sides are not contained in the lozenge,
but are in the boundary of the lozenge.
There may be singularities in the boundary of the lozenge.
See fig. \ref{loz}, b.
\end{define}


There are no singularities in the lozenges,
which implies that
$R$ is an open region in $\mi$.


Two lozenges are {\em adjacent} if they share a corner and
there is a stable or unstable leaf
intersecting both of them, see fig. \ref{loz}, c.
Therefore they share a side.
A {\em chain of lozenges} is a collection $\cc = \{ C _i \},
i \in I$, where $I$ is an interval (finite or not) in ${\mathbb Z}$;
so that if $i, i+1 \in I$, then
$C_i$ and $C_{i+1}$ share
a corner, see fig. \ref{loz}, c.
Consecutive lozenges may be adjacent or not.
The chain is finite if $I$ is finite.

\begin{define}{(scalloped chain)}\label{def:scallop}
Let ${\mathcal C}$ be a chain of lozenges.
If any two
successive lozenges in the chain are adjacent along
one of their unstable sides (respectively stable sides),
then the chain is called {\em s-scalloped}
(respectively {\em u-scalloped}) (see
fig. \ref{pict} for an example of a $s$-scalloped chain).
Observe that a chain is s-scalloped if
and only if there is a stable leaf intersecting all the
lozenges in the chain. Similarly, a chain is u-scalloped
if and only if there is an unstable leaf intersecting
all the lozenges in the chain.
The chains may be infinite.
A scalloped chain is a chain that is either $s$-scalloped or
$u$-scalloped.
\end{define}

\begin{figure}
\includegraphics[scale=1]{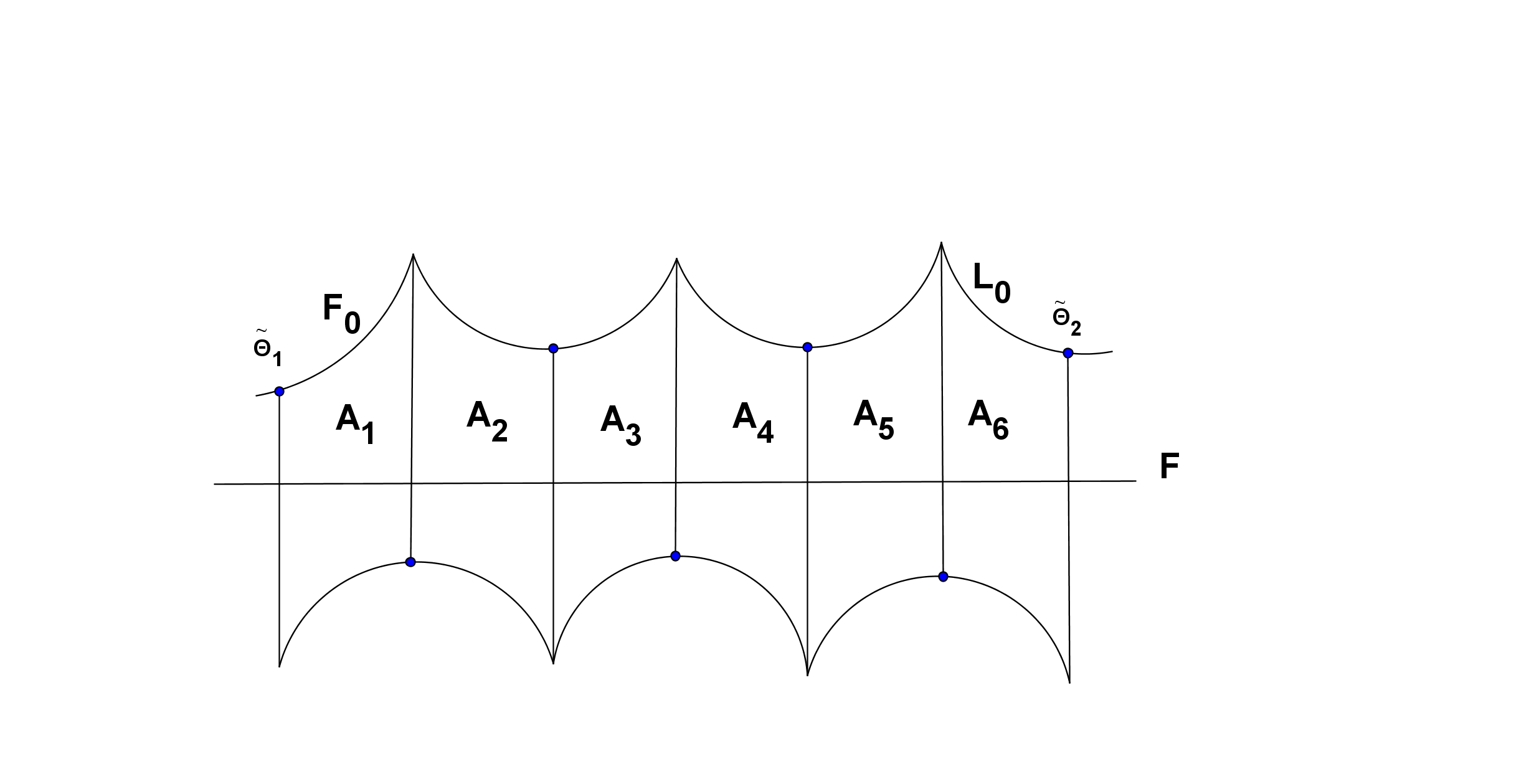}
\caption{
A partial view of a scalloped region.
Here $F, F_0, L_0$ are stable leaves,
so this is a s-scalloped region.}
\label{pict}
\end{figure}

For simplicity when considering scalloped chains we also include any half leaf which is
a boundary side of two of the lozenges in the chain. The union of these
is called a {\em {scalloped region}} which is then a connected set.

We say that two orbits $\tilde \theta_1, \tilde \theta_2$ of $\wwp$
(or the leaves $\ws(\tilde \theta_1), \ws(\tilde \theta_2)$)
are connected by a
chain of lozenges $\{ C_i \}, 1 \leq i \leq n$,
if $\tilde \theta_1$ is a corner of $C_1$ and $\tilde \theta_2$
is a corner of $C_n$.

\begin{theorem}{(\cite{Fe2,Fe3})}{}
Let $\Phi$ be a pseudo-Anosov flow in $M^3$ closed and let
$F_0 \not = F_1 \in \wls$.
Suppose that there is a non-trivial covering translation $\gamma$
with $\gamma(F_i) = F_i, i = 0,1$.
Let $\tilde \theta_i, i = 0,1$ be the periodic orbits of $\wwp$
in $F_i$ so that $\gamma(\tilde \theta_i) = \tilde \theta_i$.
Then $\tilde \theta_0$ and $\tilde \theta_1$ are connected
by a finite chain of lozenges
$\{ {C}_i \}, 1 \leq i \leq n,$ and $\gamma$
leaves invariant each lozenge
$C_i$ as well as their corners.
\label{chain}
\end{theorem}

In addition we {\underline {always assume without mention}} that
the chain is minimal: this means that there is no backtracking
and no three lozenges share a corner. That could happen
if $C_i, C_{i+1}$ share (say) a stable side (they are adjacent)
and $C_{i+1}, C_{i+2}$ share an unstable side (also adjacent),
with all three sharing a corner. In this case we eliminate
$C_{i+1}$ from the chain.
We could have a more complicated behavior if the corner
is a singular orbit with many lozenges abutting it.

The main result concerning non Hausdorff behavior in the leaf spaces
of $\wls, \wlu$ is the following:

\begin{theorem}{\cite{Fe2,Fe3}}\label{theb}
Let $\Phi$ be a pseudo-Anosov flow in $M^3$.
Suppose that $F \not = L$
are not separated in the leaf space of $\wls$.
Then $F$ is periodic and so is $L$. More precisely,
there is a non-trivial element $\gamma$ of $\pi_1(M)$ such that $\gamma(F)=F$ and $\gamma(L)=L$.
Moreover, let $\tilde \theta_1$, $\tilde \theta_2$ be the unique $\gamma$-fixed points in $F$, $L$, respectively.
Then, the chain of lozenges connecting $\tilde \theta_1$ to $\tilde \theta_2$ is s-scalloped.
\end{theorem}

\begin{remark}
\label{rk.lozengeannulus}
{\em A key fact, first observed in \cite{Ba3}, and extensively used in \cite{bafe1}, is the following:
the lifts in $\widetilde{M}$ of elementary Birkhoff annuli are precisely lozenges invariant by some cyclic subgroup of $\pi_1(M)$ (see \cite[Proposition $5.1$]{Ba3} for the case
of embedded Birkhoff annuli). It will also play a crucial role in the sequel.
More precisely: let $A$ be an elementary Birkhoff
annulus. We say that $A$ lifts to the
lozenge $C$ in $\mi$ if the interior
of $A$ has a lift which intersects
orbits only in $C$. It follows that
this lift intersects every orbit in $C$
exactly once and also that the two
boundary closed orbits of $A$ lift
to the
full corner orbits of $C$.
This uses the fact that a lozenge cannot be properly
contained
in another lozenge.}
\end{remark}

In particular
the following important property also follows: if $\theta_1$ and $ \theta_2$
are the periodic orbits in $\partial A$ (traversed in the
flow forward direction), then there are
{\underline {positive}} integers $n, m$ so that
$\theta_1^n$ is freely homotopic
to $(\theta_2^m)^{-1}$. We emphasize the free homotopy between inverses.

\begin{remark}
{\em According to Remark \ref{rk.lozengeannulus}, chains of lozenges
correspond to sequences of Birkhoff annuli,
every Birkhoff annulus sharing a common periodic orbit with the previous element
of the sequence, and also a periodic orbit with the next element in the sequence.
When the sequence closes up, it provides an immersion $f: T^2$ (or $K$,
in this article $K$ denotes the Klein bottle) $\to M$, which
is called a} Birkhoff torus {\em (if the cyclic sequence contains an even number of
Birkhoff annuli),  or a} Birkhoff Klein bottle {\em (in the other case).}
\end{remark}

\begin{remark}\label{rk:isotopybirkhofftori} $-$
{\bf {Birhoff annuli associated to the same lozenges.}} \ \
{\em
Let $A_1$, $A'_1$ be two elementary Birkhoff annuli whose lifts in $\widetilde{M}$
project to the same lozenge. Then, every orbit intersecting the interior of $A_1$ intersects $A'_1:$
there is a continuous map $t$ from the interior of $A_1$ into $\mathbb R$ such that the application $f_1$ mapping every $x$ in the interior
of $A_1$ into $\Phi^{t(x)}(x)$ realizes a homeomorphism between $A_1$ and $A'_1.$

However, in general, the map $t$ does not extend continuously on the boundary orbits of $A_1.$
In particular it may be that the continuous function
$t(x)$ is unbounded and accumulates to $\pm \infty$ as
$x$ approaches the boundary of the annulus along
certain paths or sequences of points.
This can be explicitly and easily constructed for
any Birkhoff annulus: more specifically {\underline {for any}}
elementary
Birkhoff annulus $A$ there is a Birkhoff annulus $A'$
so that $A, A'$ have lifts to the same lozenge in $\mi$
and so that the map $t$ from the interior of $A$ to the
interior of $A'$ does not extend continuously to the
boundary, meaning for example that $t(x)$ converges to plus
infinity as one approaches the boundary of $A$.

Furthermore, even if $f_1$ is defined on $\partial A_1,$ and if $A_2$ and $A'_2$ are other Birkhoff annuli,
one isotopic to the other along the flow, and adjacent to $A_1$, $A'_1,$ there is no reason for the map $f_2: A_2 \to A'_2$
to coincide with $f_1$ on the common boundary orbit.
The point is that we construct Birkhoff annuli from
lozenges in $\mi$ that are invariant by some
non trivial $g$ in $\pi_1(M)$.
In particular  these Birkhoff annuli are inherently only
topological objects, non differentiable, and
can have a very wild behavior near the tangent periodic orbit.
Consequently Birkhoff tori as defined in the
previous remark may be homotopic along
the flow only on the complement of the tangent periodic orbits.

Nevertheless, the following fact is true for embedded Birkhoff tori (the similar statement also holds for embedded Birkhoff Klein bottles
and is left to the reader).
Let $(M, \Phi)$, $(N, \Psi)$ be two pseudo-Anosov flows, and $T$, $T'$ embedded Birkhoff tori in respectively
$M$, $N.$ Let $\mathcal C$, $\mathcal C'$ be the associated chain of lozenges: they are preserved by subgroups $H$, $H'$ of $\pi_1(M)$, $\pi_1(N)$, both
isomorphic to $\mathbb Z^2$, and
corresponding to the fundamental groups of $T$, $T'$.
Assume that there is an equivariant (with respect to $H$ and $H'$)
 map $f: \mathcal C \to \mathcal C'.$ Then, we can replace $f$ by another map (still denoted by $f$) with the same properties but furthermore preserving the foliations on the lozenges induced
 by the stable/unstable foliations. It induces
a continuous map $F$ between $T_{tr}$ and $T'_{tr}$, where $T_{tr}$ and $T'_{tr}$ are the complements in the tori of the tangent periodic orbits.
Furthermore, $F$ can be chosen so that it maps the foliations induced on $T_{tr}$ by the stable/unstable foliations $\Lambda^{s,u}(\Phi)$ onto
the foliations induced on $T'_{tr}$ by the stable/unstable foliations $\Lambda^{s,u}(\Psi)$.
This map $F$ may not extend to a continuous map on the entire $T$, but it induces a bijection $\varphi$ between the tangent periodic orbits of $T$ and
the tangent periodic orbits of $T'$. Assume moreover that every periodic orbit $\theta$ tangent to $T$ has the same number of prongs and the same index than $\varphi(\theta).$ Then,
there are neighborhoods $U_\theta$, $U'_\theta$
of respectively $\theta$, $\varphi(\theta)$ and an orbital equivalence $f_\theta: U_\theta \to U'_\theta$ between the restriction
of $\Phi$ to $U_\theta$ and the restriction of $\Psi$ to $U'_\theta$ such that this orbital equivalence maps the restrictions of $\Lambda^{s,u}(\Phi)$
to $U_\theta$ onto the restrictions of $\Lambda^{s,u}(\Phi)$ to $U'_\theta$.

\vskip .2in
\noindent
{\bf {Claim:}} \  There are tubular neighborhoods $W$
and $W'$ of respectively $T$, $T'$ such that the
restriction of $\Phi$ to $W$
is orbitally equivalent to the restriction of $\Psi$ to $W'$
(but this orbital equivalence does not necessarily maps $T$ onto $T'$). Moreover, this orbital equivalence maps the restriction to $W$
of the stable/unstable foliations $\Lambda^{s,u}(\Phi)$ onto the restriction to $W'$ of $\Lambda^{s,u}(\Psi)$.

\vskip .1in
The proof of the claim
is as follows: for every Birkhoff annulus $A$ of $T$, let $A'$ be an open relatively compact sub-annulus of $A$, and
let $U_A$, $U'_A$ be the
{\underline {open}} neighborhoods of $A'$, $F(A')$ made of points of the form $\Phi^t(x)$ (respectively $\Psi^t(x)$) for $t$ small
and $x$ in $A'$ or $F(A')$. Then there is an orbital equivalence $f_A: U_A \to U'_A$ between the restrictions of $\Phi$ and $\Psi$ to $U_A$ and $U'_A$,
which preserves the restrictions of the stable/unstable foliations. We can adjust $U_\theta$ so that:

-- the union of all $U_A$ and $U_\theta$ covers $T$, and the union of all $U'_A$ and $U'_\theta$ covers $T'$,

-- for different periodic orbits $\theta_1$ and $\theta_2$ the neighborhoods $U_{\theta_1}$ and $U_{\theta_2}$ are disjoint,

-- $U_\theta$ intersects $U_A$ if and only if $\theta$ is a boundary component of $A$,

-- the intersection with $U_\theta$ of any orbit $\nu$ of the restriction of $\Phi$ to $U_A$ is either empty, or a connected relatively
compact subset of $\nu$.

The last condition means that for any point $x$ in $U_A \cap U_\theta$, the orbit of $x$ under $\Phi$ escapes in the past and in the future
from $U_\theta$ still staying in $U_A$.

Then, the union of the domains $U_\theta$ and $U_A$, and the union of $U'_A$ and $U'_\theta$, are open neighborhoods $W$, $W'$ of respectively $T$, $T'.$
Introduce a partition of unity subordinate to the covering formed by $U_\theta$ and $U_A$: these are functions $\mu_A: W \to [0, 1]$ and
$\mu_\theta: W \to [0,1]$ such that:

-- $\mu_A$ vanishes outside $U_A,$

-- $\mu_\theta$ vanishes outside $U_\theta$,

-- for every $x$ in $W$, the sum of all $\mu_A(x)$ and $\mu_\theta(x)$ is $1$ (observe that all the terms of this sum vanish except maybe
at most two of them).

Finally, we equip $W$ and $W'$ with  Riemannian metrics, that we use to reparametrize the orbits of $\Phi$ and $\Psi$ by unit lentgh,
so that these local flows are now defined over the entire $\mathbb R$.

Then, for any $x$ in $W$, we define $g(x)$ as the average along the $\Psi$-orbit of $f_\theta(x)$ with weight $\mu_\theta(x)$ and $f_A(x)$ with weight $\mu_A(x)$
(observe that if, for example, $x$ does not lie in $U_A(x)$, the affected weight $\mu_A(x)$ vanishes. Hence

$-$ $g(x)$ is either $f_\theta(x)$
if $x$ lies in $U_\theta$ but $\mu_A(x)=0$ for any Birkhoff annulus $A$,

$-$ $g(x)$ is  $f_A(x)$ is $x$ lies in $U_A$ but no $U_\theta$,  and

$-$ $g(x)$ is an average
as defined if $x$ belongs to an intersection $U_A \cap U_\theta$).

This defines a map $g: W \to W'$ which has all the required properties of the orbital equivalence we seek for, except that it may fail to be injective along
orbits of $\Phi$. But since we required that orbits of points in $U_A \cap U_\theta$ escape from $U_\theta$ by staying in $U_A$, it is easy
to see that there is $t_0 > 0$ such that for any
$x$ in $W$ we have $g(\Phi^{t_0}(x)) \neq g(x)$. It then follows that, by a classical procedure of
averaging along orbits (explained for example in the proof of Proposition $3.25$ in \cite{Ba3}), we can modify $f$ so that it is an orbital equivalence
between the restrictions of $\Phi$ and $\Psi$ (but $W'$ may have changed).
This proves the claim.

\vskip .1in
We emphasize that in general we cannot choose $g$ that takes
$T$ to $T'$ with the properties above.

}
\end{remark}

\section{Actions on the circle}

In all this section, $\bar{\Gamma}$ will be a finitely generated group. We consider  representations
$\bar{\rho}: \bar{\Gamma} \to \mbox{Homeo}(\mathbb S^1)$, that we will assume most of the time to be
non-elementary (i.e. with no finite orbit). Then, there is a unique minimal invariant closed subset $\mu$
\cite{He-Hi}. We will
also assume
that the action of $\bar{\Gamma}$ restricted to $\mu$ \textit{almost commutes} with a homeomorphism
$\tau: \mu \to \mu$, of finite order $k$, which preserves the cyclic order induced by
the cyclic order on the circle. More precisely, we will require that:

-- for every $\bar{\gamma}$ in $\bar{\Gamma}$, we have, on $\mu$, $\bar{\rho}(\bar{\gamma}) \circ \tau = \tau \circ \bar{\rho}(\bar{\gamma})$ if $\bar{\rho}(\bar{\gamma})$ preserves the orientation, and
$\bar{\rho}(\bar{\gamma}) \circ \tau = \tau^{-1} \circ \bar{\rho}(\bar{\gamma})$ if $\bar{\rho}(\bar{\gamma})$
reverses the orientation;

-- for every $x$ in $\mu$, the (small) arcs $[x, \tau(x)]$
have the following properties:
\begin{itemize}
  \item $\mathbb S^1 = [x, \tau(x)] \cup [\tau(x), \tau^2(x)] \cup ... \cup [\tau^{k-1}(x), \tau^k(x)=x]$
  \item the open arcs $]\tau^i(x), \tau^{i+1}(x)[$ for $i = 0, 1, ... , k-1$ are pairwise disjoint.
\end{itemize}

In other words, $\tau$ looks like the rotation by $1/k$, but is only defined on $\mu$.

The small arcs $[x,\tau(x)]$ are uniquely defined unless $k = 2$.
If $k = 1$, then $\mu$ is the identity and
$]x,\tau(x)[ = \mathbb S^1 - \{ x \}$.
When $k = 2$, choose an orientation in
$\mathbb S^1$ and choose $[x,\tau(x)]$ to be the arc from $x$ to $\tau(x)$ in
this orientation.
In this case $\mu = \mu^{-1}$.
For all other $k$ the small arc $[x,\tau(x)]$ is uniquely defined
by the second condition on disjointness of the open arcs.

By $x < y < \tau(x)$ we mean that $y$ is in $[x,\tau(x)]$ and is not one
of the extremities, that is, $y$ is in $]x,\tau(x)[$.

We will also assume that $\bar{\rho}$ is \textit{$\mu$-faithful}, meaning that the restriction of $\bar{\Gamma}$ to $\mu$ is faithful.

A \textit{gap} is a connected component of $\mathbb S^1  -  \mu$. For every gap $I$ and for every element $\gamma$ of $\bar{\Gamma}$, we have
the dichotomy:
\begin{itemize}
  \item $\gamma I \cap I = \emptyset$
  \item $\gamma I = I$
\end{itemize}

If $I$ is disjoint from all its iterates $\gamma I$ for $\gamma \neq 1$, it is
called a \textit{wandering gap.} If not, $I$ is called a
\textit{periodic gap.}
The group $\bar{\Gamma}$ acts by permutations on the set $\mathcal I$ of gaps of $\mu$.
We will denote by $\sigma(\bar{\gamma})$ the permutation
on $\mathcal I$ induced by $\bar{\rho}(\bar{\gamma})$. It commutes with the action induced by $\tau$
if $\bar{\rho} (\bar{\gamma})$ is orientation preserving. Observe that $\bar{\rho}$ is $\mu$-faithful
if and only if the morphism $\sigma: \bar{\Gamma} \to S(\mathcal I)$ is injective, where $S(\mathcal I)$ denotes
 the permutation group of gaps of $\mu$.

\begin{define}
Let $\mu$ be a closed perfect subset of $\mathbb S^1$, $\tau: \mu \to \mu$ a fixed point free homeomorphism of order $k$
that preserves cyclic order, and $\sigma: \bar{\Gamma} \to S(\mathcal I)$
a morphism, where $\mathcal I$ is the set of gaps of $\mu$.
A $(\mu, \tau, \sigma)$-representation
is a representation $\bar{\rho}: \bar{\Gamma} \to \mbox{Homeo}(\mathbb S^1)$
such that:

-- $\mu$ is the unique minimal invariant set in $\mathbb S^1$,

-- $\tau$ almost commutes with the restriction of the action to $\mu$,

-- the induced action on the set $\mathcal I$ of gaps is $\sigma$.
\end{define}

\subsection{Orbifold groups}
\label{sub:orbifoldcase}
In this paper, we mainly focus on the case of orbifold groups.
More precisely, let $\Gamma$ denote the fundamental group of the
Seifert fibered space $P$ with boundary.
The boundary components of $P$ are tori.
When the base orbifold $B$ is orientable, $\Gamma$ is generated by elements:
$$h, a_1, b_1, ... , a_g, b_g, d_1, ... , d_p, c_1, ... , c_q$$
satisfying the relations:
$$a_ih=h^{\epsilon}a_i, b_ih=h^{\epsilon}b_i,d_ih=hd_i, c_ih=hc_i, d_j^{\alpha_j}h^{\beta_j}=1, h^ec_1...c_q = [a_1, b_1]...[a_g, b_g]d_1...d_p$$
The integer $g$ is the genus of $B$, the number $\epsilon$ is $\pm 1$, according to the fiber-orientability of the Seifert bundle
along the appropriate curve in the base orbifold $B$. Every $d_j$ corresponds to a singularity of $B$ (of type $(\alpha_j, \beta_j)$)
and every $c_i$ corresponds to a boundary component.
Since $P$ has boundary it follows that $q \geq 1$.

When $B$ is non-orientable, $\Gamma$ is generated by elements
$$h, a_1,...,a_g,d_1, ... , d_p, c_1, ... , c_q$$
satisfying the relations:
$$a_ih=h^{\epsilon_i}a_i, d_ih=hd_i, c_ih=hc_i, d_j^{\alpha_j}h^{\beta_j}=1, h^ec_1...c_q=a_1^2...a_g^2d_1...d_p$$

\noindent
Here $g$ is the number of crosscaps needed to generate $B$.

In either case,
it follows that the fundamental group $\bar{\Gamma}$ of the base orbifold $B$, quotient of $\Gamma$ by the cyclic subgroup $H$ generated by $h$,  has the following presentation:
\begin{eqnarray}\label{eq:presentation}
  \langle a_1,b_1,...,a_g,b_g,d_1, ... , d_p, c_1, ... , c_q \ \mid \ d_j^{\alpha_j}=1, \; c_1...c_q = [a_1, b_1]...[a_g, b_g]d_1...d_p  \rangle  \mbox{ (when $B$ is orientable)} \\
  \langle a_1,...,a_g,d_1, ... , d_p, c_1, ... , c_q \ \mid \  d_j^{\alpha_j}=1, \; c_1...c_q=a_1^2...a_g^2d_1...d_p \rangle   \mbox{ (when $B$ is not orientable)}
\end{eqnarray}

We call $\bar{\Gamma}$ an \textit{orbifold group.}

We also need the following definition:

\begin{define}
Let $\tilde  \mu$ be a closed, perfect subset of $\mathbb R$,
$\tilde \tau: \tilde{\mu} \to \tilde{\mu}$
a fixed point free homeomorphism
order preserving, and $\tilde \sigma: G \to S(\widetilde{\mathcal I})$
a morphism, where $\widetilde{\mathcal I}$ is the set of gaps of $\tilde \mu$.
A $(\tilde \mu, \tilde \tau, \tilde \sigma)$-representation is a
representation $\rho: G \to \mbox{Homeo}(\mathbb R)$
such that:

-- $\tilde \mu$ is the unique minimal invariant set in $\mathbb R$,

-- $\tilde \tau$ almost commutes with the restriction of the action to
$\tilde \mu$,

-- the induced action on the set $\widetilde{\mathcal I}$ of gaps is $\tilde \sigma$.
\end{define}

\subsection{Modifying the action in a periodic gap}
From now on, $\bar{\Gamma}$ will be assumed to be an orbifold group.
In this section, we explain how it is possible to modify a $(\mu, \tau, \sigma)$-representation of $\bar{\Gamma}$ to another $(\mu, \tau, \sigma)$-representation which
essentially only differs on a  periodic gap, and with any new prescribed action on this periodic gap. More precisely:

\begin{proposition}\label{pro:blowupaction}
Let $\bar{\Gamma}$ be an orbifold group.
Let $\bar{\rho}: \bar{\Gamma} \to \mbox{Homeo}(\mathbb S^1)$ be a $(\mu, \tau, \sigma)$-representation, and let $I_0$
be a periodic gap of $\bar{\rho}(\bar{\Gamma})$. Suppose that
the stabilizer $Stab(I_0)$ of $I_0$ is generated by one of the generators $c_i$
of $\bar{\Gamma}$ as described in the previous section.
Denote by $\mathfrak J$ the union of all the iterates of $I_0$ by $\sigma(\bar{\Gamma})$.

Let $f_0$ be any homeomorphism of $I_0$, coinciding with $c_i$ on $\partial I_0$. Then, there is a new $(\mu, \tau, \sigma)$-representation
$\bar{\rho}': \bar{\Gamma} \to \mbox{Homeo}(\mathbb S^1)$ such that:

-- the action on the complement $\mathbb S^1  -  \mathfrak J$ is not modified i.e. coincides with the action induced by
$\bar{\rho}$,

-- the restriction of $\bar{\rho}'(c_i)$ on $I_0$ coincides with $f_0$.

We call such a representation \textbf{a modification of $\bar{\rho}$ on the gap $I_0$ by $f_0$.}

Furthermore, if $\bar{\rho}$ and $f_0$ are $C^k$, and $f_0$ coincides with $\bar{\rho}(c_i)$ near $\partial I_0$,
then the new representation $\bar{\rho}'$ is also $C^k,$ and for every $\gamma$ in $\bar{\Gamma}$ and every $r \leq k$, the $r$-derivatives of $\bar{\rho}(\gamma)$ and
$\bar{\rho}'(\gamma)$ coincide on $\mu$.
\end{proposition}

\begin{remark}\label{rk:topoconj}{\em
Suppose that $g_0$ is a homeomorphism of $I_0$ that is topologically conjugate to $f_0$ by some homeomorphism
$\varphi_0: I_0 \to I_0$ preserving the orientation. Let $\bar{\rho}'_1$ (respectively $\bar{\rho}'_2$) be a modification
of $\bar{\rho}$ on the gap $I_0$ by $f_0$ (respectively by $g_0$). Then, one can extend $\varphi_0$ to a homeomorphism
$\varphi: \bar{\rho}(\bar{\Gamma})(I_0) \to \bar{\rho}(\bar{\Gamma})(I_0)$ by
$\bar{\rho}(\bar{\Gamma})$-equivariance, and thereafter
to the
entire circle simply by requiring to be the identity map on $\mathbb S^1  -  \mathfrak J$.
This provides a topological conjugacy between $\bar{\rho}'_1$ and $\bar{\rho}'_2$.

In other words, modifications on $I_0$
are well defined up to topological conjugacy by the choice of $I_0$ and the topological conjugacy class of
$f_0$.
}
\end{remark}

\begin{remark}\label{rk:semiconj}{\em
More generally, let $g_0$ be a homeomorphism of $I_0$ which is topologically semi-conjugate to $f_0$. This means that there is
a continuous, surjective map
$\varphi_0: I_0 \to I_0$ that is weakly monotone and preserves the orientation,
and such that on $I_0$, we have:
$$g_0 \circ \varphi_0 = \varphi_0 \circ f_0$$
Then, one can show as above that, given any pair $\bar{\rho}'_1$ and $\bar{\rho}'_2$ of modifications
of $\bar{\rho}$ on the gap $I_0$ given respectively by $f_0$ and $g_0$, there is an extension of $\varphi_0$ on the entire circle
defining a semiconjugacy between $\bar{\rho}'_1$ and $\bar{\rho}'_2.$
}
\end{remark}


\noindent\textit{Proof of Proposition \ref{pro:blowupaction}.} The final statement in Proposition \ref{pro:blowupaction}, i.e. the fact that if $\bar{\rho}$ and $f_0$ are $C^k$, and that if $f_0$ coincides with $\bar{\rho}(c_i)$ near $\partial I_0$, then the new action is still $C^k$, follows easily in every case
considered in the proof below.

We can assume without loss of generality that $c_i$ is $c_1$. Let us first consider the case $q \geq 2$: the orbifold group
is then a free product of the cyclic subgroup generated by $c_1$ and a subgroup $\bar{\Gamma}_1$, where $\bar{\Gamma}_1$, in the orientable case,
is the subgroup generated by
$a_1$, $b_1$, $...$ , $a_g$, $b_g$, $d_1$, $...$ , $d_p$, $c_3$, ... , $c_q$,
and in the non-orientable case, the subgroup generated by
$a_1$, $...$ , $a_g$, $d_1$, $...$ , $d_p$, $c_3$, ... , $c_q$ (observe
that $c_1$ and $c_2$ have been removed).
In other words we remove $c_2$ from the generators of
$\bar{\Gamma}$  and $c_2$ can be recovered from the last relation
in either (1) or (2). When that is done, the only relations that need to be satisfied
for $\bar{\rho}(\bar{\Gamma})$ to be a representation are
$(\bar{\rho}(d_j))^{\alpha_j}   = 1, 1 \leq j \leq p$
(notice that $p$ may be $0$ in which case there are no relations at all).

The modification $\bar{\rho}'$ is then simply defined as the unique representation such that:

-- $\bar{\rho}'$ and $\bar{\rho}$ coincide on $\bar{\Gamma}_1$,

-- $\bar{\rho}'(c_1)$ is the map coinciding with $f_0$ on $I_0$, and equal to $\bar{\rho}(c_1)$ everywhere else.

Then the action of $\bar{\rho}(\bar{\Gamma})$
on the complement of $\mathfrak J$ is clearly not modified.
In addition the action of the subgroup $\bar{\Gamma}_1$ is also not modified, so the defining relations
are still satisfied. Therefore $\bar{\rho}'_1$ is still a representation and it is easy to see that
the other properties of the modification are satisfied.
The proposition follows in this case.

\vskip .15in
From now on we assume $q=1$, i.e. the orbifold $B$ has exactly one boundary component.

Consider the following generating set $\mathfrak S$ for $\bar{\Gamma}$:

-- when $B$ is orientable we put $\mathfrak S = \{ a_1^{\pm 1}, b_1^{\pm 1}, ... , a_g^{\pm 1}, b_g^{\pm 1}, d_1^{\pm 1}, ... , d_p^{\pm 1}\}$. Then:
$$c_1 = [a_1, b_1]...[a_g, b_g]d_1...d_p \;\;\; (*)$$

-- when $B$ is not orientable, we define $\mathfrak S = \{ a_1^{\pm 1},...,a_g^{\pm 1},d_1^{\pm 1}, ... , d_p^{\pm 1}\}$. We have:
$$c_1 = a_1^2...a_g^2d_1...d_p \;\;\; (**)$$

\noindent
In both cases, the only defining relations for $(\bar{\Gamma}, \mathfrak S)$ are $d_i^{\alpha_i}=1$.
In other words $\bar{\Gamma}$ is a free product of two groups. One group is a free group
generated by
either $a_1,b_1,...,a_g,...,b_g$ in the orientable  case or $a_1,...,a_g$ in the non orientable case.
The other group is freely generated by the torsion elements $d_1,...,d_p$.

In order to treat simultaneously the orientable and non-orientable cases,
let us write $w_0 = s_\ell... s_1$ for the
{\underline {word}}
$[a_1, b_1]...[a_g, b_g]d_1...d_p$ in the orientable case, and $a_1^2...a_g^2d_1...d_p$ in the non-orientable case. The crucial observation is that $w_0$ is the unique word with letters in $\mathfrak S$
representing $c_1$ and of minimal length. This is because $\bar{\Gamma}$ is a free product
of two groups as above and in either case $c_1$ is represented by the formulas
$(*)$ or $(**)$ above.
 Furthermore, since $w_0$ is cyclically reduced, it can also be easily
checked that for every integer $r$, any word with letters in $\mathfrak S$
representing a conjugate of $c_1^r$ has word-length $\geq r\ell$.

We define by induction, for every integer $i\geq1$:
$$I_i = \sigma(s_i)I_{i-1}$$

\centerline{\textit{Claim: for $0\leq i < j \leq \ell-1$, we have $I_i \neq I_j$.}}
\vspace{.5cm}

Indeed, if not, we would have an element $s_j...s_{i+1}$ of length
$<\ell$ whose image by $\bar{\rho}$ maps $I_i$ onto itself.
Therefore $s_j...s_{i+1}$ is a power of a conjugate of $c_1$.
As we have just observed, this is possible only if $s_j...s_{i+1}$ represents the trivial element, but then
we could write $c_1$ as a product of $\ell-(j-i)$ generators; contradiction.

\vskip .1in
Let us first consider the case $g>0$: the last letter $s_\ell$ is then $a_1$. In this case, our new action $\bar{\rho}'$ is obtained by applying essentially one and only one modification
to the generators: we only modify the restriction of $\bar{\rho}(s_\ell)$
to the arc $I_{\ell-1}$ (and hence $\bar{\rho}(a_1^{-1})$ on
$I_0$). More precisely, for every element $s$ of $\mathfrak S$ except for
$a_1=s_\ell$ and $a_1^{-1}$, we define
$\bar{\rho}'(s) = \bar{\rho}(s)$. We define $\bar{\rho}'(a_1)$ as follows:
outside $I_{\ell-1}$ we put $\bar{\rho}'(a_1)(x) = \bar{\rho}(a_1)(x)$,
and on $I_{\ell-1}$ we put
$\bar{\rho}'(a_1)(x) = \bar{\rho}'(s_\ell)(x) = f_0 \circ \bar{\rho}(s_{\ell-1}... s_1)^{-1}(x)$.
By the claim no $I_i = I_{l-1}$ or $I_0$  if $i < l-1$, therefore $\bar{\rho}'(a_1)$ is well defined.
Finally we define $\bar{\rho}'(a_1^{-1})$ to be the inverse of
$\bar{\rho}'(a_1)$.

Since the only relations in $\bar{\Gamma}$ are $d_i^{\alpha_i}=1$, and we have not changed the representation
on $d_i$, these prescriptions of $\bar{\rho}'$ on $\mathfrak S$ define a
representation $\bar{\rho}': \bar{\Gamma} \to \mbox{Homeo}(\mathbb S^1)$. It follows directly from our construction that
$\bar{\rho}$ and $\bar{\rho}'$ coincide outside $\mathfrak J$. Let us check the last statement to be proved, i.e. that the restriction
of $\bar{\rho}'(c_1)$ to $I_0$ is $f_0$. In the orientable case we have:
$$\bar{\rho}'(c_1) = \bar{\rho}'(a_1)\bar{\rho}'(b_1)\bar{\rho}'(a_1)^{-1} ... \bar{\rho}'(d_p)$$
Let $x$ be in $I_0$. The point $\bar{\rho}'(b_1)^{-1} ... \bar{\rho}'(d_p)(x)$ is equal to $\bar{\rho}(b_1)^{-1} ... \bar{\rho}(d_p)(x)$
since $\bar{\rho}$ and $\bar{\rho}'$ coincide on elements of $\mathfrak S  -  \{ a_1, a_1^{-1}\}$. This point belongs
to $I_{\ell-3}$. According to the Claim above, $I_{\ell-3}$ is different from $I_0$,
hence by our construction, $\bar{\rho}(a_1^{-1})$ and
$\bar{\rho}'(a_1^{-1})$ coincide on $I_{\ell-3}$: we have $\bar{\rho}'(a_1^{-1}) ... \bar{\rho}'(d_p)(x) =
\bar{\rho}(a_1^{-1}) ... \bar{\rho}(d_p)(x)$.
The equality $\bar{\rho}'(b_1)\bar{\rho}'(a_1^{-1}) ... \bar{\rho}'(d_p)(x) =
\bar{\rho}(b_1)\bar{\rho}(a_1^{-1}) ... \bar{\rho}(d_p)(x)$ follows.
Now $\bar{\rho}'(c_1)(x)$ is the image under $\bar{\rho}'(a_1)=\bar{\rho}'(s_\ell)$ of $\bar{\rho}(b_1)\bar{\rho}(a_1)^{-1} ... \bar{\rho}(d_p)(x)$. In addition this is equal to $\bar{\rho}(s_{l-1}...s_1)(x)$.
But we have defined the restriction of $\bar{\rho}'(s_\ell)(x)$ to be $f_0 \circ \bar{\rho}(s_{\ell-1}... s_1)^{-1}$. The equality
$\bar{\rho}'(c_1)(x) = f_0(x)$ follows.

The non-orientable case is treated in a similar way. We have:
$$\bar{\rho}'(c_1) = \bar{\rho}'(a_1)^2\bar{\rho}'(a_2)^2 ... \bar{\rho}'(d_p)$$

\noindent
because $\bar{\rho}'$ is a representation.
For $x$ in $I_0$, we still have the equality $\bar{\rho}'(a_2)^2 ... \bar{\rho}'(d_p)(x)=\bar{\rho}(a_2)^2 ... \bar{\rho}(d_p)(x)$,
one checks that this point lies in the region where $\bar{\rho}'(a_1)$ and $\bar{\rho}(a_1)$ coincide
because $I_{l-2}$ is not $I_{l-1}$.
The next occurence of
$\bar{\rho}'(a_1)$ has been designed so that it leads to the desired equality $\bar{\rho}'(c_1)(x) = f_0(x)$.

\vskip .1in
The last case to consider is the case $g=0$, $q=1$, ie. the case where the orbifold $B$ is a disk with a finite number $\geq 2$ of singular points.
This is because  if $B$ is non orientable then $g \geq 1$ as at least one
crosscap is needed to produce $B$.
Then $\mathfrak S$ is the collection $\{ d_1^{\pm 1}, ... , d_p^{\pm 1} \}$ of finite order elements. In this case we have $w_0 = d_1 ... d_p$. We
essentially do the same procedure as in the previous case: the first idea is to define
$\bar{\rho}'(d_i^{\pm 1}) = \bar{\rho}(d_i^{\pm 1})$ for every $i \geq 2$.
This implies that $\bar{\rho}'(d_i^{\pm 1}) = (\bar{\rho}(d_i))^{\pm 1}$ and
in addition $(\bar{\rho}'(d_i))^{\alpha_i} = 1$, for all $i \geq 2$.
Then we first define $\bar{\rho}'(d_1)(x) = \bar{\rho}(d_1)(x)$ everywhere except on the interval $I_{\ell-1} = d_1^{-1}(I_0)$
where we put
$\bar{\rho}'(d_1)(x) =  f_0(\bar{\rho}(d_p)^{-1} ... \bar{\rho}(d_2)^{-1}(x))$.
Notice it does \textbf{not} matter if we put $\bar{\rho}$ or $\bar{\rho}'$ here as they
are equal on $d_i, i \geq 2$.
 In this way  the equality $\bar{\rho}'(c_1)(x) = f_0(x)$ automatically
holds for $x \in I_0$. The problem is that after this change,
the relation $\bar{\rho}'(d_1)^{\alpha_1}=1$ is not satisfied.
This can be fixed by changing $\bar{\rho}(d_1)$ in another interval
as well in the following way:
Notice that $\alpha_1 \geq 2$.
In addition $I_{l-1} = \bar{\rho}(d_1^{\alpha_1 -1})(I_0)$ as $d_1^{\alpha_1} = 1$.
We want to define $\bar{\rho}'(d_1)$ on
$(\bar{\rho}(d_1))^{-1}(I_{l-1})$ so that $(\bar{\rho}'(d_1))^{\alpha_1} = 1$.
The only modification done so far has been on $\bar{\rho}'(d_1)$ in $I_{l-1}$.
The concern is that $\bar{\rho}(d_1^m(I_0)) = I_{l-1}$ for some
$m < \alpha_1 -1$.
But if that happens, $d_1^{m+1}$ is in $Stab(I_0)$ and is not trivial
in $\bar{\Gamma}$ because $m + 1 < \alpha_1$. By hypothesis on the
proposition this is not possible, because $Stab(I_0)$ is
generated by $c_1$ and no nontrivial element in the torsion subgroup
of $d_1$ is in the subgroup generated by $c_1$.
Therefore we can now define $\bar{\rho}'(d_1)$ on
$\bar{\rho}(d_1^{-1})(I_{l-1})$ (which would be equal to $I_0$ if
$\alpha_1 = 2$) so that $\bar{\rho}'(d_1^{\alpha_1}) = 1$.
This finishes the proof of Proposition \ref{pro:blowupaction}.
\begin{flushright}
    $\Box$
\end{flushright}

\medskip

The representations $\bar{\rho}: \bar{\Gamma} \to \mbox{Homeo}(\mathbb S^1)$
we will consider are always coming from a representation $\rho: \Gamma \to \mbox{Homeo}(\mathbb R)$
such that $\rho(h)$ is the translation by $+1$, i.e. the Galois covering for the covering map $\mathbb R \to \mathbb S^1$. The homeomorphism
$\tau$ is then the projection of a homeomorphism $\tilde{\tau}: \tilde{\mu} \to \tilde{\mu}$, where $\tilde{\mu}$ is the unique minimal closed
invariant subset of the action of $\rho(\Gamma)$, and we can assume wlog that $\tilde{\tau}$ is the restriction of the translation by $1/k$.

Observe that
the construction in Proposition \ref{pro:blowupaction} does not affect this property: this construction lifts to a modification $\rho^\ast: \Gamma \to \mbox{Homeo}(\mathbb R)$,
and the lifting is uniquely characterized by the property that the restrictions of $\rho(\Gamma)$ and $\rho^\ast(\Gamma)$ to $\tilde{\mu}$ coincide.

In the same way as in Proposition \ref{pro:blowupaction} we define
modifications of $(\tilde \mu,\tilde \tau, \tilde \sigma)$:

\begin{define}\label{blowupunivcov}
Let $\Gamma$ be the fundamental group of a Seifert manifold.
Let $\rho: \Gamma \to$ Homeo$(\mathbb R)$ be a
$(\tilde \mu, \tilde \tau, \tilde \sigma)$-representation, and let $I_0$
be a periodic gap of $\tilde{\mu}$, with stabilizer generated by some element $c_i$ of $\Gamma$.
Denote by $\mathfrak J$ the union of all the iterates of $I_0$ by $\sigma(\Gamma)$.

Let $f_0$ be any homeomorphism of $I_0$, coinciding with $c_i$ on $\partial I_0$.
Then, there is a new $(\tilde \mu, \tilde \tau, \tilde \sigma)$-representation
$\rho': \Gamma \to$ Homeo$(\mathbb R)$ such that:

-- the action on the complement $\mathbb R  -  \mathfrak I$ is not modified i.e. coincides with the action induced by
$\rho$,

-- the restriction of $\rho'(c_i)$ on $I_0$ coincides with $f_0$.

We call such a representation \textbf{a modification of $\rho$ on the gap $I_0$ by $f_0$.}

Furthermore, if $\rho$ and $f_0$ are $C^k$, and $f_0$ coincides with $\rho(c_i)$ near $\partial I_0$,
then the new representation $\rho'$ is also $C^k,$ and for every $\gamma$ in $\Gamma$ and every $r \leq k$, the
$r$-derivatives of $\rho(\gamma)$ and
$\rho'(\gamma)$ coincide on $\mu$.
\end{define}

\subsection{Groups of almost-(k)-convergence}

\begin{define}
A  $(\mu, \tau, \sigma)$-representation $\bar{\rho}: \bar{\Gamma} \to \mbox{Homeo}(\mathbb S^1)$ has the
\textbf{(discrete) (k)-convergence property} if, for every sequence $(\bar{\gamma}_n)_{n \in \mathbb N}$,
up to a subsequence the following dichotomy holds:

-- either the sequence $(\bar{\rho}(\bar{\gamma}_n))_{n \in \mathbb N}$ is stationary,

-- or there exist two $\tau$-orbits
$\{ x_0^- = \tau(x_{k-1}^-), x^-_1=\tau(x_0^-), \ldots , x_{k-1}^-=\tau(x_{k-2}^-)\}$ and
$\{ x_0^+ = \tau(x_{k-1}^+),  x^+_1=\tau(x_0^+), \ldots , x_{k-1}^+ = \tau(x_{k-2}^+)\}$ such that,
for any compact subset $K$ of $]x^-_i, x^-_{i+1}[$ (with $0 \leq i < k$) the restriction of $\bar{\rho}(\bar{\gamma}_n)$ to $K$ converges uniformly to $x^+_i$.
\end{define}

We also say that $(\bar{\gamma}_n)_{n \in \mathbb N}$ or
$(\bar{\rho}(\bar{\gamma}_n))_{n \in \mathbb N}$ satisfies the
$(k)$-convergence property.

Observe that, in particular, if $\bar{\rho}: \bar{\Gamma} \to \mbox{Homeo}(\mathbb S^1)$ has the (k)-convergence property, then the fixed point set of every non-trivial $\bar{\rho}(\gamma)$
is a union of at most $2$ orbits by $\tau$, hence contain at most $2k$ elements.
In the case $k=1$ one recovers the usual notion of convergence group.

Typical examples of discrete (k)-convergence groups are Fuchsian groups. More precisely: for every integer $k\geq1$, let PGL$_k(2, \mathbb R)$ denote the groups of projective
transformations (orientation preserving or not) of the cyclic $k$-cover $\mathbb R\mathbb P_k^1$ over the real projective line. It is also the quotient of the universal covering $\widetilde{\mbox{PGL}}(2, \mathbb R)$
of PGL$(2, \mathbb R)$ by the subgroup of index $k$ of the center of PGL$(2, \mathbb R)$. Then, any discrete subgroup of PGL$_k(2, \mathbb R)$, as group of transformation of $\mathbb R\mathbb P_k^1 \approx \mathbb S^1,$
is a discrete (k)-convergence group.

In particular, the definition of PGL$_k(2,\mathbb R)$ immediately implies that
 there is a natural projection
$\pi_k: {\mbox{PGL}}_k(2,\mathbb R) \rightarrow
{\mbox{PGL}}(2,\mathbb R)$ that is a
$k$-fold covering map and a homomorphism.

\begin{theorem}\label{thm:kconvfuchsian}
Every $(\mu, \tau, \sigma)$-representation $\bar{\rho}: \bar{\Gamma} \to \mbox{Homeo}(\mathbb S^1)$
satisfying the (k)-convergence property is topologically conjugate
to a Fuchsian action, i.e. there exists a homeomorphism $f: \mathbb S^1 \to {\mathbb R\mathbb P}_k^1$ and a representation $\bar{\rho}_0: \bar{\Gamma} \to$ {PGL}$_k(2,\mathbb R)$
such that:
$$f \circ \bar{\rho} = \bar{\rho}_0 \circ f$$
\end{theorem}

\begin{proof}
The case $k=1$ is simply a reformulation of the \textit{convergence group Theorem} proved by Gabai and Casson-Jungreis (\cite{gabaiconv, cassonjungreis}) culminating a series
of works by many others. Actually, the results in \cite{gabaiconv, cassonjungreis} are stated for actions preserving the orientation of $\mathbb S^1$, and here we have to take care of the general case
allowing orientation reversing elements.
The reference \cite{tukia} actually dealt with this general situation, and proved the conjecture except in the case of orientation preserving
actions of a triangular group $< a, b, c \;\mid\; a^p=b^q=c^r=1 >$. The triangular group
case was solved thereafter independently by Gabai and Casson-Jungreis. This most difficult case will not be used
in this paper, since we will apply this Theorem to an orbifold group of an orbifold admitting at least one boundary component, hence not a triangular group.

Now we give an outline of the way to reduce the general case $k>1$ to the case $k=1$ (compare with Lemma 3.6.2 in \cite{monclairthesis}).
The idea is to extend $\tau$ to a homeomorphism $\tau: \mathbb S^1 \to \mathbb S^1$, almost commuting with the action, and to apply the convergence group Theorem to the induced action on
the quotient circle $\mathbb S^1/\tau$.

Let $I$ be a gap of $\mu$. If $I$ is wandering, one can choose arbitrarly any orientation preserving extension of $\tau$ inside $\tau^i(I)$ for $0 \leq i \leq k-2$, and then define the restriction of $\tau$ inside $\tau^{k-1}(I)$
so that the restriction of $\tau^k$ to $I$ is the identity map. Then define the restriction of $\tau$ to every $\bar{\rho}(\bar{\gamma})(\tau^i(I))$ as $\bar{\rho}(\bar{\gamma})\tau_{\mid \tau^i (I)}\bar{\rho}(\bar{\gamma})^{-1}$.
Since $I$ is wandering, there is no relation to obey, and we define in this way an extension of $\tau$ to the entire circle, except on periodic gaps.

We consider now the case where $I$ is periodic. Let $\bar{\Gamma}_I$ be the stabilizer of $I$. Let us denote by $a$, $b$ the two extremities of $I$, so that $I=]a, b[$. If some element $\bar{\gamma}$ of $\bar{\Gamma}_I$
admits a fixed point $x$ inside $I$, then it would admit at least $2k+1$ fixed points: the orbits of $a$, $b$ by $\tau$ and $x$, and it would contradict the (k)-convergence property applied to the sequence
$(\bar{\gamma}^n)_{n \in \mathbb N}$. Therefore, the action of $\bar{\Gamma}_I$ on $I$ is free, and therefore, according to H\"{o}lder's Theorem (\cite{holder}), topologically conjugate to an action
by translations (once $I$ is identified with the real line). If $\bar{\rho}(\bar{\Gamma}_I)$ is not cyclic, then one would once more obtain a contradiction by considering a sequence $(\bar{\gamma}_n)_{n \in \mathbb N}$ made
of distinct elements mapping a point $x$ in $I$ to elements $\bar{\rho}(\bar{\gamma}_n)x$ converging to $x$.

We conclude that $\bar{\rho}(\bar{\Gamma}_I)$ is cyclic, generated by an element $\bar{\rho}(\bar{\gamma}_0)$, and that its action on $I$ preserves the orientation.
Therefore, there is a topological conjugacy between the action of $\bar{\rho}(\bar{\Gamma}_I)$ inside $I$
and its action on $\tau(I)$. More precisely, we select a point $x_i$ in every $\tau^i(I)$, and take any orientation preserving homeomorphism $\tau_i$ between $[x_i, \bar{\rho}(\bar{\gamma}_0)x_i]$ and
$[x_{i+1}, \bar{\rho}(\bar{\gamma}_0)x_{i+1}]$, only adjusting so that the composition $\tau_{k-1} \circ ... \circ \tau_0$ is trivial on $[x_0, \bar{\rho}(\bar{\gamma}_0)x_0]$. We then extend every $\tau_i$ on every $\tau^i(I)$
by $\bar{\rho}(\bar{\gamma}_0)$-equivariance, and then on every $\bar{\rho}(\bar{\gamma})(\tau^i(I))$ by $\bar{\rho}(\bar{\Gamma})$-equivariance.

All these extensions are compatible with one another, and define an extension of $\tau$ to the entire circle,
which almost commutes with the action of $\bar{\Gamma}$.
One considers then the action on the circle ${\mathbb S}^1/\tau$, which has the (1)-convergence property, and therefore is topologically conjugate to a projective action.
The lift of this topological conjugacy is the required topological conjugacy between $\bar{\rho}$ and a Fuchsian representation $\bar{\rho}_0: \bar{\Gamma} \to$ {PGL}$_k(2,\mathbb R)$.
\end{proof}

Of course, when one modifies the representation on a periodic gap as in Proposition \ref{pro:blowupaction}, the new representation does not have anymore the (k)-convergence property, since
one can increase arbitrarly the number of fixed points for a given element. However, we will see that a weak form of (k)-convergence property still holds:

\begin{define}
A $(\mu, \tau, \sigma)$-representation $\bar{\rho}: \bar{\Gamma} \to \mbox{Homeo}(\mathbb S^1)$ has the
\textbf{almost (k)-convergence property} if, for every sequence $(\bar{\gamma}_n)_{n \in \mathbb N}$, up to a subsequence, the following trichotomy holds:

-- either the sequence $(\bar{\rho}(\bar{\gamma}_n))_{n \in \mathbb N}$ is stationary,

-- or there are elements $a$, $\bar{\gamma}$ of $\bar{\Gamma}$ and a sequence $(p_n)_{n \in \mathbb N}$ of integers such that $\bar{\gamma}$ preserves a gap of $\mu$ and:
$$\forall n \in \mathbb N \;\; \bar{\rho}(\bar{\gamma}_n) = \bar{\rho}(\bar{\gamma})^{p_n} \bar{\rho}(a)$$

-- or there exist two $\tau$-orbits
$\{ x_0^-, x^-_1=\tau(x_0^-), ... , x_{k-1}^-=\tau(x_{k-2}^-)\}$ and $\{ x_0^+, x^+_1=\tau(x_0^+), ... , x_{k-1}^+=\tau(x_{k-2}^+)\}$ such that,
for any compact subset $K$ of $]x^-_i, x^-_{i+1}[$ (with $0 \leq i < k$) the restriction of $\bar{\rho}(\bar{\gamma}_n)$ to $K$ converges uniformly to $x^+_i$.
\end{define}

\begin{theorem}\label{thm:almostconvergencestable}
Let $\bar{\rho}: \bar{\Gamma} \to \mbox{Homeo}(\mathbb S^1)$ be a $\mu$-faithful non elementary $(\mu, \tau, \sigma)$-representation of an orbifold group $\bar{\Gamma}$, and
let $\bar{\rho}': \bar{\Gamma} \to \mbox{Homeo}(\mathbb S^1)$ be a modification of $\bar{\rho}$ on a periodic gap.
Then, $\bar{\rho}'$ is a group of almost (k)-convergence if and only if the same is true for $\bar{\rho}$.
\end{theorem}

\begin{proof}
Let $(\bar{\gamma}_n)_{n \in \mathbb N}$ be a sequence in $\bar{\Gamma}$. Up to a subsequence, we can assume
that $(\bar{\rho}(\bar{\gamma}_n))_{n \in \mathbb N}$ is in one the cases prescribed by
almost (k)-convergence:

\textbf{Case 1: } the sequence $(\bar{\rho}(\bar{\gamma}_n))_{n \in \mathbb N}$ is stationary. Then, since $\bar{\rho}$ is $\mu$-faithful, it follows that
the sequence $(\bar{\gamma}_n)_{n \in \mathbb N}$ is stationary, and therefore, that $(\bar{\rho}'(\bar{\gamma}_n))_{n \in \mathbb N}$ is stationary.

\textbf{Case 2: } there are elements $a$, $\bar{\gamma}$ of $\bar{\Gamma}$ and a sequence $(p_n)_{n \in \mathbb N}$ of integers such that $\bar{\gamma}$ preserves a gap of $\mu$ and:
$$\forall n \in \mathbb N \;\; \bar{\rho}(\bar{\gamma}_n) = \bar{\rho}(\bar{\gamma})^{p_n} \bar{\rho}(a)$$
Then, as in Case $1$, since $\bar{\rho}$ is $\mu$-faithful, it follows
that $\bar{\gamma}_n =
\bar{\gamma}^{p_n} a$. Applying $\bar{\rho}'$ shows the same property holds for $\bar{\rho}'$.

\textbf{Case 3:} \  there exist two $\tau$-orbits
$\{ x_0^-, x^-_1=\tau(x_0^-), ... , x_{k-1}^-=\tau(x_{k-2}^-)\}$ and $\{ x_0^+, x^+_1=\tau(x_0^+), ... , x_{k-1}^+=\tau(x_{k-2}^+)\}$ such that,
for any compact subset $K$ of $]x^-_i, x^-_{i+1}[$ (with $0 \leq i < k$) the restriction of $\bar{\rho}(\bar{\gamma}_n)$ to $K$ converges uniformly to $x^+_i$.

In particular this implies that $x^+_i$ is in $\mu$.
In addition, considering the sequence $(\bar{\rho}(\bar{\gamma}_n)i)^{-1}$,
one sees that $x^-_i, x^-_{i+1}$ are also in $\mu$.
In this case, we will show that either the sequence $(\bar{\rho}'(\bar{\gamma}_n))_{n \in \mathbb N}$ satisfies the same property, or satisfies the property described in case $2$.

Assume that for some compact arc $K=[a, b]  \subset ]x^-_i, x^-_{i+1}[$, the iterates $\bar{\rho}'(\bar{\gamma}_n)K$ do not shrink to the point $x^+_i$. We can assume wlog that $i=0$.
Let $\alpha$ be the unique element of $\mu \cap [x_0^-, a]$ such that $]\alpha, a]$ is disjoint from $\mu$ (if $a$ lies in $\mu$, then we have $\alpha=a$).
Consider similarly the unique element $\beta$ of $\mu \cap [b, x_1^-]$ for which
$\mu \cap [b, \beta[=\emptyset$. If $\alpha \neq x_0^-$ and $\beta \neq x_1^-$, then the interval $[\alpha, \beta]$ shrinks under the action of $\bar{\rho}(\bar{\gamma}_n)$ to the point
$x^+_0$. But since $\alpha$ and $\beta$ are in $\mu$, and since $\bar{\rho}'$ differs from $\bar{\rho}$ only by its action inside gaps, the same property holds for the sequence $\bar{\rho}'(\bar{\gamma}_n)[\alpha, \beta]$.
It follows that $K \subset [\alpha, \beta]$ also shrinks to $x^+_0$ under the action of $\bar{\rho}'(\bar{\gamma}_n)$, contradiction.

Hence, we must have $\alpha = x^-_0$ (or $\beta=x^-_1$, but the treatment of this case is similar, and will not be considered here). In other words, $a$ lies in
a gap $]x_0^-, \alpha'[$. If we had also $\beta=x^-_1$, we would also conclude that there is a gap $]\beta', x_1^-[$, but then $x_1^-$ would be at the boundary of two different gaps:
$]\beta', x_1^-[$ and also $\tau(]x_0^-, \alpha'[)=]x_1^-, \tau(\alpha')[$. It is impossible since $\mu$ is perfect, therefore we have $\beta < x^-_1$. We conclude that the segment
$[\alpha', \beta]$ shrinks to $x^+_0$ under the action of $\bar{\rho}(\bar{\gamma}_n)$, and under the action of $\bar{\rho}'(\bar{\gamma}_n)$ as well since $\alpha'$ and $\beta$
both lie in $\mu$.

It follows that the iterates $\bar{\rho}'(\bar{\gamma}_n)[a, \alpha']$ do not shrink to a point, but to a segment $[x^+, x_0^+]$ with $x^+ < x_0^+$ (up to a subsequence).
Hence the iterates under $\bar{\rho}'(\bar{\gamma}_n)$ of the gap $]x_0^-, \alpha'[$ also do not converge to a point, but to a non-trivial segment. This limit segment $I_\infty$ must be a gap;
and since, for every $\epsilon$, there is only a finite number of gaps of length $\geq \epsilon$, it follows that, up to a subsequence, the gaps
$\bar{\rho}'(\bar{\gamma}_n)(]x_0^-, \alpha'[)=\sigma(\bar{\gamma}_n)(]x_0^-, \alpha'[)=\bar{\rho}(\bar{\gamma}_n)(]x_0^-, \alpha'[)$ are all equal to $I_\infty$.
Hence for every $n$, $\bar{\gamma}_n \bar{\gamma}_1^{-1}$ is in the stabilizer of $I_\infty$. Since we are in case $3$ for $\bar{\rho}(\bar{\gamma}_n)$, we know that the iterates $\bar{\rho}(\bar{\gamma}_n)[a, \alpha']$
shrink to $x^+_0$, hence the segment
$[a,\alpha']$ lies in the domain where we have modified the action. It follows that $I_\infty$ is in the $\sigma(\bar{\Gamma})$-orbit of the periodic gap $I_0$ where  the
action has been modified, and therefore, according to the hypothesis of Proposition \ref{pro:blowupaction}, the stabilizer of $I_\infty$ is cyclic, generated by some element $\bar{\gamma}$.

In summary, and denoting $\bar{\gamma}_1$ by $a$, we have proved that every $\bar{\gamma}_n$ is of the form $\bar{\gamma}^{p_n}a$. We are in case $2$, and the Theorem is proved.
\end{proof}

\begin{proposition}\label{prop:hypfixed}
Let $\bar{\rho}: \bar{\Gamma} \to \mbox{Homeo}(\mathbb S^1)$ be a $\mu$-faithful non elementary $(\mu, \tau, \sigma)$-representation satisfying the almost (k)-convergence property. Then, $\bar{\rho}$ has the
(k)-convergence property if and only if for every periodic gap $I$ we have:
\begin{itemize}
\item  the action of the stabilizer $\bar{\Gamma}_I$ of $I$ on $I$ is free,
\item for every non-trivial element $\bar{\gamma}$ of $\bar{\Gamma}_I$ the points in $\mathbb S^1$ fixed by $\bar{\gamma}$ are exactly
the iterates under $\tau$ of the extremities $\partial I$, and they are all hyperbolic fixed points.
\end{itemize}
\end{proposition}

\begin{proof}
One implication is clear: if $\bar{\rho}$ has the (k)-convergence property, then it  is topologically conjugate to a Fuchsian action, and the two conditions are necessarily satisfied.

Let now $\bar{\rho}: \bar{\Gamma} \to \mbox{Homeo}(\mathbb S^1)$ be a $\mu$-faithful non elementary $(\mu, \tau, \sigma)$-representation satisfying the almost (k)-convergence property and the two conditions stated in the proposition.
Notice that the fixed points of non-trivial $\bar{\rho}(\bar{\gamma})$ are in
$\mu$.
By the first hypothesis,
the action $\bar{\rho}(\bar{\Gamma})$ on $\mathbb S^1  -  \mu$ is free. Let $(\bar{\gamma}_n)_{n \in \mathbb N}$ be a sequence in $\bar{\Gamma}$. Up to a subsequence, $(\bar{\rho}(\bar{\gamma}_n))_{n \in \mathbb N}$ is in one the three cases imposed by
almost (k)-convergence. It satisfies the condition obeyed by sequences under the (k)-convergence property, except maybe if
we are in the case where $\bar{\rho}(\bar{\gamma}_n) = \bar{\rho}(\bar{\gamma})^{p_n} \bar{\rho}(a)$ for some sequence $(p_n)_{n \in \mathbb N}$ of integers and for two elements $a$, $\bar{\gamma}$ of $\bar{\Gamma}$,
where $\bar{\gamma}$ preserves a gap $I$ of $\mu$. Then since $\bar{\gamma} \in \bar{\Gamma}_I$, by hypothesis, it admits exactly
$2k$ fixed points: the points in the $\tau$-orbit of $\partial I$; and all these fixed points
are hyperbolic. It follows that $\bar{\rho}(\bar{\gamma})$ is topologically conjugate to a projective transformation, hence that $(\bar{\rho}(\bar{\gamma})^{p_n})_{n \in \mathbb N}$ satisfies the (k)-convergence property, and therefore the same is true for $(\bar{\gamma}_n)_{n \in \mathbb N}$.
\end{proof}

\begin{corollary}\label{cor:cor}
Any modification of a (k)-convergence group on a periodic gap $I$ by a homeomorphism $f: I \to I$ without
fixed points and such that $f$ is hyperbolic near $\partial I$ is a (k)-convergence group.  \hfill$\Box$
\end{corollary}



\subsection{A dynamical characterization of (k)-convergence groups}


\begin{theorem}\label{thm:caracteriseconvergencegroup}
Let $\bar{\rho}: \bar{\Gamma} \to \mbox{Homeo}(\mathbb S^1)$ be a $(\mu, {\tau}, \sigma)$-representation. Let $\bar{\Gamma}_0$ be the index $2$ subgroup
made of elements preserving the orientation. Assume that
$\bar{\rho}$ satisfies the following properties:
\begin{enumerate}
  \item every gap of $\mu$ is periodic,
  \item for every $x$ in $\mathbb S^1$, the stabilizer of $x$ is trivial or cyclic,
  \item for every non-trivial element $\bar{\gamma}$ of $\bar{\Gamma}_0$ the fixed point set of $\bar{\rho}(\bar{\gamma})$ is either trivial, or one orbit of ${\tau}$, or the union of two orbits by $\tau$, one made of attractive fixed points and the other made of repellent fixed points,
  \item  if $(x_0, y_0)$ is a pair of fixed points of some element of $\bar{\Gamma}$ with $x_0 < y_0 < \tau(x_0)$,  then the $\bar{\rho}(\bar{\Gamma}_0)$-orbit by the diagonal action of $(x_0, y_0)$ in the space $U = \{ (x,y) \in \mu \times \mu \;|\; x < y < \tau(x)\}$ is closed and discrete.
\end{enumerate}
Then, $\bar{\rho}$ has the (k)-convergence property, and thus is topologically conjugate to a Fuchsian representation.
\end{theorem}

\begin{remark}
{\em Item $(4)$ is coherent: for every non-trivial
element $\bar{\gamma}_0$, the fixed points of $\bar{\rho}(\bar{\gamma}_0)$ all lie in $\mu$. Indeed, if not,
$\bar{\rho}(\bar{\gamma}_0)$ would admit a fixed point $x$ in a gap $]a, b[$. Then, $\bar{\rho}(\bar{\gamma}_0^2)$
would admit at least $1+2k$ fixed points: $x$, and the orbits of $a$ and $b$ by $\tau$. It contradicts item $(3)$.

Therefore, if $(x_0, y_0)$ is a pair of fixed points of $\bar{\gamma}_0$ with $x_0 < y_0 < \tau(x_0)$,
then $(x_0, y_0)$ lies indeed in $U = \{ (x,y) \in \mu \times \mu \;|\; x < y < \tau(x) \}$.}
Recall the definition of $x < y < \tau(x)$ in the beginning
of this section.
\end{remark}

\noindent\textit{Sketch of proof. } We skip the elementary case where $\bar{\Gamma}$ is a cyclic group. Let $\sim$ be the equivalence relation identifying all points in $\bar{I}$
for every gap $I$ of $\mu$. Then, the quotient space $\mathbb S = {\mathbb S^1}/\sim$ is homeomorphic to the circle,
and $\bar{\rho}$ induces an action of $\bar{\Gamma}$ on $\mathbb S$ by homeomorphisms.
After the collapsing the action is minimal.
 Moreover, $\tau$ induces
a homeomorphism $\breve{{\tau}}$ on $\mathbb S$ of order $k$, and the quotient $\breve{\mathbb S} = \mathbb S/\breve{\tau}$ is a circle too, on
which $\bar{\Gamma}$ acts naturally.  We focus on the induced action of $\bar{\Gamma}_0$. This action is minimal, and satisfies the same properties with $k=1$: \ 1) there is no gap, hence no wandering gap;
\ 2) the stabilizer of any point is trivial or cyclic;
\ 3) every element admits at most $2$ fixed points, and if it admits $2$, it is of hyperbolic type.
Finally, if $(\breve{x}_0, \breve{y}_0)$ with $\breve{x}_0 \neq \breve{y}_0$ is fixed by some element, then its $\bar{\Gamma}_0$-orbit is closed and discrete in $\breve{\mathbb S} \times \breve{\mathbb S}  -  \Delta$, where $\Delta$ is the diagonal (for details, see the proof of Theorem $2.6$ in \cite{Ba3}).

Then, Theorem $2.6$ in \cite{Ba3} implies that this quotient action has the convergence property,
hence is Fuchsian. 
Since $\bar{\Gamma}_0$ is finitely generated and the limit set is the entire circle, the quotient hyperbolic surface $\bar{\Gamma}_0\backslash\mathbb H^2$ has finite volume.
Let now $\bar{\gamma}_0$ be any element of $\bar{\Gamma}  -  \bar{\Gamma}_0$. Then $\bar{\gamma}_0$ induces a involution on $\bar{\Gamma}_0\backslash\mathbb H^2$. It is a easy case for the Nielsen realization problem
\cite{Ke}: there exist a hyperbolic metric on
the surface so that the involution is an isometry. It follows that the whole action of $\bar{\Gamma}$
on $\breve{\mathbb S}$ is Fuchsian.

The initial action $\bar{\rho}$ is obtained by opening some cusps in the associated hyperbolic orbifold, replacing them by funnels, and taking a finite covering. The Theorem follows (cf. in particular Corollary \ref{cor:cor}). \hfill$\Box$

\begin{remark}{\em
Some of the results in this section (in particular
Theorem \ref{thm:kconvfuchsian}, Proposition \ref{prop:hypfixed}
and Theorem \ref{thm:caracteriseconvergencegroup}) are related
to the results in Mann's article \cite{Mann}, and possibly
could be implied by Mann's results or techniques if interpreted
correctly. Mann's powerful techniques involve carefully
analysing rotation numbers of homeomorphisms of the circle
in the corresponding group representations.
}
\end{remark}

\section{Blowing up pieces of geodesic flows}
\label{sec:blowing}

In this section, we usually let $\Gamma$ be the fundamental group of a Seifert fibered space $P$
(or $P_0$)  with boundary,
and $\bar{\Gamma}$ is the orbifold fundamental group of the base space
$B$ of $P$.
In addition
$\rho: \Gamma \to$ Homeo$(\mathbb R)$ is a representation, such that the image $\rho(h)$ of
the element corresponding to regular fibers is the translation by $+1$. It will always be a lift
of a $(\mu, \tau, \sigma)$-representation $\bar{\rho}: \bar{\Gamma} \to$ Homeo$(\mathbb S^1)$,
satisfying the almost (k)-convergence property.
Since $\bar{\Gamma}$ is the orbifold fundamental group of the base space $B$ of $P$, then
$\bar{\Gamma} = \Gamma / <h>$.
We denote by $\tilde{\mu}$ the lift
of $\mu$ (it is therefore the minimal invariant closed subset for $\rho(\Gamma)$) and by
$\tilde{\tau}: \tilde{\mu} \to \tilde{\mu}$ the lift of $\tau$: it is a homeomorphism, almost commuting
with the restriction of $\rho(\Gamma)$,
and satisfying: $\tilde{\tau}^k = \rho(h) |_{\tilde{\mu}}$.

We will denote by $\tilde{\tau}_0$ any homeomorphism from $\mathbb R$ onto $\mathbb R$, coinciding
with $\tilde{\tau}$
on $\tilde{\mu}$ and such that $\tilde{\tau}_0^k=\rho(h)$ (but not necessarily almost commuting with $\rho(\Gamma)$ outside
$\tilde{\mu}$).
Up to a conjugation in Homeo($\mathbb S^1$) one may assume that
$\tilde{\tau}_0$ is the translation by $+1/k$.

We will modify the representation along periodic gaps (we have already observed in section \ref{sub:orbifoldcase}
that all the modifications of $\bar{\rho}$ lift to representations of the same group $\Gamma$). The goal
is to construct two foliations on the Seifert fibered space $P$, transverse
to each other, that can be considered
as blow ups of the stable and unstable foliations of the geodesic flow for some hyperbolic metric on
the base orbifold $B$.

For our blow up construction it will be helpful to present the geodesic flow of
hyperbolic surfaces in
the following manner.

\vskip .1in
\subsection{An alternative construction of geodesic flows}\label{sub:geodflow}
\label{alter}

We start with a convex cocompact subgroup $\bar{\Gamma}$ of $\mbox{PGL}_k(2, \mathbb R)$.
Since we will have to consider
the non-orientable case, we carefully define this notion which involves
some subtleties.
First: The group ${\mbox{PGL}}(2, \mathbb R)$ is the isometry group of the hyperbolic plane.
This is not the usual approach, so
let us point out for the reader's convenience that the M\"obius transformation
of the semi-plane model of the hyperbolic plane defined by an element $\left(
    \begin{array}{cc}
      a & b \\
      c & d \\
    \end{array}
  \right)$ with negative determinant is $z \mapsto \frac{a\bar{z}+b}{c\bar{z}+d}$.
It is a symmetry through a geodesic and it reverses orientation
in $\mathbb H^2$.

\begin{define}\label{defi:convexcocompact}
A subgroup $\bar{\Gamma}$ of $\mbox{PGL}_k(2, \mathbb R)$ is \textit{convex cocompact}
if its projection $\check{\Gamma}$
in $\mbox{PGL}(2, \mathbb R)$ is discrete, finitely generated,
admits no parabolic element, and such that elements of $\check{\Gamma}$
reversing the orientation have infinite order.
\end{define}

\vskip .1in

\begin{remark}{\em
The last condition may look unnecessary, even a bit unusual.
The reason for this condition is that we want the action of $\check{\Gamma}_0$
(the subgroup of index two of orientation preserving elements)
on T$^1\mathbb H^2$ to be free, so that the unit tangent bundle $\bar{\Gamma}\backslash T^1\mathbb H^2$ is a manifold.
An elementary example of a discrete subgroup of $\mbox{PGL}(2, \mathbb R)$ satisfying all the hypotheses of
the convex cocompact definition except the last one is the quotient
of a hyperbolic genus $2$ surface $\Sigma$ by an involution on a ``middle'' simple closed geodesic of $\Sigma$.
}
\end{remark}

Let $\Gamma$ be the preimage of $\bar{\Gamma}$ in $\widetilde{\mbox{PGL}}(2, \mathbb R)$. Then, the inclusion
$\rho_0: \Gamma \hookrightarrow \widetilde{\mbox{PGL}}(2, \mathbb R) \subset$ Homeo$(\mathbb R)$ is a
$(\tilde{\mu}, \tilde{\tau}_0, \tilde{\sigma})$-representation. We will also consider the preimage $\ddot{\Gamma}$ in $\widetilde{\mbox{PGL}}(2, \mathbb R)$ of $\check{\Gamma}:$
$\Gamma$ is a finite index subgroup of $\ddot{\Gamma}.$

Here, $\tilde{\tau}_0$ is a generator of the pseudo center of $\widetilde{\mbox{PGL}}(2, \mathbb R)$,
in the sense of almost commutation defined in section 2.
Therefore $\tilde{\tau}$ is defined on the
entire universal covering $\widetilde{\mathbb R\mathbb P}^1$.

\vskip .1in
There are natural projections
$\widetilde{\mbox{PGL}}(2,\mathbb R) \rightarrow
{\mbox{PGL}}_k(2,\mathbb R) \rightarrow {\mbox{PGL}}(2,\mathbb R)$.
The first is infinite to one and the second is finite to one.
We denote by $\tilde{\pi}_k$ the first projection and by $\tilde{\pi}$ the composition
of the two projections.
For convenience of the reader in future referencing we recall

$$\bar{\Gamma} \ \leq \ {\mbox{PGL}}_k(2,\mathbb R), \ \ \ \ \ \
\check{\Gamma} = \pi_k(\bar{\Gamma}) \ \leq \ {\mbox{PGL}}(2,\mathbb R),$$

$$\Gamma  = \tilde{\pi}^{-1}(\bar{\Gamma}) \ \leq \ \widetilde{\mbox{PGL}}(2,\mathbb R), \ \ \ \ \ \
\ddot{\Gamma} = \tilde{\pi}_k^{-1}(\check{\Gamma}) \ \leq \ \widetilde{\mbox{PGL}}(2,\mathbb R).$$

\noindent
Here $\leq$ means being a subgroup.

\vskip .1in
The data of an oriented geodesic in $\mathbb H^2$ is equivalent to the data of a pair of distinct points $(x,y)$ in
$\partial\mathbb H^2$. Hence the space of oriented geodesics is ${\mathbb R\mathbb P}^1 \times \mathbb R\mathbb P^1  -  \Delta$, where $\Delta$
is the diagonal. The action of ${\mbox{PGL}}(2, \mathbb R)$ on ${\mathbb R\mathbb P}^1 \times \mathbb R\mathbb P^1  -  \Delta$ corresponding
to the action on geodesics is simply the diagonal action.

Let $M_0$ be the unit tangent bundle of
the hyperbolic orbifold $_{\check{\Gamma}\backslash}\mathbb H^2$.
Denote by $\Psi_0^t$ the geodesic flow on $M_0$.
Geodesics of $\mathbb H^2$ are simply projections of
orbits of the lift of $\Psi_0^t$ to $T^1\mathbb H^2$.
Let $\Psi_0$ be the one-dimensional foliation of $M_0$ induced by
$\Psi_0^t$.

The orbit space of the lift $\tilde{\Psi}_0$ of the geodesic flow
in the universal covering $\widetilde{M}_0$ of $M_0$
is the universal covering of ${\mathbb R\mathbb P}^1 \times \mathbb R\mathbb P^1  -  \Delta$.
This is identified with
the open domain $\Omega_0$ in $\widetilde{\mathbb R\mathbb P}^1 \times \widetilde{\mathbb R\mathbb P}^1$ between the graphs of the identity map id and of the map $\tilde{\tau}_0$,
that is,

$$\Omega_0 \ = \ \{ (x,y) \in
\widetilde{\mathbb R\mathbb P}^1 \times
\widetilde{\mathbb R\mathbb P}^1, \ \ x < y < \tilde{\tau}_0(x) \}$$

\noindent
The delicate
point is to understand how the action of ${\mbox{PGL}}(2, \mathbb R)$ on ${\mathbb R\mathbb P}^1 \times \mathbb R\mathbb P^1  -  \Delta$ lifts to an action of $\widetilde{\mbox{PGL}}(2, \mathbb R)$ on $\Omega_0$: it is \underline{not} the
restriction of the diagonal action since this action does not preserve $\Omega_0$: the diagonal action of elements of
$\widetilde{\mbox{PGL}}(2, \mathbb R)  -  \widetilde{\mbox{PSL}}(2, \mathbb R)$ permutes $\Omega_0$
and the domain $\{ (x, y) \in \widetilde{\mathbb R\mathbb P}^1 \times \widetilde{\mathbb R\mathbb P}^1 \;|\;
\tilde{\tau}_0^{-1}(x) < y < x\}$. The lifted action is actually the following:

-- if $\gamma$ is an element of $\widetilde{\mbox{PSL}}(2, \mathbb R)$, define $\gamma.(x,y) = (\gamma x, \gamma y)$;

-- if $\gamma \in \widetilde{\mbox{PGL}}(2, \mathbb R)$ reverses the orientation, define $\gamma.(x,y) = (\gamma x, \tilde{\tau}_0(\gamma y))$.

One easily checks that this is an action, and a lift of the diagonal action on ${\mathbb R\mathbb P}^1 \times \mathbb R\mathbb P^1  -  \Delta$. Moreover this action preserves the domain $\Omega_0$.

\vskip .1in
A key fact for us is that one can reconstruct the geodesic flow
$\Psi^t_0$ on $M_0$ from the data of the action of $\ddot{\Gamma}$ on $\Omega_0$:
Let PT$\Omega_0$ be the projectivized tangent bundle of $\Omega_0$.
The action of
$\widetilde{\mbox{PGL}}(2, \mathbb R)$ on $\mathbb R$ is differentiable and hence so is the
action on $\Omega_0$. Therefore
$\widetilde{\mbox{PGL}}(2, \mathbb R)$ acts on PT$\Omega_0$.
Let $\partial_x$ and $\partial_y$ be the horizontal and vertical vector fields on $\Omega_0$. The action of
$\widetilde{\mbox{PGL}}(2, \mathbb R)$ preserves the lines defined by these vector fields, hence restricts
to a natural action on PT$\Omega_0$ with the vertical and horizontal directions removed.
We focus on this last action. It admits
two orbits: one orbit is the open domain

$$\widetilde{M}(\Omega_0) \ \ = \ \  \{ (x, y, \xi), \ \
\xi = a\partial_x + b\partial_y, \ {\rm with} \ ab > 0 \}$$

\noindent
The other orbit is the collection of $(x,y,\xi)$ for
which $ab$ is negative.
The action preserves each of these because the action on $\Omega_0$ preserves orientation.
Below we show that these are actually orbits, that is, the respective actions are transitive.
Observe that $\widetilde{M}(\Omega_0)$ can also be naturally parametrized by $(x,y, m)$, $(x,y) \in \Omega_0$, and
where $m$ is the real number $\log(b/a)$, where as above the point is given by
$(x,y, a\partial_x + b\partial_y)$.

The action of $\widetilde{\mbox{PGL}}(2, \mathbb R)$ on $\widetilde{M}(\Omega_0)$ is not free:
as an isometry of $\mathbb H^2$, a reflection
$R$ along a geodesic with extremities (the projections in $\mathbb S^1$ of) $x$ and $y$ fixes every
element of the form $(x, y, \xi)$:
For simplicity
we can think of $R$ acting on $\mathbb R$ as $z \rightarrow - z$ and $x = y = 0$.
Hence when acting on T$\Omega_0$,
$R$ will preserve the
lines in the tangent bundle T$\Omega_0$ at the point $(x,y)$.
However the stabilizer of this action at any point of $\widetilde{M}(\Omega_0)$ over $(x,y)$ is precisely
the group of order $2$ generated by $R$. Therefore, $\widetilde{M}(\Omega_0)$ is naturally identified with
the right quotient $\widetilde{\mbox{PGL}}(2, \mathbb R)/R$, where the action of $\widetilde{\mbox{PGL}}(2, \mathbb R)$
on $\widetilde{M}(\Omega_0)$ corresponds to the (left) action of $\widetilde{\mbox{PGL}}(2, \mathbb R)$ on $\widetilde{\mbox{PGL}}(2, \mathbb R)/R$.

On the other hand, the action of $\widetilde{\mbox{PGL}}(2, \mathbb R)$ on the universal covering $\widetilde{T^{1}\mathbb H}^2$
is transitive, and the stabilizer of any unit vector tangent to the geodesic $(x,y)$ is the group of order $2$ generated by $R$:
it follows that the quotient $\widetilde{\mbox{PGL}}(2, \mathbb R)/R$ can also be identified with $\widetilde{T^{1}\mathbb H}^2$.

Therefore, $\widetilde{M}(\Omega_0)$ is naturally identified with $\widetilde{T^{1}\mathbb H}^2$; orbits of the geodesic flow
in $\widetilde{T^{1}\mathbb H}^2$ correspond to the fibers of the projection of $\widetilde{M}(\Omega_0) \subset$ PT$\Omega_0$ over $\Omega_0$, and
this identification is equivariant with respect to the actions of $\widetilde{\mbox{PGL}}(2, \mathbb R)$. In particular,
the induced action of $\ddot{\Gamma}$ on $\widetilde{M}(\Omega_0)$ is properly discontinuous. Furthermore,
by our requirement in the definition of convex cocompact subgroups
(see Definition \ref{defi:convexcocompact}), this action is free.

\vskip .1in
The quotient
$M_{\ddot{\Gamma}}(\Omega_0) =  {\ddot{\Gamma}\backslash}\widetilde{M}(\Omega_0)$
is a $3$-manifold equipped with an one-dimensional foliation
(by an abuse of notation still denoted by)
$\Psi_0$. This foliation
is the projection by $\ddot{\Gamma}$ of  the restriction to $\widetilde{M}(\Omega_0)$
of the foliation of PT$\Omega_0$ induced by the fibration over $\Omega_0.$
Furthermore, this foliation is oriented: the orientation
in the leaf defined by $(x, y) \in \Omega_0$ is the one for which the logarithmic slope $m$ is increasing.

In summary, the oriented foliations $(M_0,\Psi_0)$
and $(M_{\ddot{\Gamma}}(\Omega_0), \Psi_0)$
have the same associated orbit space $\Omega_0$ in their universal covers, with the same group action (on
$\Omega_0$) of the fundamental group $\ddot{\Gamma}$. The identifications between $\widetilde{M}(\Omega_0)$, $\widetilde{T^{1}\mathbb H}^2$ and
$\widetilde{\mbox{PGL}}(2, \mathbb R)/R$ produce a topological conjugacy
between $(M_0, \Psi_0)$ and $(M_{\ddot{\Gamma}}(\Omega_0), \Psi_0)$.

Moreover, $M_{\ddot{\Gamma}}(\Omega_0)$ has two invariant foliations: one
provided by the vertical foliation $dx = 0$ of $\Omega_0$, and denoted by
$\Lambda^s(\Psi_0)$; the other by the horizontal foliation $dy=0$ and denoted
by $\Lambda^u(\Psi_0)$. It is easy to check that, as suggested by the
notations, the first one is the weak stable foliation of the geodesic flow
$(M_0, \Psi_0)$,
the other is the unstable foliation of this flow.

In a more general way, finite coverings of the geodesic flow of hyperbolic surfaces
are all obtained as quotients
$M_\Gamma(\Omega_0) := {\Gamma\backslash}\widetilde{M}(\Omega_0)$ where $\Gamma$ is any convex cocompact subgroup
of $\widetilde{\mbox{PGL}}(2, \mathbb R)$ (we mean not necessarily the entire preimage $\ddot{\Gamma}$ in
$\widetilde{\mbox{PGL}}(2, \mathbb R)$ of its projection in PGL$(2, \mathbb R)$).

\begin{remark}{}
{\em There is another well known model of $T^1 \mathbb H^2$ as the triples $(x,y,z)$
with the counterclockwise order. The identification is via the {\underline {orientation}}
preserving isometries of $\mathbb H^2$. In our situation we will have to consider
isometries that are orientation reversing $-$ for example that would come from
an orientation reversing  geodesic in a non orientable hyperbolic surface.
For this reason we were not able to use this model in our work and decided to
use the projectivized model.}
\end{remark}

\subsection{Cutting a compact Seifert piece in the model geodesic flow}\label{sub:cutseifertgeod}
In the previous section, we have shown how to identify finite coverings $(M_0,
\Psi_0)$ of the geodesic flow of a convex cocompact orbifold
$O := {\check{\Gamma}\backslash}\mathbb H^2$
with one of our models $M_\Gamma(\Omega_0)$. The orbifold $O$ has a compact convex core $K$ delimited in $O$ by a finite number of closed simple geodesics $c_1$, ... , $c_k$. More precisely,
$O$ is the union of $K$ and a finite number of flaring
annuli, one for each $c_s$. Observe that singularities of the orbifold $O$ are all
contained in $K$, their preimage in $M_\Gamma(\Omega_0)$ are precisely the singular fibers of the Seifert fibered
structure.
For each $c_s$, unit tangent vectors based at points of $c_s$ form an embedded torus, which lifts to
an embedded torus $T_s$ in
$(M_0, \Psi_0) \approx (M_\Gamma(\Omega_0), \Psi_0)$.
These tori are quasi-tranverse: outside a finite number of periodic orbits,
they are transverse to the flow $\Psi_0$
(or to $\Psi_0$ in $M_\Gamma(\Omega_0)$).
More precisely, the region in the torus between two successive periodic orbits is an elementary Birkhoff annulus.

\vskip .08in
\noindent
{\bf {The submanifold $P_0$:}} \
The union of the tori above is the boundary of a compact submanifold $P_0$
of $M_0$ (equivalently $M_{\Gamma}(\Omega_0)$)  that is a finite covering of the unit tangent bundle of the convex core
$K$.
In particular $P_0$ is a Seifert fibered space.

\begin{define}
The manifold $P_0$ with boundary, equipped with the restriction of the
oriented foliation $\Psi_0$ and the restrictions $\hat{\Lambda}^s_0$, $\hat{\Lambda}^u_0$ of the stable, unstable foliations,
is called a \textbf{piece of geodesic flow.}
\end{define}

\begin{figure}
 \includegraphics[scale=1]{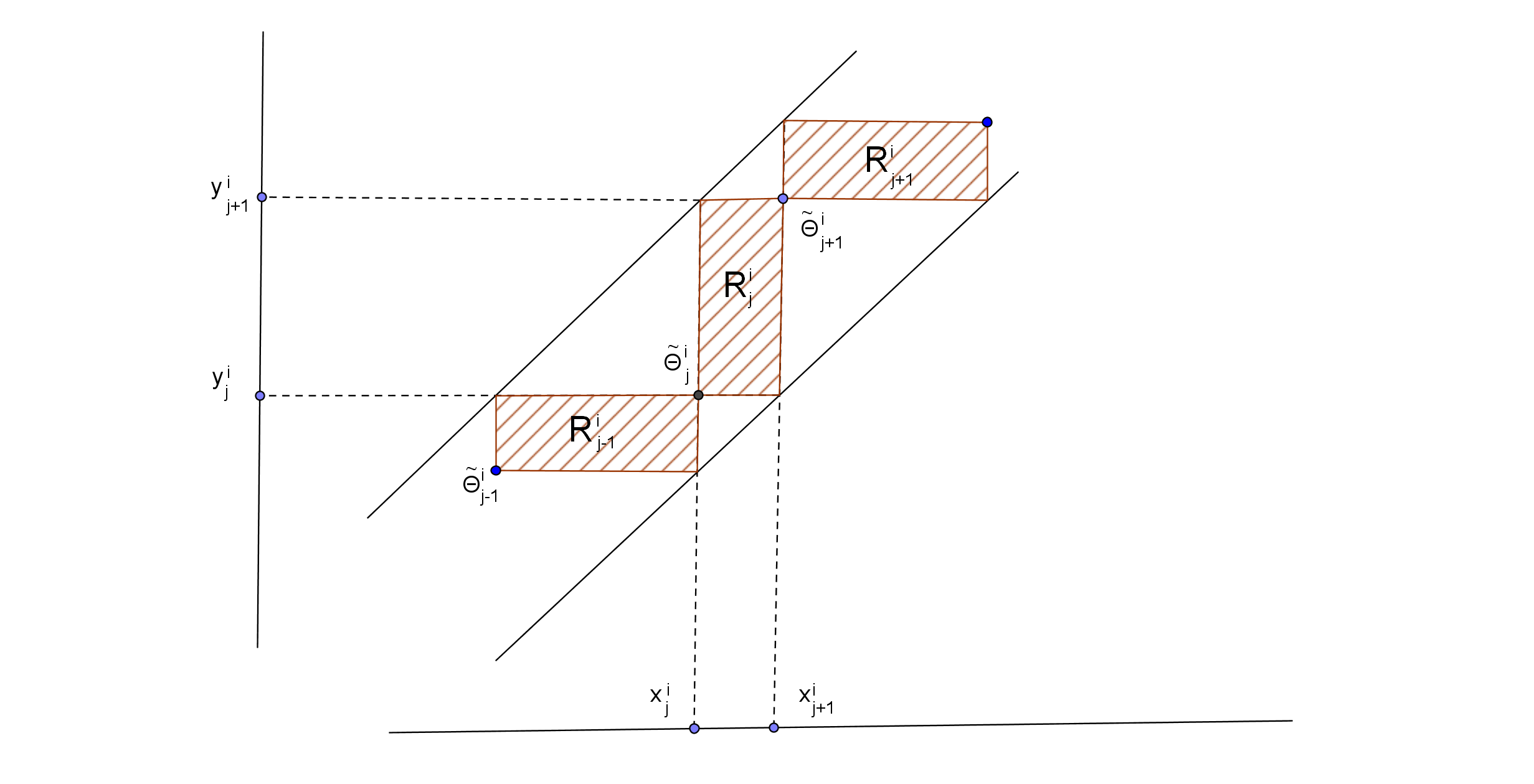}
   \caption{Chain of lozenges for the geodesic flow.}\label{fig:chainlozenge}
\end{figure}

Every torus $T_s$ as above lifts in the universal covering $\widetilde{M}_0$ to infinitely many
properly embedded planes. We will denote the collection of all such lifts for the union of the
$T_s$ as $\{ \widetilde T_i \}$.
Each $\widetilde{T}_i$ is invariant by a maximal free abelian subgroup $H_i$ of rank $2$ of $\Gamma$.
The union of these planes bounds a unique
 region $\widetilde{P}_0$ in $\widetilde{M}(\Omega_0)$. For each $\widetilde{T}_i$, the complement $\widetilde{M}(\Omega_0)  -  \widetilde{T}_i$ has two connected components. One of these connected components contains the interior of $\widetilde{P}_0$ and all the others lifts of quasi-transverse tori: we denote it by $\widetilde{T}_i^+$ . We denote by $\widetilde{T}_i^-$ the other connected component
of $\widetilde M(\Omega_0)  -  \widetilde T_i$.
Observe that the interior of $\widetilde{P}_0$ is the intersection of all $\widetilde{T}^+_i$.

One can describe precisely what is the projection in the orbit space $\Omega_0$ of each of these regions (cf. \cite[Sect. $3.1$]{Ba3}): the projection of $\widetilde{T}_i$ is a $H_i$-invariant string of lozenges $(R^i_j)_{j \in \mathbb Z}$
and their corners, where each $R^i_j$ has two corners $\theta^i_{j}$, $\theta_{j+1}^i$.
For each $i$ the $\widetilde T_i$ produces $\mathbb Z$ many corners $\theta^i_j$.
These corners  make up all orbits tangent to $\widetilde{T}_i$, each preserved by a cyclic subgroup of $H_i$. See figure \ref{fig:chainlozenge}.

More precisely, if $x^i_j$, $y^i_j$ are the coordinates of $\theta^i_j$ in $\Omega_0 \subset \widetilde{\mathbb R\mathbb P}^1 \times \widetilde{\mathbb R\mathbb P}^1$
we have:
$$R^i_j := \{ (x,y) \in \widetilde{\mathbb R\mathbb P}^1 \times \widetilde{\mathbb R\mathbb P}^1 \; \mid \; x^i_j < x < x^i_{j+1}, \; y_j^i < y < y^i_{j+1} \}$$
so $R^i_j$ is a lozenge.
In addition
each lozenge $R^i_j$ is the projection of the lift of an elementary Birkhoff annulus $A^i_j$ in some
torus $T \subset \partial P_0$. The interior of the Birkhoff annulus is
transverse to the flow. There are two cases: $A^i_j$ is either an entrance region into $P_0$
or an exit region for $\Psi_0$. Equivalently the vector field generating
$\widetilde{\Psi}_0$ might be pointing in the direction of $\widetilde{T}_i^+$ or in the direction of $\widetilde{T}_i^-$.
On the other hand, one and only one of the two sides $]x^i_j,  x^i_{j+1}[$ and $]y^i_j,  y^i_{j+1}[$ of
the rectangle $R^i_j$ is a gap of the minimal set $\tilde{\mu}$. The following characterization
will be crucial:

\begin{itemize}
\item
If the gap of $\tilde{\mu}$  is the horizontal side
$]x^i_j,  x^i_{j+1}[$, then the corresponding Birkhoff annulus $A^i_j$  is an exit annulus.

\item
If the gap is  $]y^i_j,  y^i_{j+1}[$, the annulus $A^i_j$ is an entrance annulus.

\end{itemize}

\noindent
Here $\mu$ is the minimal set of the associated action of $\check{\Gamma}$ on $\mathbb S^1 \cong \mathbb{RP}^1$.
In the case of pieces of geodesic flows this characterization is easy to see.
For an explicit proof that applies to a more general setting, see
\cite[Corollaire $3.15$]{Ba3}.

We can say more: assume that we are in the case of entering annulus, i.e. the case where $]y^i_j,  y^i_{j+1}[$ is a gap of $\tilde{\mu}$. Then, the following ``triangle''

$$\Delta(\theta_j^i) :=  \{ (x,y) \in \widetilde{\mathbb R\mathbb P}^1 \times \widetilde{\mathbb R\mathbb P}^1 \; \mid \; \tilde{\tau}^{-1}_0(x^i_{j+1}) < x < x^i_j , \; \ y_j^i < y < y^i_{j+1} \}$$

\noindent
has the following geometric interpretation: it is the projection of the orbits in $\widetilde{T}_i^-$ trapped between the part of the stable and unstable leaves of $\theta^i_j$ contained in $\widetilde{T}_i^-$. In the case
where $A^i_j$ is an exit annulus, the triangle in $\Omega_0$ satisfying this property is:
$$\Delta(\theta_j^i) :=  \{ (x,y) \in \widetilde{\mathbb R\mathbb P}^1 \times \widetilde{\mathbb R\mathbb P}^1 \; \mid \; x^i_j < x < x^i_{j+1} , \; \ \tilde{\tau}^{-1}_0(y^i_{j+1}) < y < y^i_{j} \}$$

\noindent
These formulas define $\Delta(\theta^i_j)$ in the case of entering and exiting annuli.

\begin{figure}
  \includegraphics[scale=1]{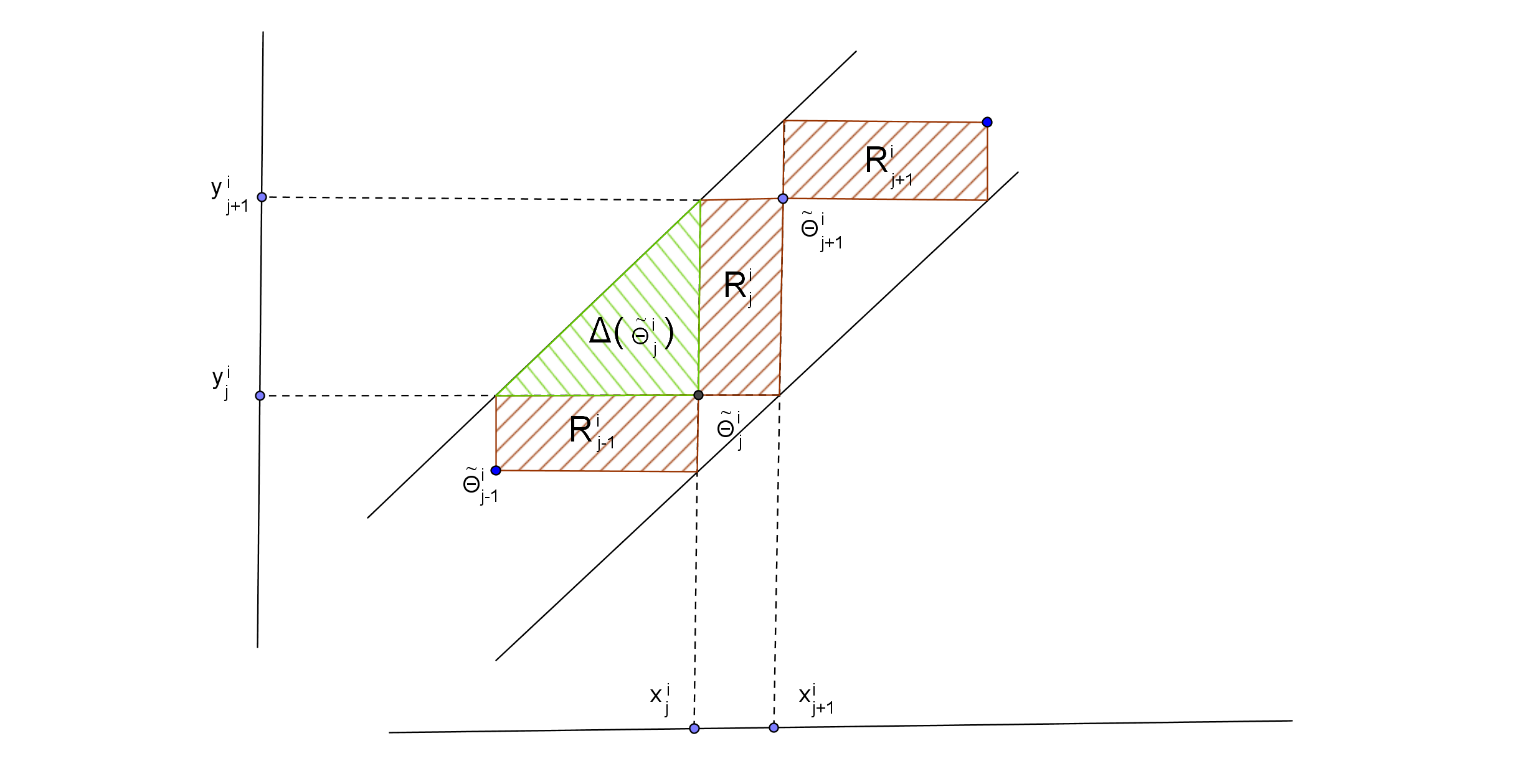}\\
  \caption{Triangle for a tangent periodic orbit in the case where $]y^i_j,  y^i_{j+1}[$ is a gap of $\tilde{\mu}$, hence
  in the case of entering annulus.}\label{fig:triangle}
\end{figure}

The non-wandering set $\mathcal M_0$ of $\Psi_0$ is precisely the projection in $M_0$ of the union $\widetilde{\mathcal M}_0$ of the orbits whose projection in $\Omega_0$ lies in $\tilde{\mu} \times \tilde{\mu}$. It is also the closure of
the union of the periodic orbits of $\Psi_0$.
Observe that the lift $\widetilde{\mathcal M}_0$ to the universal
cover contains the tangent periodic orbits $\theta_j^i$ for every $i$, $j$. In addition
this lift is contained in $\widetilde{P}_0$.

Finally, the projection of $\widetilde{P}_0$ (i.e. the set of orbits intersecting $\widetilde{P}_0$) is the complement in $\Omega_0$ of the union of the closure of the triangles $\Delta(\theta_j^i)$.
Here $\theta^i_j$ runs over all the lifted tori $\widetilde T_i$ and all lifts of periodic orbits tangent to these
$\widetilde T_i$.

\subsection{Construction of flows on $3$-manifolds via group actions on the line}

\subsubsection{Hyperbolic blow up}
\label{blowup}

Recall that $\Gamma = \pi_1(P_0)$ where $P_0$ is a Seifert fibered space
and $\bar{\Gamma} = \Gamma/<h>$ where $h$ represents a regular fiber of the Seifert fibration
of $P_0$.

\begin{define}
A modification of a $(\sigma, \mu, \tau)$-representation $\rho_0: \Gamma \to$
Homeo$(\mathbb S^1)$ on a periodic gap $I_0$ is \textbf{hyperbolic} if the modified
restriction $f_0: I_0 \to I_0$ has only a finite number of fixed points in $I_0$, and
that all these fixed points (including the extremities $\partial I_0$ as maps
of $\mathbb S^1$) are hyperbolic.

A \textit{hyperbolic blow up of $\rho_0$} is a $(\sigma, \mu, \tau)$-representation obtained from the Fuchsian
representation $\rho_0$ by a finite number of hyperbolic modifications on periodic gaps.
\end{define}

A corollary of Remark \ref{rk:topoconj} is that the topological conjugacy class of a hyperbolic modification
is uniquely determined by the (even) number of fixed points introduced in the periodic gap. It also follows
from Remark \ref{rk:semiconj} that hyperbolic blow ups are semiconjugate to the initial Fuchsian representations
they are constructed from.

The chronological order of the hyperbolic modifications has no incidence on the conjugacy class of the resulting hyperbolic blow up.

According to Theorem \ref{thm:almostconvergencestable}, the representation $\bar{\rho}: \bar{\Gamma} \to$ Homeo$(\mathbb S^1)$ associated to a hyperbolic blow up has the almost (k)-convergence property. It follows that every element $\bar{\gamma}$ for which $\bar{\rho}(\bar{\gamma})$ is not of finite order has only a finite number of fixed points, which are all of hyperbolic type, but the number of fixed points is not always $2k$: some elements may have more than $2k$ fixed points. Observe also that every gap is periodic, since it is already true for Fuchsian actions.

\subsubsection{Constructing the orbit space}\label{sub:consorbit}
Let us now consider a Fuchsian representation $\rho_0: \Gamma \to \widetilde{\mbox{PGL}}(2, \mathbb R)$ as in
subsection \ref{alter}
(the inclusion map of a convex cocompact subgroup). It is a
$(\tilde{\sigma}, \tilde{\tau}_0, \tilde{\mu})$-representation.
Let $\rho^\ast_0 : \Gamma \to \widetilde{\mbox{PGL}}(2, \mathbb R)$ be the \textit{twisted} representation: for every $\gamma$ in $\Gamma$, if $\gamma$ preserves the orientation we have $\rho^\ast_0(\gamma)=\rho_0(\gamma)$,
and if $\gamma$ reverses the orientation we have $\rho^\ast_0(\gamma) = \tilde{\tau}_0 \circ \rho_0(\gamma)$. We have already observed that $\rho^\ast_0$ is indeed a representation, and that the action
$(x,y) \mapsto (\rho_0(\gamma)x, \rho^\ast_0(\gamma)y)$ preserves the open domain $\Omega_0 = \{ (x,y)  \;|\;  x < y < \tilde{\tau}_0(x) \}$.
It is a $(\tilde{\sigma}^\ast, \tilde{\tau}_0, \tilde{\mu})$-representation, where $\tilde{\sigma}^\ast(\gamma)$ coincides with $\tilde{\sigma}(\gamma)$ when $\gamma$ is orientation preserving, and
$\tilde{\sigma}^*(\gamma)$ coincides with $\tilde{\tau}_0 \circ \sigma(\gamma)$ if not.

Orientation reversing elements do not stabilize any gap.
The reason is the following. Suppose that an orientation reversing element preserves a gap.
Then the  unperturbed representation also satisfies this. This representation is a lift of
a Fuchsian representation,
so it comes from a representation into PGL$(2,\mathbb R)$. Then there would be an orientation
reversing element preserving a gap. But gaps associated with Fuchsian
representations correspond to boundary elements in the
associated orbifold and these are orientation preserving, so this cannot happen and
establishes this fact.


\vskip .2in
\noindent
{\bf {The representations $\rho_1$ and $\rho_2$}} $-$
Let now  $\rho_1$, $\rho_2$ be hyperbolic blow ups of respectively  $\rho_0$, $\rho^\ast_0$.


We do not require the hyperbolic blow ups to be performed on the same gaps
for $\rho_1$ and $\rho_2$, they are performed in an independent way. We consider the action of $\Gamma$ on $\mathbb R \times \mathbb R$ defined by:
$$\gamma.(x, y) := (\rho_1(\gamma)x, \rho_2(\gamma)y)$$

Let $\Gamma^\ast$ be the subgroup
of index at most $2$ comprising orientation preserving elements of $\Gamma$.

\vskip .1in
We now very carefully construct
the orbit space of our eventual ``flow''. Its orbit space will be a $\Gamma$-invariant  open domain of $\Omega_0 \subset \mathbb R \times \mathbb R$, with properties
similar to that of $\Omega_0$. In particular, it will be the region between the graphs of two monotone non-decreasing maps from $\mathbb R$ into $\mathbb R$.

Let $I = ]a, b[$ be a gap of $\tilde{\mu}$. It is periodic. Let $\Gamma_I$ the stabilizer of $I$. It is generated by an element $\gamma_I$, that we can select so that $a$ is a repelling fixed point for $\rho_i(\gamma_I)$.
We compare the actions of $\rho_1(\gamma_I)$ and $\rho_2(\gamma_I)$ on $I$.

\vskip .15in
\noindent
{\bf {\underline{Construction of the map $\alpha_1$}}}

Let
$$x_1^1=a < x_2^1 < ... < x_{2p}^1=b \ \ {\rm be \ the \ fixed \ points \ of} \ \ \rho_1(\gamma_I) \ \ {\rm in} \ \
\bar{I}$$

\noindent
and
$$x_1^2=a < x_2^2 < ... < x_{2q}^2=b \ \ {\rm be \ the \ fixed \ points \ of} \ \ \rho_2(\gamma_I) \ \ {\rm in} \ \
\bar{I}$$

\noindent
Every $x^j_k$ is a repelling fixed point of $\rho_j(\gamma_I)$ if $k$ is odd, and an attracting fixed point if $k$ is even.
We select any $\Gamma_I$-equivariant increasing homeomorphism between $[a, x^1_2]$ and $[x^2_{2q-1}, b]$ realizing a topological conjugacy between $\rho_1(\gamma_I)$ and $\rho_2(\gamma_I)$ on these intervals.
If $x^1_2 < b$, we then extend this map on $[x^1_2, b]$ as the constant map taking the value $b$.
This defines a map $\alpha_1: I \to I$ which is a semi-conjugacy between the restrictions to $I$ of $\rho_1(\gamma_I)$ and $\rho_2(\gamma_I)$. We actually make the choices so that on $I$ we have always $x < \alpha_1(x)$
if $x$ is not an endpoint of $I$. Therefore

$$\alpha_1([x^1_2, b]) = b, \ \ \alpha_1([a ,x^1_2]) = [x^2_{2q-1}, b]$$

\vskip .15in
\noindent
{\bf {\underline{Construction of the map $\beta_1$}}}

Similarly, let

$$y_1^2=\tilde{\tau}_0(a) < y_2^2 < ... < y_{2r}^2=\tilde{\tau}_0(b) \ \ {\rm be \ the \ fixed \ points \ of}
 \ \ \rho_2(\gamma_I) \ \ {\rm in} \ \
\tilde{\tau}_0(I)$$

\noindent
The map $\beta_1$ is chosen to be the constant map on $[a, x^1_{2p-1}]$, taking the value $\tilde{\tau}_0(a)$, and
to be an increasing $\Gamma_I$ equivariant,
topological conjugacy between the restrictions of $\rho_1(\gamma_I)$ and $\rho_2(\gamma_I)$
on respectively $[x^1_{2p-1}, b]$ and $[\tilde{\tau}_0(a), y^2_2]$.
We actually adjust so that $\beta_1(x) < \tilde{\tau}_0(x)$ for $x$ in $I$.
Therefore

$$\beta_1([a, x^1_{2p-1}]) = \tilde \tau_0(a), \ \ \beta_1([x^1_{2p-1}, b]) = [\tilde \tau_0(a),y^2_2]$$

\begin{figure}
  \includegraphics[scale=1]{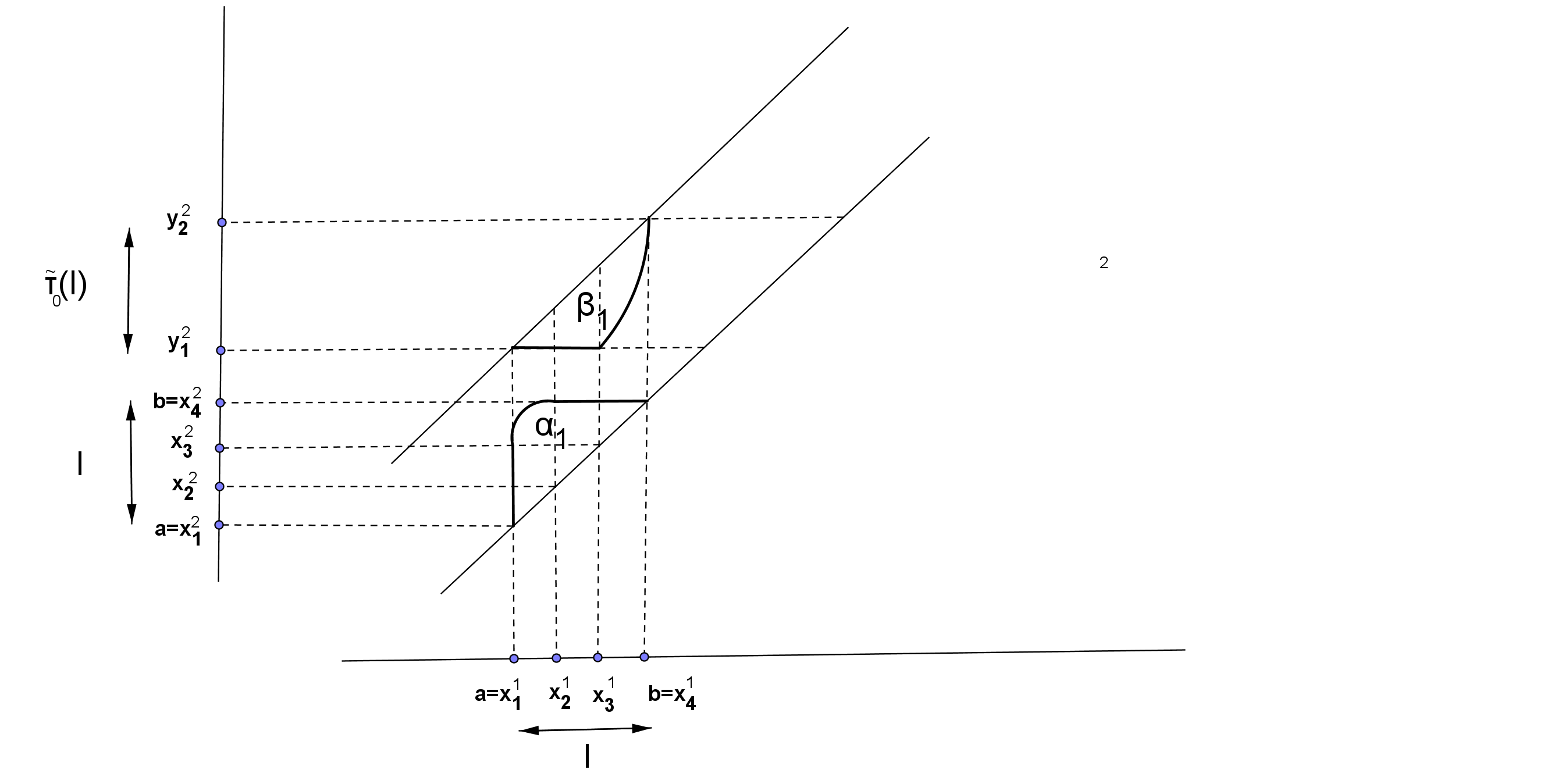}\\
  \caption{The construction of $\alpha_1$ and $\beta_1$ on a gap $I$. Here, $\rho_1(\gamma_I)$ has two
  fixed points in $I$, and $\rho_2(\gamma_I)$ has two fixed points in $I$ and no fixed points in $\tilde{\tau}(I)$.}\label{fig:alphabetaI}
\end{figure}

We then
extend $\alpha_1$ and $\beta_1$ to $\bigcup_{\gamma \in \Gamma^\ast} \tilde{\sigma}(\gamma)I$ by $\Gamma^\ast$-equivariance: the restriction of $\alpha_1$ (respectively $\beta_1$)
to $\rho_1(\gamma)I$ is defined as the conjugate $\rho_2(\gamma) \circ \alpha_1 \circ \rho_1(\gamma)^{-1}$ (respectively $\rho_2(\gamma) \circ \beta_1 \circ \rho_1(\gamma)^{-1}$).
Everywhere else we define $\alpha_1$ to be the identity and $\beta_1$ to coincide with $\tilde \tau_0$.

For orientation reversing element $\gamma$, we define the restriction of $\alpha_1$ to $\rho_1(\gamma)I$ as the conjugate $\rho_2(\gamma) \circ \beta_1 \circ \rho_1(\gamma)^{-1}$,
and the restriction of $\beta_1$ on the same interval as $\rho_2(\gamma) \circ \alpha_1 \circ \rho_1(\gamma)^{-1}$.

This is done for one $\sigma(\Gamma)$-orbit of a  gap.
We then apply the same procedure for other gaps $I$ in other $\sigma(\Gamma)$-orbits of gaps. The result are non-decreasing maps $\alpha_1, \beta_1: \mathbb R \to \mathbb R$ such that, for every element $\gamma$ of $\Gamma$, we have $\rho_2(\gamma) \circ \alpha_1 = \alpha_1 \circ \rho_1(\gamma)$ and $\rho_2(\gamma) \circ \beta_1 = \beta_1 \circ \rho_1(\gamma)$ if $\gamma$ preserves the orientation, and $\rho_2(\gamma) \circ \alpha_1 = \beta_1 \circ \rho_1(\gamma)$ and $\rho_2(\gamma) \circ \beta_1 = \alpha_1 \circ \rho_1(\gamma)$ if not.
We furthermore have $x \leq \alpha_1(x)$ and $\beta_1(x) \leq \tilde{\tau}_0(x)$ for every $x$ in $\mathbb R$.

It is easy to see that by construction
the map $\alpha_1$ is continuous on the right, but not on the left, and that $\beta_1$ is continuous on the left.
In the above setup,
for every gap $[a, b]$, the subinterval $]a, x^2_{2q-1}[$
(which may be empty if $x^2_{2q-1} = a$, that is, no blow up in that interval)
is not in the image of $\alpha_1$. We add every such vertical
segment $\{ a \} \times [a, x^2_{2q-1}[$ to the graph of $\alpha_1$: the result is a closed embedded line $L_-$ in $\mathbb R^2$. Similarly, we add vertical segments
to the graph of $\beta_1$,
obtaining a closed embedded line $L_+$ in the plane. The union $L_+ \cup L_-$
is the boundary of an open domain denoted by $\Omega$,
 which is invariant by the action of $\Gamma^\ast$.
The domain $\Omega$ can also be simply defined as follows:
$$\Omega := \{ (x,y) \in \mathbb R^2 \;|\; \alpha_1(x) < y < \beta_1(x) \}$$

Observe that since $x \leq \alpha_1(x) < \beta_1(x) \leq \tilde{\tau}_0(x)$, the region $\Omega$ is naturally included in $\Omega_0$: we call this
(non-equivariant) inclusion the \textit{canonical inclusion.}

We now introduce two maps $\alpha^-_1, \beta^+_1: \mathbb R \rightarrow \mathbb R$ defined as follows:
For every $x$ in $\mathbb R$, the intersection between $L_-$ and the vertical line $\{ x \} \times \mathbb R$ is a segment $\{ x \} \times [\alpha^-_1(x), \alpha_1(x)]$, whose maximal element is indeed $\alpha_1(x)$. Similarly, the intersection between $L^+$ and $\{ x \} \times \mathbb R$ is a segment $[\beta_1(x), \beta_1^+(x)]$. The maps $\alpha_1^-$ and $\beta_1^+$
are non decreasing and $\Gamma^\ast$-equivariant,
for instance because $\alpha^-_1 = \alpha_1$ in the gaps of $\tilde \mu$.
Moreover, $\alpha_1^-$ and $\beta^+_1$ coincide on $\tilde{\mu}$ with respectively the identity map and $\tilde{\tau}_0$.

It is easy to see that $\alpha^-_1$ is continuous on the left (whereas $\alpha_1$ is continuous on the right), and that $\beta^+_1$ is continuous on the right. Actually, $\alpha_1^-(x)$ can be defined as the limit
of $\alpha_1(x')$ for $x'$ converging to $x$ at the left.

\begin{figure}
  \includegraphics[scale=1]{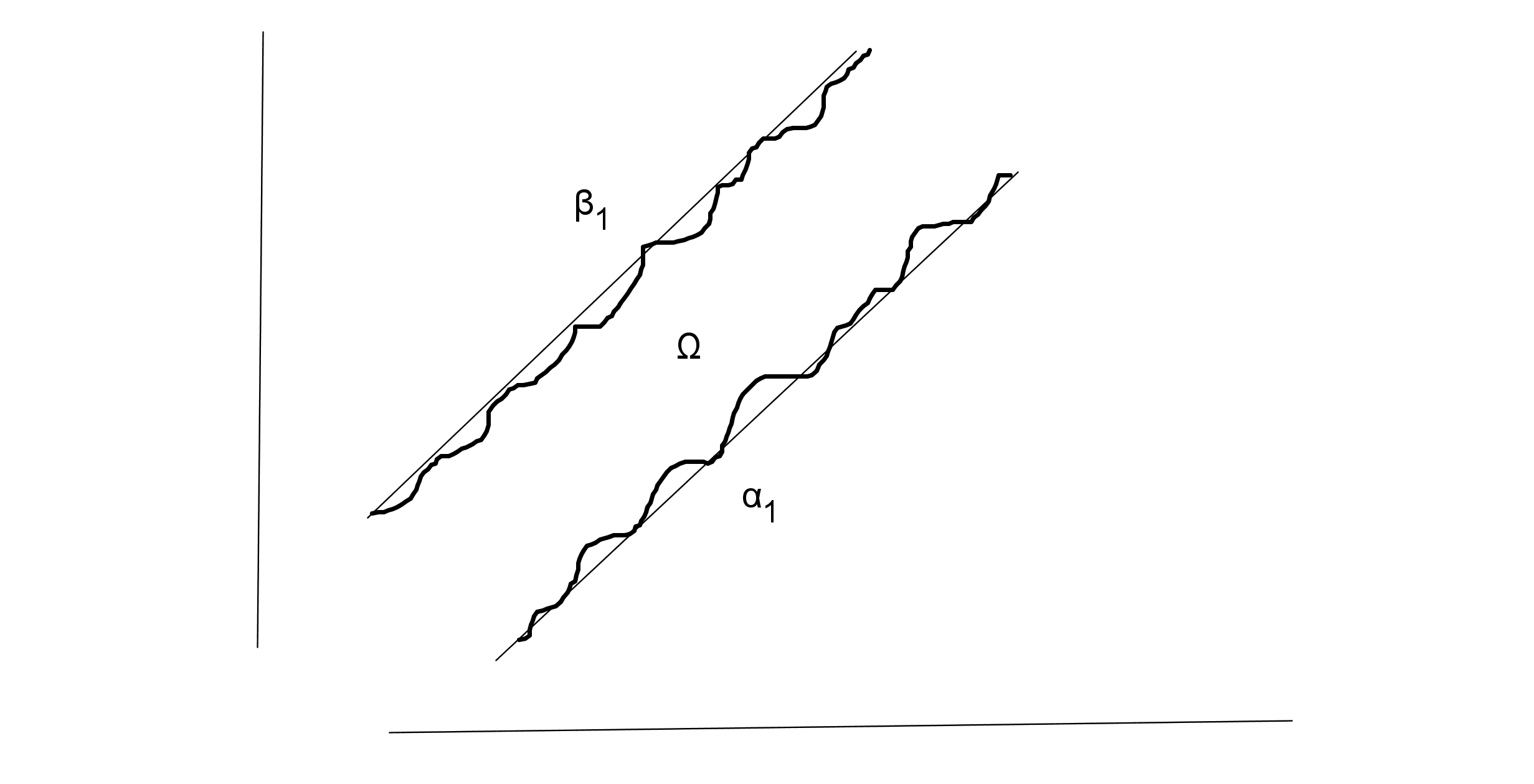}\\
  \caption{The maps $\alpha_1$, $\beta_1$ and the domain $\Omega$.}\label{fig:alphabetaglobal}
\end{figure}

Finally, we observe that $L_-$ and $L_+$ can also be considered as generalized
``graphs'' of maps
$\alpha_2$, $\beta_2$ in another way: the intersection between $L_+$ (respectively $L_-$) and every horizontal $\{ y \} \times \mathbb R$ is a segment (maybe reduced to a point) $[\alpha^-_2(y), \alpha_2(y)]$ (respectively $[\beta_2(y), \beta^+_2(y)]$. These
maps are non decreasing and $\Gamma^\ast$-equivariant.
The  open domain $\Omega$ can also be defined as:
$$\Omega := \{ (x,y) \in \mathbb R^2 \;|\; \alpha_2(y) < x < \beta_2(y) \}$$

The action of $\Gamma$ on $\mathbb R^2$ preserves $\Omega$ (but orientation reversing elements permute the connected components $L^\pm$ of the boundary).

\subsubsection{Constructing a model hyperbolic
blow up flow}
According to Remark \ref{rk:semiconj} there exist topological semiconjugacies $\varphi_1: \mathbb R \to \mathbb R$ and $\varphi_2: \mathbb R \to \mathbb R$
such that:
\begin{eqnarray*}
  \forall \gamma \in \Gamma & \rho_0(\gamma) \circ \varphi_1 = \varphi_1 \circ \rho_1(\gamma)&  \\
  \forall \gamma \in \Gamma & \rho^\ast_0(\gamma) \circ \varphi_2 = \varphi_2 \circ \rho_2(\gamma)&
\end{eqnarray*}

More precisely, we construct these semiconjugacies in the following way, using the notations introduced in
the previous subsection \ref{sub:consorbit}:
on $\tilde{\mu}$ the maps $\varphi_1$ and $\varphi_2$ coincide with the identity map.
Using the notation introduced in the previous subsection, then
on every (periodic) gap $I = ]a,b[$, the map $\varphi_1$ takes the
constant value $b$ on $[x^1_2, b]$, and on the interval $]a, x^1_2[$, $\varphi_1$ is any conjugacy between the restriction of $\rho_1(\gamma_I)$ and
the restriction of $\rho_0(\gamma_I)$ to $I$.
In addition we want that $\varphi_1$ is the identity in any gap $J$ that has not been
modified.

We then define $\varphi_2$ on $I$ as the unique map taking the constant value $a$ on $]a, x^2_{2q-1}[$ and
such that on $[x^2_{2q-1}, b]$ is satisfies:
$$\varphi_2 \circ \alpha_1 = \varphi_1$$
We then extend $\varphi_1$ and $\varphi_2$ on the entire $\mathbb R$ so that they are $\Gamma$-equivariant.

Then we define the map

$$\chi: \mathbb R^2 \to \mathbb R^2, \ \ \
{\rm by} \ \ \ \chi(x,y) \ := \ (\varphi_1(x), \varphi_2(y))$$

\noindent
This map is $\Gamma$-equivariant.
Furthermore, it follows from our choices that $\chi$ maps $L_-$ onto the graph of the identity map. Moreover, for every gap $I$
we have the following dichotomy concerning the image by $\chi$ of the
``triangle" $T_I := \{ (x,y) \;|\; x \in \bar{I}, \; \tilde{\tau}_0(a) \leq y \leq \beta_1(x) \}$:
\begin{itemize}
  \item either $\rho_0$ has not been modified in $I$ and $\rho_0^\ast$ has not
been modified in $\tilde{\tau}_0(I)$: in this case, $\chi$ maps the triangle
  $T_I$ (that in this case is $\{ (x,y) \;|\; x \in \bar{I}, \; \tilde{\tau}_0(a) \leq y \leq \tilde{\tau}_0(x) \}$) onto itself;

  \item or one (maybe both) of the actions on $I$ or $\tilde{\tau}_0(I)$ has
been modified: in this case, $\chi$ maps $T_I$ into
the union of the two sides
  $\{ (x,y) \;|\; x \in \bar{I}, \; y = \tilde{\tau}_0(a)\}$  $-$ the horizontal side of $T_I$,
and $\{ (x,y) \;|\; x = b, \; y \in \tilde{\tau}_0(\bar{I})\}$ $-$ the vertical side of $T_I$.

\end{itemize}

We explain this dichotomy. We use the notation of the definition of $\alpha_1$
and $\beta_1$ in subsection \ref{blowup}, with the fixed points $x$'s and $y$'s.
Let $J = \tilde \tau_0(I)$.
\begin{itemize}
\item
In the first case there is no change of the representations, therefore $p = 1$ and $r = 1$ $-$ notation
from subsection \ref{blowup}, hence on $I$ we have
$\alpha_1 = id$ and $\beta_1 = \tilde \tau_0.$
Therefore the triangle $T_I$ is an actual triangle, and in this case $\varphi_1 = id$
in $I$ and $\varphi_2 = id$ in $J$.

\item
In the second case, suppose first that there is some modification of $\rho_0$ in $I$.
Therefore $p \geq 2$ and in particular $x^1_2 < x^1_{2p-1}$.
Here $\alpha_1$ sends $[x^1_2,b]$ to $b$ and $\beta_1$ is constant $= \tilde \tau_0(a)$
on $[a,x^1_{2p-1}]$. If $x \leq x^1_{2p-1}$, then $T_I \cap \{ x \} \times {\mathbb R}
= (x,\tilde \tau_0(a))$, so $\beta_1(x) = \tilde \tau_0(a)$. Then $\chi$ sends the
point $(x,\tilde \tau_0(a))$ to a point in the horizontal side of $T_I$. Recall
that $\varphi_1, \varphi_2$ are equal to $id$ in $\tilde \mu$.
If on the other hand $x > x^1_{2p-1}$ then by definition $\varphi_1(x) = b$ and the
image under $\chi$ of that part of $T_I$ is in the vertical side of $T_I$.

\item
Finally suppose that $\rho^*_0$ has been modified in $J = \tilde \tau_0(I)$.
Points in $T_I$ have the form $(x,y)$ with $y \in [\tilde \tau_0(a), \beta_1(x)]$.
Since $\rho^*_0$ has been modified in $\tilde \tau_0(I)$, the definition of
$\beta_1$ implies that $y \leq y^2_2$ in $I$.
We now consider the image of this under $\varphi_2$. The subtle point is that we
have to apply the definition of $\varphi_2$ to the interval $J = \tilde \tau_0(I)$ and
not to $I$. In particular $a$ of that definition corresponds to $\tilde \tau_0(a)$ in $J$
and $x^2_{2q-1}$ of that definition  corresponds to $y^2_{2q-1}$. Since $q \geq 2$,
$2q-1 > 2$ and so $\varphi_2$ sends $[\tilde \tau_0(a),\beta_1(y)]$ to $\tilde \tau_0(a)$,
so the image is in the horizontal side of $T_I$.
This finishes the proof of the dichomotomy.

\end{itemize}

It follows that $\chi$ maps $L_+$ into the closure of $\Omega_0$, and therefore
that the image of $\Omega$ by $\chi$ is contained
in $\Omega_0$ (see figure \ref{fig:collapsingsomegaps}).

\begin{figure}
  \includegraphics[scale=1]{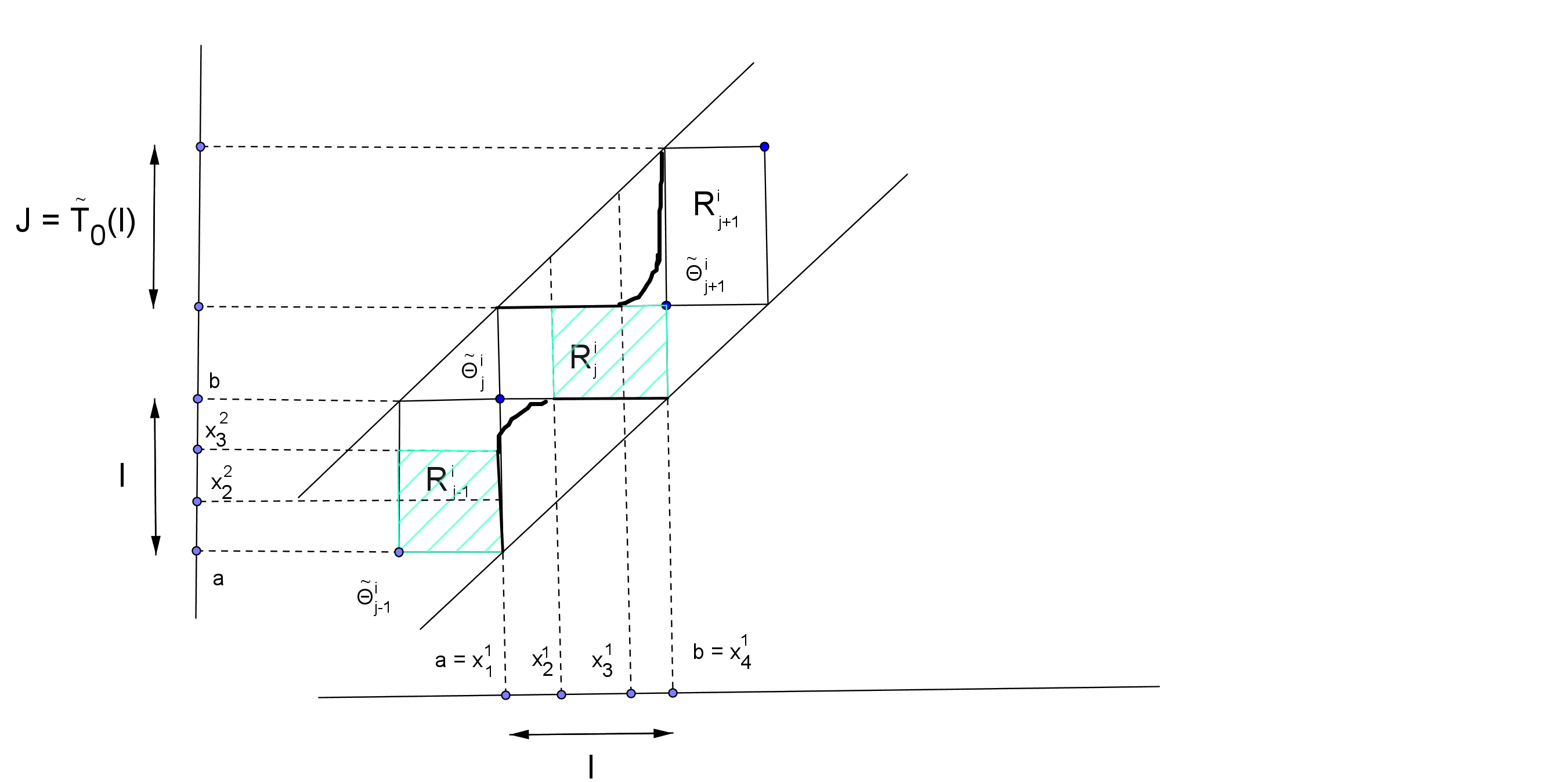}\\
  \caption{The map $\chi$ collapses the stripped rectangles, and the triangle contained in $I \times J$.}\label{fig:collapsingsomegaps}
\end{figure}

We now consider the oriented line bundle $\widetilde{M}(\Omega)$ over $\Omega$
which is the pull-back by $\chi$ of the line bundle
$\widetilde{M}(\Omega_0) \to \Omega_0$ defined in section \ref{sub:geodflow}. In other words:
$$\widetilde{M}(\Omega):= \{ (x,y,m) \in \mathbb R^3 \;|\; \alpha_1(x) < y < \beta_1(x)\}$$
equipped with the following action of $\Gamma$:
$$\forall \gamma \in \Gamma \;\; \
\gamma.(x,y,m) \ = \ (\rho_1(\gamma)x, \rho_2(\gamma)y, \ m - \log|(\rho_0(\gamma))'(\varphi_1(x))| + \log|(\rho_0^\ast(\gamma))'(\varphi_2(y))|)$$
Therefore, $\chi$ induces a $\Gamma$-equivariant map from $\widetilde{M}(\Omega)$ into $\widetilde{M}(\Omega_0),$ simply defined by:
$$(x,y,m) \to (\chi(x,y),m)$$
Since the action on
$\widetilde{M}(\Omega_0)$ is free and proper, the action of $\Gamma$ on $\widetilde{M}(\Omega)$ is free and proper too.
Let the quotient
$$M_\Gamma(\Omega) \ := \ {\Gamma \backslash}\widetilde{M}(\Omega)$$

\noindent
This is a $3$-manifold equipped with an oriented
one dimensional foliation $\Psi$,
which is the projection of the foliation
induced by the fibers over $\Omega$.
More specifically the projection of $(x,y) \times {\bf R}$
for $(x,y)$ in $\Omega$.
The manifold $M_\Gamma(\Omega)$
is also equipped with two foliations of codimension one
$\Lambda^s(\Psi)$ and $\Lambda^u(\Psi)$ induced by the horizontal and vertical foliations of $\Omega$.
Since the modification is a hyperbolic blow up, it follows that these foliations have
transversely hyperbolic behavior, that is, the flow $\Psi$ is
``essentially" an Anosov flow.
The map $\chi$
induces a
{\underline {semi-conjugacy}} $\chi_\Gamma: M_\Gamma(\Omega) \to
M_\Gamma(\Omega_0)$, mapping the oriented orbits of $\Psi$ onto
the oriented orbits of $\Psi_0$, and mapping the foliations $\Lambda^s(\Psi)$,
$\Lambda^u(\Psi)$ onto the stable/unstable
foliations $\Lambda^s(\Psi_0)$, $\Lambda^u(\Psi_0)$.

\subsubsection{Cutting a Seifert piece in the blow-up flow}\label{sub:cutseifert}
In this section, we show that there is in $M_\Gamma(\Omega)$ a compact manifold $P_\Gamma(\Omega)$ which is a Seifert bundle homeomorphic to
the Seifert piece $P_0$ (see Section \ref{sub:cutseifertgeod}) and whose
boundary is a union of embedded Birkhoff tori for the flow $\Psi$.
Moreover, the submanifold $P_\Gamma(\Omega)$ is unique up to topological conjugacies preserving the restrictions of the
 stable and unstable foliations.

Recall that in subsection \ref{sub:cutseifertgeod} we did
the following: $M_{\Gamma}(\Omega_0)  = \Gamma \backslash
 \widetilde{M}(\Omega_0)$
and $P_0$ is a compact submanifold bounded by Birkhoff tori.
In this section we consider
$M_{\Gamma}(\Omega_0)  = \Gamma \backslash \widetilde{M}(\Omega)$.
Notice that we abuse notation using the same $\Gamma$ for the
original action on $\widetilde{M}(\Omega_0)$ and the new
action from the hyperbolic blow up on $\widetilde{M}(\Omega)$.

Let $T_i$ be one peripheral torus of $M_0 \cong M_{\Gamma}(\Omega_0)$,
and let $H_i \approx \mathbb Z^2 \subset \Gamma$ be its fundamental group.
As recalled in Section \ref{sub:cutseifertgeod}, $H_i$ preserves a string of lozenges $(R^i_j)_{j \in \mathbb Z}$ in $\Omega_0$,
and the stabilizer of the corners $\tilde{\theta}^i_j$ is a cyclic subgroup $D_i$.
Recall that
the blow up action is obtained by blowing up gaps of $\mu$. These are associated with entering or
exiting annuli for $\Psi_0$, which generated the lozenges $R^i_j$.
These lozenges $R^i_j$ are
still contained in $\Omega$, i.e. in the image of the canonical inclusion $\Omega \subset \Omega_0$, and are in some sense still
preserved by $H_i$, for the new action through $(\rho_1, \rho_2)$
(instead of $(\rho_0,\rho^*_0)$).
The stabilizer of the corners is still the cyclic subgroup $D_i$.
However, some $R^i_j$
may have been decomposed into finitely many $D_i$-invariant sub-lozenges 
because there may be more fixed points under the (new)
action of $D_i$ (under $(\rho_1,\rho_2)$).
But still there are invariant lozenges exactly because
the blow up is hyperbolic.
More specifically, suppose that the associated gap of $\tilde \mu$ is $]a,b[$ and
the gap is exiting, that is, the gap is in the $x$ direction. Then the lozenge $R^i_j$
corresponds to the rectangle $]a,b[ \times ]b,\tilde \tau_0(a)[$.
Here $b$ is in $\tilde \mu$ and is isolated on the $]a,b[$ side,
therefore $b$ is not isolated in $\tilde \mu$ in the $]b,\tilde \tau_0(a)[$
side.
In particular $]b,\tilde \tau_0(a)[$
is \underline{not} a gap of $\tilde \mu$.
It follows that there is no blow up in the interval $]b,\tilde \tau_0(a)[$ abutting $b$ and
the action of $D_i$ on this interval is hyperbolic without fixed points.
So if $a = x^1_1, x^1_2,... , x^1_{2p} = b$ are the fixed points of the blown up action
in $]a,b[$ it follows that the lozenge $R^i_j$ splits into $2p - 1$ lozenges, forming
a chain of adjacent lozenges, all intersecting a common unstable leaf.


It follows that
one can easily find as in \cite{Ba3} a $D_i$-invariant section of the bundle $\widetilde{M}(\Omega)$
over each new lozenge $S^i_j$, so that their
images are $D_i$-invariant Birkhoff bands transverse to the fibers and with boundary the fibers above the (new) corners $\tilde{\theta}^i_j$ and $\tilde{\theta}_{j+1}^i$.
Altogether all these Birkhoff bands form a closed embedded plane $\widetilde{T}'_i$ that one can select to be $H_i$-invariant (simply by
taking, for every $\gamma$ in $H_i$, the section over $\gamma S^i_j$ the image under $\gamma$ of the section over $S^i_j$). As in \cite{Ba3}
or \cite{bafe1}, one shows, by cut and paste techniques,  that the collection of
these embedded planes, when we consider all the
boundary tori $T_i$ of $M_0$, can be chosen so that their
projections in $M_\Gamma(\Omega)$ are embedded Birkhoff tori $T'_i$ that are two by two disjoint.

\vskip .07in
\noindent
{\bf {Conclusion:}} There are finitely many embedded tori
$\{ T'_i \}$ in $M_{\Gamma}(\Omega)$ with union a compact
subset of $M_{\Gamma}(\Omega)$.
\vskip .05in

Let $\widetilde{\mathcal M}$ be the preimage in $\widetilde{M}(\Omega)$ of $\tilde{\mu} \times \tilde{\mu}$, and let $\mathcal M$ be its projection
in $M_\Gamma(\Omega)$. Observe that the restriction of
$\chi_{\Gamma}$
to $(\tilde{\mu} \times \tilde{\mu}) \cap \Omega$ is a homeomorphism onto
$(\tilde{\mu} \times \tilde{\mu}) \cap \Omega_0 =
(\tilde{\mu} \times \tilde{\mu}) \cap \Omega$ (it is the identity map!).
It follows that $\mathcal M$ is naturally identified with the non-wandering
set for the geodesic flow $(M_\Gamma(\Omega_0), \Psi_0)$, and
therefore the very important consequence that $\mathcal M$ is
compact. Observe also that $\mathcal M$ intersects the tori $T'_i$ only at
(possibly a subcollection of) their tangent periodic orbits. In particular, there is no
orbit in $\mathcal M$ crossing one $T'_i$.
This is because an orbit crossing $T'_i$ lifts to an orbit in a lozenge as above. Then either
its stable or unstable leaf is in a gap of $\tilde \mu$.

\vskip .1in
Consider one of  these tori, and denote it by $T'_i$.
Consider one of its lifts $\widetilde{T}'_i$. Since
$\widetilde{T}'_i$ is a closed embedded plane, it disconnects
$\widetilde{M}(\Omega)$ in two connected components $(\widetilde{T}'_i)^+$, $(\widetilde{T}'_i)^-$.
We show that one of these components does not intersect $\widetilde{\mathcal M}$.
Assume by way of contradiction
that $(\widetilde{T}'_i)^+$ and $(\widetilde{T}'_i)^-$ each contain
one element $\tilde{\theta}^+$, $\tilde{\theta}^-$ respectively
of $\widetilde{\mathcal M}$. Periodic orbits are dense in $\mathcal M$, hence we can assume that $\tilde{\theta}^+, \tilde{\theta}^-$ are both lifts of periodic orbits.
Moreover, there is a sequence $\tilde{\theta}_0$, ... , $\tilde{\theta}_{2n}$ of elements in $\widetilde{\mathcal M}$ such that:

--  $\tilde{\theta}_0 = \tilde{\theta}^+$ and $\tilde{\theta}_{2n} = \tilde{\theta}^-$,

-- every $\tilde{\theta}_{2k}$ is the lift of a periodic orbit (for $0 \leq k \leq n$),

-- for every $k$ between $0$ and $n-1$, $\tilde{\theta}_{2k+1}$ is the intersection between the stable leaf of $\tilde{\theta}_{2k}$ and the unstable leaf of $\tilde{\theta}_{2k+2}$,
or the intersection between the unstable leaf of $\tilde{\theta}_{2k}$ and the stable leaf of $\tilde{\theta}_{2k+2}$
(this sequence corresponds to a polygonal line in $\Omega$ made of
vertical and horizontal segments joining $\tilde{\theta}^+$ to $\tilde{\theta}^-$).
This is certainly true for the geodesic flow and hence it follows for $\Psi$ also.
We can furthermore assume that no $\tilde{\theta}_k$ is one orbit tangent to $\widetilde{T}'_i$.

Then, as explained above,
none of the $\tilde{\theta}_k$ crosses $\widetilde{T}'_i$.
Since $\tilde{\theta}^+$ and $\tilde{\theta}^-$ lie in different connected components of
the complement, there is some integer $k$ such that
$\tilde{\theta}_{2k}$ lies in $(\widetilde{T}'_i)^+$ and $\tilde{\theta}_{2k+2}$ lies in $(\widetilde{T}'_i)^-$.
Then, $\tilde{\theta}_{2k+1}$ is an orbit that
in the past gets closer and closer to, say $\tilde{\theta}_{2k}$,
and in the future gets closer and closer to $\tilde{\theta}_{2k+2}$. Since $\tilde{\theta}_{2k}$ and $\tilde{\theta}_{2k+2}$ project
to periodic orbits, they stay a minimum distance from $\widetilde{T}'_i$. Hence, $\tilde{\theta}_{2k+1}$ must cross $\widetilde{T}'_i$: contradiction.

\vskip .1in
Therefore, one of the connected components $(\widetilde{T}'_i)^+$, $(\widetilde{T}'_i)^-$ is disjoint from
$\widetilde{\mathcal M}$: we fix the notation so that this connected component is $(\widetilde{T}'_i)^-$. Consider the intersection of all
connected components $(\widetilde{T}'_i)^+$ through all the possible
embedded Birkhoff planes $\widetilde{T}'_i$.
This intersection
is the interior of a manifold with boundary $\widetilde{P}(\Omega)$, which contains $\widetilde{\mathcal M}$ and also every Birkhoff plane $\widetilde{T}'_i$ (indeed, $\widetilde{T}'_i$ contains elements of $\widetilde{\mathcal M}$: lifts
of some tangent periodic orbits, hence $\widetilde{T}'_i$
cannot be on the $(\widetilde{T}'_j)^-$ side of some other Birkhoff plane $\widetilde{T}'_j$). Moreover, $\widetilde{P}(\Omega)$ is $\Gamma$-invariant, hence projects in $M_\Gamma(\Omega)$ to a manifold with boundary ${P}_\Gamma(\Omega)$, whose boundary is the union of all the Birkhoff tori $T'_i$.

Since $\widetilde{T}_i$ is a Birkhoff plane, it follows as in
\cite{Ba3} that every orbit of $\widetilde{\Psi}$
can intersect any $\widetilde{T}'_i$ at most once. In addition $(\widetilde{T}'_i)^-$ is disjoint from the
other Birkhoff planes $\widetilde{T}'_j$. It follows that if an orbit
of $\widetilde{\Psi}$ crosses $T'_i$ transversely and enters $(\widetilde{T}'_i)^-$,
it cannot intersect $\widetilde{T}'_i$ after that,
and hence stays trapped in $(\widetilde{T}'_i)^-$. Moreover an element of $\Gamma$ preserving
$(\widetilde{T}'_i)^-$ also preserves its boundary $(\widetilde{T}'_i)^-$, because $\widetilde{P}(\Omega)$
is $\Gamma$ invariant.
Hence this element of $\Gamma$ must
be an element of $H_i$. It follows that the projection of $(\widetilde{T}'_i)^-$ in $M_\Gamma(\Omega)$ is a domain
$(T'_i)^-$ disjoint from $\mathcal M$, with boundary $T'_i$,
homotopic to $T'_i \times ]0, +\infty|$,
and such that every orbit of $\Psi$ crossing $T'_i$ and entering $(T'_i)^-$ remains trapped in $(T'_i)^-$.

The next step is to identify what are the entrance/exit Birkhoff annuli in the boundary
$\partial {P}_\Gamma(\Omega)$.
As before let $S^i_j = ]x^i_j,  x^i_{j+1}[ \ \times \ ]y^i_j,  y^i_{j+1}[$
be a lozenge in $\Omega$, projection
of a Birkhoff band of a Birkhoff plane $\widetilde{T}'_i$.
It corresponds to a transverse
annulus $(A')^i_j$ of $T'_i$. One (and only one) of the sides $]x^i_j,  x^i_{j+1}[$, $]y^i_j,  y^i_{j+1}[$ is a gap of the minimal set $\tilde{\mu}$.
If the gap is the horizontal side $]x^i_j,  x^i_{j+1}[$, we define:
$$\Delta'(\tilde{\theta}_j^i) :=  \{ (x,y) \in \Omega \; \mid \; x^i_j < x < x^i_{j+1} , \; \tilde{\tau}^{-1}_0(y^i_{j+1}) < y < y^i_{j} \}$$

If the gap is $]y^i_j,  y^i_{j+1}[$, we define:
$$\Delta'(\tilde{\theta}_j^i) :=  \{ (x,y) \in \Omega \; \mid \; \tilde{\tau}^{-1}_0(x^i_{j+1}) < x < x^i_j , \; y_j^i < y < y^i_{j+1} \}$$

Then, as in the case of the geodesic flow (Section \ref{sub:cutseifertgeod}) the first case is the case where $A'_i$ is an exit annulus, whereas
in the second case, $A'_i$ is an entrance annulus. Indeed, for example in the second case, the
lozenge $S^i_j$ cannot be crossed
by a horizontal-unstable leaf containing an orbit of $\widetilde{\mathcal M}$, hence $S^i_j$
is crossed by vertical-stable leaves
containing orbits of $\widetilde{\mathcal M}$.
Therefore, the projection of an orbit in such a stable leaf must accumulate in the future
in an element of $\mathcal M$, meaning than
it cannot enter in $(T'_i)^-$ since it would be a ``non-return in ${P}_\Gamma(\Omega)$'' option (compare with \cite[Corollaire $3.15$]{Ba3}).

Furthermore, similarly to the situation
in \cite{Ba3}, the triangles $\Delta'(\tilde{\theta}_j^i)$ are the projections in $\Omega$ of the orbits in $(\widetilde{T}'_i)^-$ trapped between the part of the stable and unstable leaves of $\tilde{\theta}^i_j$.
The reason is that $\Delta'(\tilde{\theta}_j^i)$ is one of the four connected
components of $\Omega$ with the stable and unstable leaves of $\tilde{\theta}^i_j$ removed.
This is true even if one blows up the vertical
interval $](\tilde{\tau}_0)^{-1}(y^i_{j+1},y^i_j[)$
because we only consider points in $\Omega$.
This is different from what happened in \cite{Ba3}.
This component is not one of the components whose projection
contains a lozenge adjacent to $\tilde{\theta}^i_j$, hence orbits in $\Delta'(\tilde{\theta}_j^i)$ do not cross $\widetilde{T}'_i$.
Furthermore, since one of the sides
of $\Delta'(\tilde{\theta}_j^i)$ is a gap of $\tilde \mu$, $\Delta'(\tilde{\theta}_j^i)$ contains no element of $\widetilde{\mathcal M}$, hence
cannot be in the quadrant
of $\tilde{\theta}^i_j$ containing elements of $\widetilde{\mathcal M}$ accumulating non-trivially in $\tilde{\theta}^i_j$,
i.e. the $(\widetilde{T}'_i)^+$ side. The claim follows.

Now we prove a very important property:

\begin{lemma}{}{}
The set ${P}_\Gamma(\Omega)$ is compact.
\end{lemma}

\begin{proof}{}
Let $U$ be a relatively compact open neighborhood
in $M_{\Gamma}(\Omega)$ of the (compact)
union of all the tori $T'_i$ and the compact invariant set $\mathcal M$.
Let $\theta = (x, y)$ be an element of $\Omega$.
Here we think of $\theta = (x,y)$ (or another element
of $\Omega$) as both a point in the plane
and as an orbit in $\widetilde{M}(\Omega)$.

We have three possibilities:

-- either $x$ and $y$ both lie in $\tilde{\mu}$; then $\theta \in \widetilde{\mathcal M}$,

-- or one of them lies in $\tilde{\mu}$, and the other lies in a gap. Suppose
that $x$ lies in a gap $]a,b[$.  There is $(z,y)$ in $\widetilde{\mathcal M}$.
Then the orbit of the flow associated with $(x,y)$ is backards asymptotic
to an orbit $(z,y)$ of $\widetilde{\mathcal M}$ and since $x$ is in a gap,
the forward orbit leaves $\widetilde{P}(\Omega)$ so the orbit $(x,y)$ crosses
some $\widetilde{T}'_i$ in a point $w$.
So this orbit is in $\widetilde{P}(\Omega)$ flow backwards
from $w$ and is in $(\widetilde T'_i)^-$ flow forwards from
$w$,


-- The final possibility is that
$x$ and $y$ are both outside the minimal set $\tilde{\mu}$. In this case there are two
possibilities.
One possibility is that $\theta$ lies in a triangle $\Delta'(\tilde{\theta}_j^i)$.
In this case
the orbit does not intersect any $\widetilde{T}'_i$ and hence this orbit
projects to an orbit in $M_{\Gamma}(\Omega)$
outside ${P}_\Gamma(\Omega)$. The other possibility is that
the orbit enters $\widetilde{P}(\Omega)$.
Since $x$ is in a gap of $\tilde \mu$ then this orbit has to
{\underline {exit}}
$\widetilde{P}(\Omega)$ through some $\widetilde{T}'_i$.
Similarly since $y$ is in a gap then the orbit has to
enter $\widetilde{P}(\Omega)$.
In other words
$\theta$ lies in two different lozenges associated
with chains of lozenges of tori,
and it crosses two different Birkhoff planes.

\vskip .08in
It follows that for every point $p$ in ${P}_\Gamma(\Omega)$, the future orbit of $p$ either intersects one Birkhoff torus $T'_i$
and afterwards enters in $({T}'_i)^-$, or the orbit is
forward asymptotic to an orbit in $\mathcal M$.
In both cases, there is a non negative
time $t$ such $\Psi^t(p)$ lies in $U$.
The same $t$ will apply to a neighborhood of $p$
as $U$ is open.
In the same way there is a non positive $t$ so that $\Psi_t(p)$ is
also in $U$.

Now since
the boundary of $U$ is compact, we claim that
there is a uniform positive upper bound $T$, meaning that for every
$p$ in $P_{\Gamma}(\Omega)$ there is a
time $0 < t < T$ such that $\Psi^t(p)$ lies in $U$.
We explain further: for any $p$ let $t_p$ be the infimum
of $t > 0$ so that $\Psi_{t}(p)$ is in $U$. If $p$ is in
$U$ then this is zero and $\Psi_{t_p}(p)$ is in $U$.
Otherwise
$\Psi_{t_p}(p)$ is in $\partial U$ and not in $U$.
If the claim is not true there are $p_i$ in $P_{\Gamma}(\Omega)$
so that $(t_{p_i})$ converges to infinity.
We cannot get a convergent subsequence of the $(p_i)$ as we
do not know yet that $P_{\Gamma}(\Omega)$ is compact. But up
to subsequence asssume that $q_i = \Psi_{t_{p_i}}(p_i)$
converges to $q$ in $\partial U$ as this is compact.
The $\Psi$ orbit segments from $p_i$ to $q_i$ intersect
$U \cup \partial U$ only in $q_i$. These orbit segments
have length converging to infinity at $t_i \rightarrow
\infty$ and they are entirely contained in $P_{\Gamma}(\Omega)$.
It follows that the {\underline {backward}} orbit of $q$
is contained in $P_{\Gamma}(\Omega)$ and
{\underline {does not intersect $U$}}.
But this contradicts the fact that there is $t' \leq 0$ so
that $\Psi_{t'}(q)$ is in $U$.
This proves the claim.

\vskip .08in
Similarly, there is a time $T'$ such that for every $q$ in
$P_{\Gamma}(\Omega)$
there is a time $t$ with $-T' < t < 0$ such that
$\Psi^t(p)$ lies in $U$.
It follows that every $p$ in ${P}_\Gamma(\Omega) U$
lies in a segment of orbit of time-length $< T+T'$ joining two points in $\partial U$.
Since the closure of $U$ is compact,
it now follows that ${P}_\Gamma(\Omega)$ is compact.

This proves the lemma.
\end{proof}

Finally, ${P}_\Gamma(\Omega)$ is irreducible (since its universal covering $\widetilde{P}(\Omega)$ is contractible) and its fundamental group
$\Gamma$ is the fundamental group of the Seifert manifold $P_0$: it follows that ${P}_\Gamma(\Omega)$ is homeomorphic to $P_0$.
This follows from Scott's result that there are no fake Seifert fibered spaces
\cite{Sco2}.
The boundary components of ${P}_\Gamma(\Omega)$ are embedded Birkhoff tori, whose associated chain of lozenges in the orbit space $\Omega$
is prescribed from the beginning.

\begin{define}\label{def:hyperbolicblowup}
The manifold with boundary ${P}_\Gamma(\Omega)$, equipped with the restriction
of the oriented foliation $\Psi$ and the restrictions
$\hat{\Lambda}^s_\Gamma$, $\hat{\Lambda}^u_\Gamma$ of the stable, unstable
foliations $\Lambda^s(\Psi)$, $\Lambda^u(\Psi)$,
is a \textbf{hyperbolic blow up piece of a geodesic flow.}
\end{define}

\begin{remark}\label{rk:uniqueblowup}{\em
During the construction, we have performed several choices: the maps
$\alpha_1$, $\beta_1$, $\chi$, the boundary tori $T'_i.$
But the orbits in $M_\Gamma(\Omega)$ intersecting
${P}_\Gamma(\Omega)$ are the projections of the elements of $\Omega$ which are
\underline{not} in the triangles $\Delta'(\tilde{\theta}^i_j)$.
Consider a modification of $\rho_0$ in a gap $I$ of
$\tilde\mu$. In other words $\alpha_1$ and $\beta_1$
are {\underline {not}} $id$ and $\tilde\tau_0$ in that gap. But the
corresponding graphs of $\alpha_1$ and $\beta_1$ are
both in the excluded triangles: the graph of $\alpha_1$
is in $\Delta'(\tilde\theta^i_j)$ for some $j$ and
the graph of $\beta_1$ is in $\Delta'(\tilde\theta^i_{j+1})$.
Hence the region in $\Omega$ corresponding to orbits intersecting
${P}_\Gamma(\Omega)$ does not depend on these choices. In other words, the
choices of $\alpha_1$, $\beta_1$, $\chi$
only contribute to the definition of the flow in
the regions $(T'_i)^-$, i.e. outside ${P}_\Gamma(\Omega)$.

Hence the only choice that matters is the selection of the embedded Birkhoff tori $T'_i$. But,
as we will see in the proof of Theorem \ref{main}, it follows from
Remark \ref{rk:isotopybirkhofftori}
that the hyperbolic blow up $({P}_\Gamma(\Omega), \hat{\Lambda}^s_\Gamma, \hat{\Lambda}^u_\Gamma)$
does not depend on these choices, and therefore, is uniquely defined, up to orbital equivalence,
by the initial piece of geodesic flow and the number
of tangent periodic orbits introduced in every boundary torus.
}
\end{remark}

\section{Leaf spaces and orbit spaces associated to a free Seifert piece}
\label{leaforbit}

From now on, we consider a free Seifert piece $P$ of
an {\underline {arbitrary}}
pseudo-Anosov flow $(M, \Phi)$. We denote by $h$ an element of $\pi_1(P)$ represented by a regular fiber
of the Seifert fibration. Therefore, $h$ is a generator of the pseudo-center of $\pi_1(P)$.
Recall that $\hhs$ (respectively $\hhu$) is the
leaf space of $\wls$ (resp. $\wlu$).

\subsection{Existence of $h$-invariant axis in the leaf spaces}

\begin{proposition}{}{}
The action of $h$ on $\mathcal H^s$ (respectively $\mathcal H^u$) admits an axis $\mathcal A^s$ (respectively $\mathcal A^u$)
which is a properly embedded real line. In other words these embedded lines are $h$-invariant and the action of $h$ on each of them is free, i.e. an action by translation.
\label{axesfirst}
\end{proposition}

\begin{proof}{}
In the case that $P$ is all of $M$ this was proved in  Theorem 4.1 of \cite{bafe1}.
The proof in the case $P$ is not all of $M$ is similar. We refer to \cite{bafe1} whenever
details are the same as in the case $P = M$.
Consider the action of $h$ on $\hhs$. Since $P$ is a free Seifert piece, $h$ acts freely on $\hhs$.
In \cite{Fe5} it is shown that there is a unique axis $\mathcal A^s$ for $h$.
There are two options: either $\mathcal A^s$ is a real line or $\mathcal A^s$ is an infinite union of closed
intervals. First we show  that the second case cannot happen.
Suppose that

$$\mathcal A^s \ = \ \bigcup_{i \in {\mathbb Z}} [x_i,y_i] \ = \ \bigcup_{i \in {\mathbb Z}} B_i,$$

\noindent
where $B_i = [x_i,y_i]$ are closed segments in $\hhs$ and $y_i$ is not separated from
$x_{i+1}$. Since the axis $\mathcal A^s$ is unique and $h$ is in the pseudo-center of $\pi_1(P)$,
then every element of $\pi_1(P)$ permutes the collection $\{ B_i \}$.
In fact $\pi_1(P)$ acts on the indexing set of this collection which is ${\mathbb Z}$.
The action preserves elements being neighbors. Some elements can reverse
the order in ${\mathbb Z}$.
As in Case 2 of Theorem 4.1 of \cite{bafe1}, $\pi_1(P)$ acts on ${\mathbb Z}$
and has a subgroup of index $\leq 4$ which is ${\mathbb Z}^2$. The difference
from \cite{bafe1} is that when $P = M$ this quickly implies a contradiction, which
is not the case here. In any case since $M$ is assumed orientable, then Lemma 5.3 of
\cite{bafe1} implies that $P$ is either $T^2 \times [0,1]$, where $T^2$ is the torus; or $P$ is a twisted $I$ bundle over the Klein bottle.

In the first case recall that the torus decomposition into Seifert fibered pieces is minimal, the only possibility
is that $P$ is the only piece of the JSJ decomposition and $M$ would be obtained by gluing one to the other the boundary components of
$M.$ Then $M$ would be a torus bundle over the circle, hence, by \cite{bafe1}, the pseudo-Anosov flow would be the suspension
of a linear diffeormorphism: it is in contradiction with the hypothesis $\mathcal A^s \neq \mathbb R.$

In the second case
$\pi_1(P) = \pi_1(K) = < a, b |  a b a^{-1} = b^{-1} >$, where $K$ is the Klein bottle.
In this situation
the only possible regular fibers for a Seifert fibration of $P$ are represented
either by $a^2$ or $b$. In our situation $\pi_1(P)$ acts on ${\mathbb Z}$
and $h$ acts freely on ${\mathbb Z}$. Otherwise for some $j$, $h(B_j) = B_j$ and
either $h(x_j) = x_j$ or if $h$ reversed the orientation in $[x_j,y_j]$ then
$h$ would fix an interior point of $[x_j,y_j]$, both
options contradict the free action of $h$ on $\hhs$.
Therefore $h$ acts as a translation on ${\mathbb Z}$.

Suppose first that $h$ is represented by $b$. Then $a h a^{-1} = h^{-1} \ (*)$. If $a$
acts freely on ${\mathbb Z}$, it is a translation and the equation is impossible. If $a$
fixes an element in ${\mathbb Z}$, then so does $a^2$ and so $P$ has a Seifert
fibration where the fiber does not act freely and $P$ would be a periodic Seifert
piece, contradiction to assumption.
Now suppose that $h$ is represented by $a^2$. Hence $a^2$ acts freely
on $\hhs$ and as proved in \cite{Fe5}, $a$ also acts freely on $\hhs$.
If $b$ acts freely, this contradicts equation $(*)$ above. If $b$ does not
act freely then $P$ has a Seifert fibration where the fiber is periodic, contradiction.
We conclude that this case for $\mathcal A^s$ cannot happen.

\vskip .1in
\noindent
{\bf {Conclusion}} $-$ The axis of $h$ is a real line $\mathcal A^s$.

Now we need to check whether $\mathcal A^s$ is properly embedded or not.
Examples where an axis of an element acting freely on $\hhs$ is not properly
embedded are very common and occur for instance in the Bonatti-Langevin
example \cite{Bo-La}. Parametrize the leaves of $\wls$ in $\mathcal A^s$
as $\{ s(t),  t \in \rrrr \}$ where $s(t)$ separates $s(t')$ from $s(t")$ if
and only if $t$ is between $t'$ and $t"$ in $\rrrr$. Without loss of generality
suppose that $\mathcal A^s$ is not properly embedded for $t \rightarrow \infty$, so

$$\lim_{t \rightarrow \infty} s(t) \ \ = \ \ \bigcup_{i \in I} C_i$$

\noindent
where $C_i$ is a collection of leaves non separated from each other in the side
the collection $\{ s(t) \}$ is limiting on. The set $I$ is an interval in ${\mathbb Z}$
which can be either finite or ${\mathbb Z}$ \cite{Fe2,Fe3}. Since $< h >$ is normal in
$\pi_1(P)$ and $\mathcal A^s$ is the unique axis for the action of $h$ on $\hhs$, it follows
that $\pi_1(P)$ preserves $\mathcal A^s$. In addition a subgroup $G$ of index at most $2$
in $\pi_1(P)$ preserves orientation in $\mathcal A^s$, so it preserves the collection
$\{ C_i, i \in I \}$. This already implies that $I = {\mathbb Z}$ and that $\pi_1(P)$ has
a finite index subgroup isomorphic to ${\mathbb Z}^2$. In the first part of this proof
we showed that then the pseudo Anosov flow is a suspension, and therefore $\mathcal H^s = \mathcal A^s$ is an embedded
real line as required.

We conclude that both $\mathcal A^s$ and $\mathcal A^u$ are homeomorphic to $\rrrr$ and each
is properly embedded in $\hhs$ or $\hhu$ respectively.
\end{proof}

\begin{remark}\label{rk:criteriainA}
{\em
One checks easily that element of $\mathcal A^s$ are characterized by the following property:
an element $s$ of $\hhs$ lies in $\mathcal A^s$ if and only if there is a component $U$ of the complement
of $s$ in $\hhs$ such that $h(U) \subset U$. Indeed, if $s$ lies in $\mathcal A^s$, then we can take as $U$ the
component containing $h(s)$. If $s$ does not lie in $\mathcal A^s$, then for any connected component
$U$ of $\hhs  -  \{ s \}$ we have:

-- either $U$ contains $\mathcal A^s$: then $s$ lies in $h(U)$ which therefore cannot be contained in $U$;

-- or $U$ is disjoint from $\mathcal A^s$: in this case, $U$ and $h(U)$ are disjoint.

Similarly, an element $u$ of $\hhu$ lies in $\mathcal A^u$ if and only if there is a component $U$ of the complement
of $s$ in $\hhu$ such that $h(U) \subset U$.
}
\end{remark}

\begin{remark}\label{rk:orienteA}
{\em
From now on, we fix the orientation of $\mathcal A^s$ and $\mathcal A^u$ (hence of $\hhs$ and $\hhu$) so that, for the induced
total order on $\mathcal A^s$ and $\mathcal A^u$ we have $h(x) > x$ for every leaf $x$.
}
\end{remark}

\subsection{Identifying $\Omega'$ in the orbit space}\label{sub:omegainorbitspace}

In this section, we construct a $\pi_1(P)$-invariant domain $\Omega_P$ in the orbit space $\mathcal O$ of our fixed flow $\Phi$. This domain is
naturally identified to a domain $\Omega'$ in
$\mathcal A^s \times \mathcal A^u$, enjoying properties similar to the properties satisfied by the domain $\Omega$ described in Section \ref{sub:consorbit}.

Given our fixed pseudo-Anosov flow $\Phi$, let

$$\Omega' \  =  \ \{ (s,u) \in \mathcal A^s \times \mathcal A^u \ | \
s \cap u \ \not = \ \emptyset \}$$

\noindent
Then, the map $\flat: \Omega' \to \oo$ mapping $(s,u)$ into the unique intersection point between $s$ and $u$ is continuous. This intersection
point is an orbit of $\widetilde{\Phi}$.
We denote by $\Omega_P$
the image of $\flat$.

The first step is to prove that this set is non-empty.

The properties of the stable and unstable foliations are similar, up to inversion of the flow. In the following lemma,
properties are stated for both foliations, but the proof is written for only one of them.

\begin{lemma}
\label{le:nonempty}
The set $\Omega_P$ intersects every leaf in $\mathcal A^s$ and every leaf in $\mathcal A^u$.
\end{lemma}

\begin{proof}
Let $s$ be an element of $\mathcal A^s$.
We abuse notation and think of $s$ both as a leaf in $\mathcal A^s$ and as a subset of the orbit space $\oo$.
Assume by contradiction that $s$ does not intersect any leaf in $\mathcal A^u$.
In particular since $\mathcal A^u$ is $h$-invariant then $h^n(s)$ also does not intersect any leaf in $\mathcal A^u$.
Define $C_0$ as the unique connected component of $\oo  -  s$ which contains $h^{-1}(s)$ (despite the notation, $C_0$ is not a lozenge).
The set $C_0$ It is an open
subset, saturated by $\oo^s$, with boundary contained in $s$. For every integer $n$, let $C_n = h^n(C_0)$. Then every $C_n$ is contained in $C_{n+1}$, more precisely,
the closure of $C_n$ is contained in $C_{n+1}$.

The union $C_\infty$ of all the $C_n$ is a $h$-invariant open subset of $\oo$, saturated by $\oo^s$. We claim that $C_\infty$ is the entire $\oo$. If not, since $\oo$ is connected,
$C_\infty$ is not closed: there is an element $p$ of $\oo  -  C_\infty$  and a sequence of points $p_n$ converging to $p$, such that every $p_n$ lies in some $C_{k(n)}$.
Consider a small neighborhood $U$ of $p$: for $n$ sufficiently big, every $p_n$ lies in $U$. If $U$ intersects only finitely many different iterates $h^k(s)$, then every $p_n$
eventually belongs to the same $C_k$, hence $p$ lies in the closure of $C_k$, i.e. in $h^k(s)$. But then $p \in C_{k+1}$, contradiction.
Therefore, $U$ intersects infinitely many iterates $h^n(s)$. It means that the leaf $\oo^s(p)$ is a limit of iterates of the leaf $s$.
Hence $h(\oos(p))$ is non separated from $\oos(p)$, a contradiction to the axis of $h$ being properly
embedded in $\hhs$.

We have proved that $C_\infty$ is the entire orbit space. Let $u$ be a leaf in $\mathcal A^u$, and $q$ an element of $u$. Let $n$ be the smaller integer such that $q$ belongs to $C_n$. Since $u$ is disjoint from $h^{n}(s)$, we have $u \subset C_n  -  C_{n-1}$. In fact, the union of unstable leaves belonging to $\mathcal A^u$ is connected because $\mathcal A^u$ is a properly embedded
line in $\hhu$. It follows that all the elements of $\mathcal A^u$ are leaves contained in $C_n  -  C_{n-1}$.
But this is a contradiction since $\mathcal A^u$ is $h$-invariant and that $C_n  -  C_{n-1}$ is disjoint from its $h$-iterates.
\end{proof}

Let $z \in \oo$. A stable prong of $z$ is a component of $\oos(z) - \{ z \}$. Sometimes we include
$z$ itself in the prong. The point $z$ is singular if and only if there are more than two prongs
at $z$. We also say these are the prongs of $\oos(z)$ at $z$.

\begin{lemma}
\label{le:2prongs}
Let $p$ be a point in $\Omega_P$. Then there are at most
two stable (respectively unstable) prongs at $p$ intersecting $\Omega_P$.
\end{lemma}

\begin{proof}
Let $s$ be the stable leaf of $p$.
First assume that there are $3$ points $p_1$, $p_2$, $p_3$ in $\Omega_P$ lying in different prongs of $s$ at $p$.
Let $u_1$, $u_2$, $u_3$ be the unstable leaves of $p_1$, $p_2$, $p_3$.
The leaves $u_1$, $u_2$, $u_3$ all lie in the axis $\mathcal A^u$, which is a line.
Hence one them,
say $u_1$, must disconnect the other two.
On the other hand, one can join $u_2$ to $u_3$ by a segment in $s$ avoiding the prong containing $p_1$.
This stable segment does not intersect $u_1$, hence $u_1$ does not disconnect $u_2$ from $u_3$. Contradiction.
\end{proof}

\begin{lemma}
\label{le:segment}
The intersection between $\Omega_P$ and a stable (or unstable) leaf is a segment.
\end{lemma}

\begin{proof}
According to Lemma \ref{le:2prongs} we just have to prove that the intersection between a stable leaf $s$ and $\Omega_P$ is connected.
Let $p_1$, $p_2$ be two elements of $s \cap \Omega_P$, and $p$ any element of the segment $[p_1, p_2]$ in $s$. Then the unstable
leaf $\oo^u(p)$ disconnects $\oo^u(p_1)$ from $\oo^u(p_2)$. It follows that it must be an element of $\mathcal A^u$. Therefore, $p = s \cap \oo^u(p)$
lies in $\Omega_P$.
\end{proof}

\begin{lemma}
\label{le:interval}
For any element $z$ of $\Omega_P$, there is a point $p$  in $\oo^s(z)$ so that $p$ is in the interior of a segment $I$ contained in $\oo^u(p) \cap \Omega_P$.
In particular if $q$ is in $I \subset \oou(p)$, then $\oos(q)$ is in the stable axis $\mathcal A^s$.
\label{interval}
\end{lemma}

\begin{proof}
Let $s$ be the stable leaf through $z$, and let $u$ be
the unstable leaf through $z$. We have $s \in \mathcal A^s$,
$u \in \mathcal A^u$.

\textbf{Case 1 - The stable leaf $s$ is not singular.}
In this case, we just take $p=z$.
Consider the projection $\oo \rightarrow \hhs$ taking a point to the stable leaf containing it.
The projection in $\cH^s$ of the points in $u$ near $p$ lie
in a neighborhood of $s$ in $\cH^s$. Since $s$ is not a branch point of $\hhs$, then these projections
are in $\mathcal A^s$ if the points in $u$ are sufficiently close to $z$.
Therefore, a small segment in $u$ is the required interval $I$.

\textbf{Case 2 - $s$ is singular}

\textit{Case 2.1 -  $z$ is the singular point.}
Let $(s_i)$ be a sequence in $\mathcal A^s$, with
$s_i > s$ and $(s_i)$ converging to $s$ in $\hhs$.
Since $z$ is the {\underline {singular point}}
in $s$ then for big enough $i$ the leaf $u$ intersects
$s_i$, consequently a prong of $u$ at $z$ intersects
$s_i$. It is crucial here that $z$ is the singular point,
for otherwise this may not be true. The same
reasoning applies if $s_i < s$ in $\mathcal A^s$.
Hence
there are
{\underline {exactly}} two prongs of $u$ at $z$ which project, near $z$, into the axis $\mathcal A^s$. Choose small segments of $u$
in these prongs. The union (including $z$) is the segment $I$ as desired.

\textit{Case 2.2 -  $z$ is not the singular point.}
Let $q$ be the singular point in $s$. Once more, since a neighborhood of $s$ in $\cH^s$ is described by the projections of all the unstable prongs at $q$, there are exactly
two prongs $I_1$, $I_2$ of $\oo^u(q)$ at $q$, so that for any
$w \in I_1 \cup I_2$ then $\oos(w)$ is in $\mathcal A^s$.

Let $s_1$ be the prong of $s$ at $q$ containing $z$.
Suppose first that there is no prong of $\oos(q) = s$ at $q$
between $s_1$ and $I_1$ and also no prong of $\oos(q) = s$ at
$q$ between $s_1$ and $I_2$.
Then, the projection of $I_1 \cup \{ q \} \cup I_2$ in $\cH^s$ near $q$
coincides with the projection of $u = \oou(z)$ near $z$.
In this case let $p = z$ and $I$ a small unstable segment in $u$
containing $p$ in the interior and we are done.

Hence we are left with the case that say there is a prong $s'$ of $s$ at
$q$ between $s_1$ and $I_1$.
Let $s^*$ be a leaf in $\mathcal A^s$ intersecting $I_1$ in
$w^*$. Then, according
to Lemma \ref{le:nonempty}, $s^*$ contains a point $x_1$ in $\Omega_P$,
and hence $\oou(x_1)$ is in $\mathcal A^u$.
Suppose first that $x_1$ is not $w^*$.
Then, $\oo^u(q)$
disconnects $\oo^u(x_1)$ from $u$.  Since $u$ and $\oo^u(x_1)$ both lie in $\mathcal A^u$, it follows that $\oo^u(q)$ lies in $\mathcal A^u$.
On the other hand if $x_1 = w^*$ then obviously $\oou(x_1) =
\oou(w^*) = \oou(q)$ (because $w^*$ is in $I_1 \subset \oou(q)$)
is also in $\mathcal A^u$.

Hence $\oou(q)$ is in $\mathcal A^u$ and $\oos(q) = s$ is in
$\mathcal A^s$.
Therefore, $q$ lies in $\Omega_P$.
We are back to the situation of Case 2.1: let $p = q$,
and let $I$ be a small segment contained in $I_1 \cup \{ q \}
\cup I_2$, with $p$ in the interior..
\end{proof}

Let $p_s: \Omega_P \rightarrow \hhs$ be the projection map.

\begin{lemma}
\label{le:piecewisepath}
$\Omega_P$ is pathwise connected, in fact any two points $p$, $q$ in $\Omega_P$ are connected by a piecewise
path made of stable and unstable segments.
\end{lemma}

\begin{proof}
By Lemma \ref{le:interval}, for any $s$ in $\mathcal A^s$ there is an open segment
$O$ in $\mathcal A^s$ containing $s$
so that any point $q$ in $\Omega_P \cap (p_s)^{-1}(O)$ can be reached from any point $r$  in $s \cap \Omega_P$ by a desired path.
That is, from a point $z$ in $s \cap \Omega_P$ go to the point $p$ as in Lemma \ref{le:interval} then along
the segment $I$ along an unstable leaf as in Lemma \ref{le:interval} and then along the stable leaf to the point $r$.
We are also using Lemma \ref{le:segment} to obtain this.
Any closed interval $J$ in $\mathcal A^s$ can be covered by these open intervals and hence has a finite subcovering.
Concatenating the paths above proves the Lemma.
\end{proof}

By Lemmas \ref{le:nonempty} and \ref{le:segment}, for every element $s$ of $\mathcal A^s$ the intersection $s \cap \Omega_P$ is a
non-empty segment (possibly a single point).
The projection $I^u(s)$ of $s \cap \Omega_P$ in $\cH^u$ is then a segment in $\mathcal A^u \approx \mathbb R$. Observe that this interval can be open or closed at each of its extremities.
In fact, the interval may be closed at an extremity whenever
there are singular orbits at the ``boundary" of the
Seifert piece $P$.
Moreover, a priori the interval
can be the entire axis $\mathcal A^u$ or a ray in it.
Recall that we have a total order $<$ on $\mathcal A^u \approx \mathbb R$ such that for every $u$ in $\mathcal A^u$ we have $h(u) > u$,
and similarly for $\mathcal A^s$.
We can then define functions $\alpha_s$, $\beta_s: \mathcal A^s \rightarrow \mathcal A^u \cup \{\pm  \infty \}$ with the following conventions:

-- If $I^u(s)$ is not bounded from above
in $\mathcal A^u$, then let $\beta_s(s) = +\infty$, otherwise let
$$\beta_s(s) \ = \  \mbox{Sup} \{ \  v \ | \ v \in I^u(s) \} \ \in \ \mathcal A^u$$

-- If $I^u(s)$ is not bounded from below in $\mathcal A^u$, then let
$\alpha_s(s) = -\infty$,
otherwise let
$$\alpha_s(s) \ = \ \mbox{Inf} \{ \ v \  \ | \ v \in I^u(s) \} \ \in \ \mathcal A^u$$


We can define in a similar way two functions $\alpha_u$, $\beta_u: \mathcal A^u \rightarrow \mathcal A^s \cup \{\pm \infty \}$. Observe that all these maps are clearly
$h$-equivariant.

We insist on the fact that $\alpha_s(s)$ and $\beta_s(s)$ even if finite,
may or may not belong to $I^u(s)$. Actually:

\begin{lemma}
\label{le:boundaryleaf}
If $\alpha_s(s)$ is the projection of a point $z$ in $s \cap \Omega_P$, that is, $\alpha_s(s)$ is in
$I^u(s)$,  then the leaf $\alpha_s(s) = \oo^u(z)$ is singular.
Similarly for $\beta_s(s)$, and the stables leaves $\alpha_u(u)$, $\beta_u(u)$.
\end{lemma}

\begin{proof}
Suppose that $\oo^u(z)$ is not singular. Then we are in Case 1 of Lemma \ref{le:interval}, switching
stable and unstable, and we can take $p = z$. The conclusion of that Lemma is that $z$ is in
the interior of a segment $I$ contained in $\oos(z) \cap \Omega_P$,
 and again since $\oou(z)$ is
non singular this segment projects in
$\hhs$ to a neighborhood of $\oou(z)$ in $\hhu$. In particular
the segment also projects to
a neighborhood of $\oou(z)$ in $\mathcal A^u$.  Therefore $\oou(z)$ cannot be $\beta_s(s)$, contradiction.
\end{proof}

\begin{figure}
  \centering
   \includegraphics[scale=0.8]{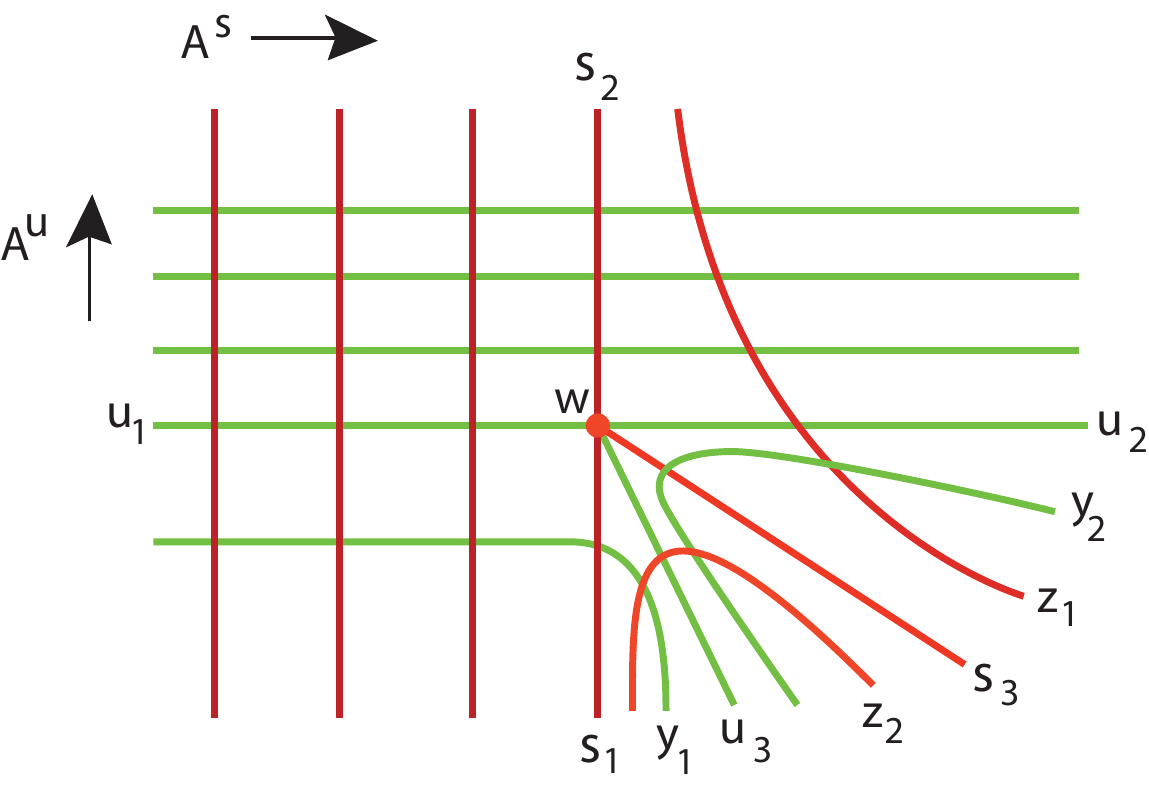}
\caption{Boundary points in $\Omega_P$.
The singular orbit $w$ is in $\Omega_P$. The figure
depicts the situation of a $3$-prong orbit.
The unstable prongs $u_1, u_2$ of $w$ are contained in
$\Omega_P$ and the unstable prong $u_3$ is not
in $\Omega_P$. Similarly the stable prongs
$s_1, s_2$ are contained in $\Omega_P$ and $s_3$ is
not. We also depict additional leaves: $y_1, y_2$ unstable
leaves, $y_1$ is in $\mathcal A^u$, $y_2$ is not
in $\mathcal A^u$; and $z_1, z_2$ stable leaves,
where $z_1$ is in $\mathcal A^s$ and $z_2$ is not
in $\mathcal A^s$.
It follows that the unstable prong $u_2$ is contained in
the boundary of $\Omega_P$ and so is the stable
prong $s_1$. Locally the boundary of $\Omega_P$ near
$w$ is made up of $s_1 \cup \{ w \} \cup u_2$.}
\label{axesomega}
\end{figure}

\begin{remark}{\em
The situation that $\alpha_s(s)$ is the projection of a point
$z$ in $s \cap \Omega_P$ is quite possible, and occurs if
there is a $p$-prong singularity in the ``boundary" of $P$,
in other words, the singularity is
in the union of the boundary Birkhoff annuli. An example of this
is the following: consider  a branched cover of the geodesic
flow in the unit tangent bundle $M^*$ of a hyperbolic
surface $S$ with a simple geodesic $\beta$ of symmetry,
so that $S - \{ \beta \}$ is the union of two surfaces
$S_1, S_2$. Then $T_1 S_1$ survives in the branched cover $M$,
and it generates a free piece $P$ of the branched pseudo-Anosov
flow that has a singularity in the boundary of $P$.
Lift this to $\widetilde M$. We choose $w$ a lift of the
corresponding singular orbit in the boundary of $P$,
so that $\oos(w)$ and $\oou(w)$ are in
$\mathcal A^s, \mathcal A^u$ respectively.
Since $w$ is singular,
Case 1 of Lemma \ref{le:interval}
implies that there are exactly two prongs $s_1, s_2$
of $\oos(w)$ at $w$
so that they are contained in $\Omega_P$.
As $w$ is singular there are other prongs of $\oos(w)$
at $p$ that are not contained
in $\Omega_P$. In the same way there are two unstable
prongs $u_1, u_2$ at $p$ contained in $\Omega_P$.
The four prongs $s_1,s_2,u_1,u_2$ have to be consecutive
around $w$, that is, after rearranging we may assume
that they go in the following order: $s_1,u_1,s_2,u_2$.
In other words $s_1, u_1$ bound a ``quadrant'' of the
stable and unstable foliations at $w$.
A {\em {quadrant}} is a component of the complement
of the union of $\oos(w)$ and $\oou(w)$. The
quadrant in question is contained in
$\Omega_P$. Similarly for $u_1,s_2$ and $s_2,u_2$.
There are exactly 3 quadrants at $w$ contained
in $\Omega_P$. If $w$ is a $p$-prong singular orbit,
then in total there are $2p$ quadrants at $w$, and
$2p-3$ open quadrants disjoint from $\Omega_P$.
It now follows that $u_1$ and $s_2$ are in the interior
of $\Omega_P$ and $s_1, u_2$ are in the boundary of
$\Omega_P$.
If $s_3$ is another stable prong at $w$ and $y_2$ is an
unstable leaf intersecting $s_3$ very near $w$, then
$y_2$ is {\underline {not}} in $\mathcal A^u$ and hence
the intersection $u \cap s_3$ is not in $\Omega_P$.
We refer to figure \ref{axesomega}.
}
\end{remark}

\begin{lemma}
\label{le:fini}
The functions $\alpha_s$, $\beta_{s}$ (respectively $\alpha_u$, $\beta_u$) take values in $\mathcal A^u$ (respectively $\mathcal A^s$).
\end{lemma}

\begin{proof}
The meaning is that these functions never attain the values $\pm \infty$.
We only prove the statement for $\alpha_s$, the proof for the other functions is similar.
Let $s_0$ be an element of $\mathcal A^s$ such that $\alpha_s(s_0) =-\infty$. Then for every integer $n$ we have $\alpha_s(h^n(s_0)) = -\infty$, because $h$ preserves the whole set $\mathcal A^u$.
Assume that there is an element $s$ in the interval $(s_0, h(s_0))$ such that $\alpha_s(s) $ is an element $u$ of $\mathcal A^u$.
Then, there is an element $u'$ of $\mathcal A^u$ strictly smaller
than all of $u$, $\beta_s(s_0)$, and
$\beta_s(h(s_0))$, since
$\alpha_s(s_0), \alpha_s(h(s_0))$ are both equal to $-\infty$.
According to Lemma \ref{le:piecewisepath} there is a stable/unstable path in $\Omega_P$
connecting $(s_0, u')$ to $(h(s_0), u')$. Since this path is contained in $\Omega_P$,
it must cross the set $I^u(s)$. Hence the path must contain a point
$(s, u'')$ with $u'' \geq u > u'$. Since the path starts at $(s_0, u')$ and finishes at $(h(s_0), u')$, the path
must cross, for every element $v$ of the interval $(u',u)$ in $\mathcal A^u$,
the $v$-level (that is, intersects the unstable
leaf $v$) twice. This means that
the path contains an element of the form $(s_-, v)$ with $s_0 \leq s_- < s$, and another element $(s_+, v)$ with
$s < s_+ \leq h(s_0)$. Hence, intersection between the unstable leaf $v$ and $\Omega_P$ admits at least two connected components:
one containing $s_- \cap v$, and the other containing $s_+ \cap v$. Observe that the
intersection cannot contain $s \cap v$ since $v < u = \alpha_s(s)$.
This contradicts Lemma \ref{le:segment}.

Therefore, $\alpha_s$ attains the value $-\infty$ on the entire segment $[s_0, h(s_0)]$.

There is a stable/unstable path from $(s_0,t_1)$ in $\Omega_P$ to $h(s_0,t_2)$
in $\Omega_P$ for some $t_1, t_2 \in \mathcal A^u$. By compactness
this path attains a mininum $t \in \mathcal A^u$ therefore $t$ intersects every leaf $s \in [s_0,h(s_0)]$
because $\alpha_s \equiv -\infty$ in this interval.
Now recall that $\mathcal A^u$ is properly embedded in $\hhu$. Therefore for every stable leaf $s \in
[s_0,h(s_0)]$, $s$ does not intersect any other unstable leaf in the negative direction.
This implies that the unstable segment in $t$ from $(s_0 \cap t)$ to $(h(s_0) \cap t)$ is the
base segment of a stable product region. By Theorem \ref{prod} it follows that $\Phi$ is
topologically conjugate to a suspension, contrary to hypothesis
that $\pi_1(P)$ is not elementary
in our study of free Seifert pieces $P$.
This finishes the proof of the Lemma.
\end{proof}

\begin{lemma}
\label{le:monotone}
The functions $\alpha_s$, $\beta_s: \mathcal A^s \rightarrow \mathcal A^u$ are both weakly monotone increasing.
\end{lemma}

\begin{proof}
Suppose by way of contradiction that there is $s_2 > s_1$ in $\mathcal A^s$ with $\alpha_s(s_2) < \alpha_s(s_1)$.
Let $k$ be the unique non negative integer such that $s_1 < h^{-k}(s_2) \leq h(s_1)$. Then

$$\alpha_s(h^{-k}(s_2)) = h^{-k}\alpha_s(s_2) < \alpha_s(s_2) < \alpha_s(s_1),$$

\noindent
hence we can assume without loss of generality that $k=0$, ie. that $s_1 < s_2 \leq h(s_1)$.
Then, if we put $s_3 = h^{-1}(s_2)$, we have:
$$\alpha_s(s_3) = h^{-1}\alpha_s(s_2) < \alpha_s(s_2) < \alpha_s(s_1)$$

The contradiction is then obtained as in the proof of Lemma \ref{le:fini}. There is a
unstable leaf $v$ in $\mathcal A^u$ such that $\alpha_s(s_3) < \alpha_s(s_2) < v < \alpha_s(s_1)$.
The leaf $v$
intersects two leaves $s_-, s_+$ with
$s_3 \leq s_-  < s_1$ and $s_2 \geq s_+  > s_1$, but $v$ does not intersect $s_1$.
Again this contradicts Lemma \ref{le:segment}.
\end{proof}


\begin{lemma}
\label{le:nosingleton}
For any $s$ in $\mathcal A^s$, we have $\alpha_s(s) < \beta_s(s)$.
\end{lemma}

\begin{proof}
Assume by contradiction that the equality $\alpha_s(s) = \beta_s(s)$ holds for some $s$ in $\mathcal A^s$.
This means that $s \cap \Omega_P$ is a single point $z$. Let $u$ be the unstable leaf through $z$.
Then $u = \alpha_s(s)$.
By Lemma \ref{le:boundaryleaf} $u$ is singular.
Let $p$ be the singular point in $u$. If $p = z$, then by Case 2.1 of Lemma \ref{le:interval}, with stable and unstable switched, the following
happens:
there are two stable prongs $J_1, J_2$ of $s = \oos(z) = \oos(p)$
at $p$ so that $J_1$ and $J_2$ project to $\mathcal A^u$.
In other words for any $t$ in $J_1 \cup J_2$ then $\oou(t) \in
\mathcal A^u$. Then $J_1 \cup \{ p \} \cup J_2$ is an
in $s$ entirely contained in $\Omega_P$, contradiction.

Now we consider the case that $p$ is not $z$.
Since $\oou(p) \in \mathcal A^u$ and $p$ is singular, then
as in the proof of Lemma \ref{le:interval}, case 2.2, there are exactly two
{\underline {stable}} prongs $I_1, I_2$ at $p$
which locally project into $\mathcal A^u$.
Let $v_1, v_2$ in these prongs very near $p$,
so $U_1 = \oou(v_1), U_2 = \oou(v_2)$ are in $\mathcal A^u$.
Notice that $s$ does not intersect either $U_1$ or $U_2$
and $s$ does not disconnect $U_1$ from $U_2$.
Since both $U_1$ and $U_2$ are in the same complementary
component of $s$, then $s$ separates either $h^{-1}(s)$
or $h(s)$ from both  $U_1$ and $U_2$.
Assume wlog the second option that is $s$ separates $h(s)$
from both $U_1$ and $U_2$.
Let $W$ be the component of $\oo - (U_1 \cup U_2)$ containing
$s$.
Replacing $U_1$ with $U_2$ if necessary we may assume
that $U_1 < U_2$ in $\mathcal A^u$. This implies that
$U_2$ separates $h(W)$ from $W$. In particular
since $s \subset W$ then $h(s)$ is separated from $s$ by
$U_2$. But we proved above that $s$ separates $h(s)$
from $U_2$, contradiction.
This finishes the proof of the Lemma.
\end{proof}


In summary, we have proved that $\Omega'$ is a region
very similar to the region
$\Omega$ studied in section \ref{sub:consorbit}:
whereas $\Omega$ is the region between two maps $\alpha_1$, $\beta_1$,
$\Omega'$ is the region
in $\mathcal A^s \times \mathcal A^u$ limited by the graphs of
two functions $\alpha_s$, $\beta_s$
which are weakly monotone, and which do not coincide anywhere.
$\Omega'$ may contain elements of
the graph of $\alpha_s$ or $\beta_s$.
Notice that it is not necessarily true that if $\alpha_s(s) < u < \beta_s(s)$ in
$\mathcal A^u$ then the point $(s,u)$ is in the interior of $\Omega_P$. For example it could be that
$(s,u)$ is in the interior of a nondegenerate segment $\{ s \} \times [\alpha_s(s), u_0]$
that is contained in the boundary of $\Omega'$.

With these properties one can prove that there are infinitely
many periodic orbits ``contained'' in the piece $P$ and uncountably
many full orbits ``contained'' in $P$, so the dynamics
of the flow in $P$ is extremely complicated. We do not
prove this separately as it follows immediately
from the Main theorem.

\subsection{The minimal set and fixed points}\label{sec:minimalfixed}

The group $\pi_1(P)$ acts on $\mathcal A^s$. The following result is standard:

\begin{lemma}{}{}
There is a unique minimal set $\tilde{\mu}_s \subset \mathcal A^s$ which is non empty
and $\pi_1(P)$ invariant. Since $\pi_1(P)$ is not virtually cyclic, the set
$\tilde{\mu}_s$ can either be the entire $\mathcal A^s$ or a Cantor set.
In particular $\tilde{\mu}_s$ is a perfect set.
Similarly, there exist a unique minimal non-empty and $\pi_1(P)$ invariant set $\tilde{\mu}_u \subset \mathcal A^u$.
\end{lemma}

\begin{lemma}{}{}\label{le:alphabetamu}
The restrictions to $\tilde{\mu}_s$  of $\alpha_s$ and $\beta_s$ are almost injective. More precisely: if $s$, $s'$ are two elements of
$\tilde{\mu}_s$ such that $\alpha_s(s) = \alpha_s(s')$ or $\beta_s(s)=\beta_s(s')$, then $s=s'$ or $]s, s'[$ is a connected component of $\mathcal A^s - \tilde{\mu}_s$.
Similarly, the same property holds for the restrictions of $\alpha_u$ and $\beta_u$ to $\tilde{\mu}_u$.
\end{lemma}

\begin{proof}
Let $U$ be the open subset of $\mathcal A^s$ comprising points where $\alpha_s$ (or $\beta_s$)  is locally constant. Then $\mathcal A^s - U$ is a closed invariant subset, hence contains
$\tilde{\mu}_s$. The Lemma follows.
\end{proof}

In the next section, we will improve Lemma \ref{le:alphabetamu} (see Corollary \ref{cor:alphabetamu2}).


Let $\gamma_0$ be a non-trivial element of $\pi_1(P)$ preserving an element $s_0$ of $\mathcal A^s$.
Replacing $\gamma_0$ by its square if necessary,
we can assume that it commutes with $h$.
Let $\tilde\theta_0$ be the unique fixed point of
 $\gamma_0$ in $s_0$, and let $u_0$ be the unstable leaf
$\oo^u(\tilde\theta_0)$.
As we will see $(s_0,u_0)$ may not belong to $\Omega'$,
i.e. $\tilde\theta_0$ may be outside $\Omega_P$.
However, since $s_0$ lies in $\mathcal A^s$,
and $u_0 = \oou(\tilde\theta_0)$ with $\tilde\theta_0$
periodic, it follows that
the leaf $u_0$ intersects a maximal
segment $]s_{-1}, s_1[$ of stable leaves in $\mathcal A^s$
which has the property of being $\gamma_0$-invariant, and
containing only one fixed point: the leaf $s_0$. Let $\tilde{\theta}_{-1}$ and $\tilde\theta_1$ be the $\gamma_0$-fixed points in respectively
$s_{-1}$ and $s_1$.

Notice that $\gamma_0 h(\tilde\theta_0) = h \gamma_0(\tilde\theta_0)
= h (\tilde\theta_0)$.
According to Theorem \ref{chain} there is a chain of
lozenges $C_1,$ $C_2$, ... $C_{k}$ such that $\tilde \theta_0$ is corner of $C_1$ and $h(\tilde \theta_0)$ a corner of $C_k$.
We extend it to
a bi-infinite chain of lozenges $\{ C_i \}_{(i \in \mathbb Z)}$ which is $H$-invariant where $H$ is the group generated by $\gamma_0$ and $h$:
for every integer $i$ we have $\gamma_0(C_i)=C_i$ and $h(C_i) = C_{i+k}$. But it is not clear that these lozenges, or their interiors, are contained
in $\Omega_P$.
We will prove this fact, and this will be done by
reconstructing the chain of lozenges in $\Omega_P$.

\begin{proposition}\label{pro:lozengeinomega}
The $\gamma_0$ fixed point $\tilde\theta_0$
is the corner of a lozenge $C'_0$ which is contained in
$\Omega_P$, except maybe
the corners. The preimage by $\flat$ of $C'_0$ with the corners removed
is the rectangle $]s_0, s_1[ \ \times \ ]u_0, \beta_s(s_0)[
\ \subset \Omega'$.
\end{proposition}

\begin{proof}
We distinguish two cases:

\textbf{Case 1 - $\tilde\theta_0$ is in $\Omega_P$}

In this case $u_0 = \oou(\tilde\theta_0)$ is
in $\mathcal A^u$, and
consequently , the segment $]s_{-1}, s_1[ \times \{ u_0 \}$
lies in $\Omega'$. Therefore, we have $\alpha_s(s) \leq u_0$ for every $s$ in $[s_0, s_1[$.
Moreover, since $\beta_s$ is non-decreasing, we also have $\beta_s(s) \geq \beta_s(s_0)$ for every $s$ in $[s_0, s_1[$. It follows
that the rectangle $]s_0, s_1[ \times ]u_0, \beta_s(s_0)[$
is contained in $\Omega'$.
Moreover, we have $\alpha_s(s_1) \geq u_0$ (if not, $(\alpha_s(s_1), u_0)$ would correspond to a second fixed point in the unstable leaf $u_0$).
In addition if $u_1 = \beta_s(s_0)$ then $\gamma_0(u_1) = u_1$
as $\gamma_0$ commutes with $\beta_s$ and $\gamma_0(s_0) = s_0$.

We then have two subcases:

-- \emph{either $\beta_s(s) > \beta_s(s_0)$ for some $s$ in $[s_0, s_1[$:} then it is true for every $s$ in $[s_0, s_1[$ since $\beta_s$ is non-decreasing and
commutes with $\gamma_0$. It means that the leaf $u_1 = \beta_s(s_0)$ intersects every $s$ in $]s_0, s_1[$. Let $\tilde\theta'_1$ be the unique $\gamma_0$
fixed point in $u_1$. The union of the
intersections $u_1 \cap s$ for $s \in \ ]s_0, s_1[$ \ is a component of
$u_1 - \{ \tilde{\theta}'_1 \}$ which makes a perfect fit with $s_0$.
Since $\tilde\theta'_1$ is a periodic
orbit, every unstable leaf close to $u_1$ intersects $\oo^s(\tilde{\theta}'_1)$. In particular, the intersections
between $\oo^s(\tilde{\theta}'_1)$ and leaves $u$ in $]u_0, u_1[$ describes a half leaf in $\oo^s(\tilde{\theta}'_1)$ which makes a perfect fit with
$s_0$. It follows that $\tilde\theta_0$ and $\tilde\theta'_1$ are the corners of a lozenge $C'_0$, such that $C'_0$ is the rectangle
$]s_0, s_1[ \times ]u_0, \beta_s(s_0)[$. Moreover, in this case we have $\tilde\theta'_1 = \tilde\theta_1 \in \Omega_P$: we have shown that
$\tilde\theta_1$ is the corner of a lozenge
entirely contained in $\Omega_P$, corners included and also
the side $]s_0,s_1[ \times u_1$ included
(see figure \ref{lozomega1}, this figure does not
illustrate all subcases: as explained in
Lemma \ref{le:boundaryleaf}, it could happen that one corner, if singular, lies on the graph of $\alpha_s$ or $\beta_s$).

\begin{figure}
  \centering
   \includegraphics[scale=0.5]{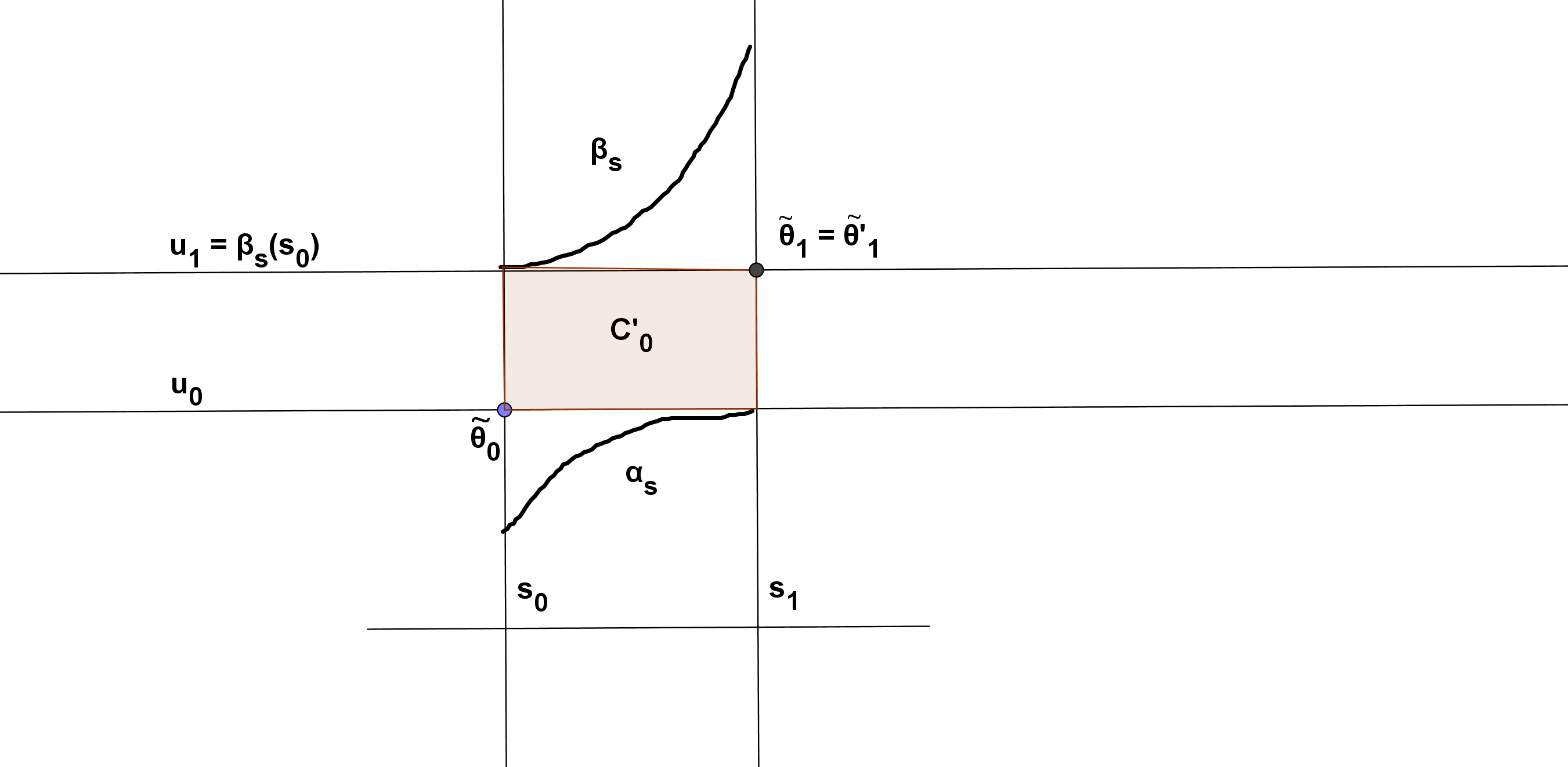}
\caption{A lozenge entirely in $\Omega_P$.}
\label{lozomega1}
\end{figure}

-- \emph{or for every $s$ in $[s_0, s_1[$ we have $\beta_s(s) = \beta_s(s_0)$:}
Since $\tilde\theta_0$ is the periodic orbit in $s_0$, then
for $s \in ]s_0,s_1[$ near $s_0$ it follows that
$s$ intersects $u_0$ and hence $\alpha_s(s) \leq u_0$.
Again by equivariance under $\gamma_0$ this is true for
all $s$ in the interval. Similarly $\beta_s \geq u_1$
throughout the interval.
As before $\beta_s(s_0)$ makes a perfect with $s_0$.
Similarly $u_0$ makes a perfect fit with $s_1$.
In other words the $\gamma_0$ invariant lozenge
$]s_0,s_1[ \times ]u_0, \beta_s(s_0)[$ is entirely
contained in $\Omega_P$.
The difference is that in this case $]s_0,s_1[ \times u_1$
is not in the interior of $\Omega_P$, but rather it is
in the boundary of $\Omega_P$. It may be contained
in $\Omega_P$ or not, and similarly for $\tilde\theta'_1$,
see figure \ref{lozomega2}.
In this case again $\tilde\theta'_1 =
\tilde\theta_1$.

\begin{figure}
  \centering
   \includegraphics[scale=0.8]{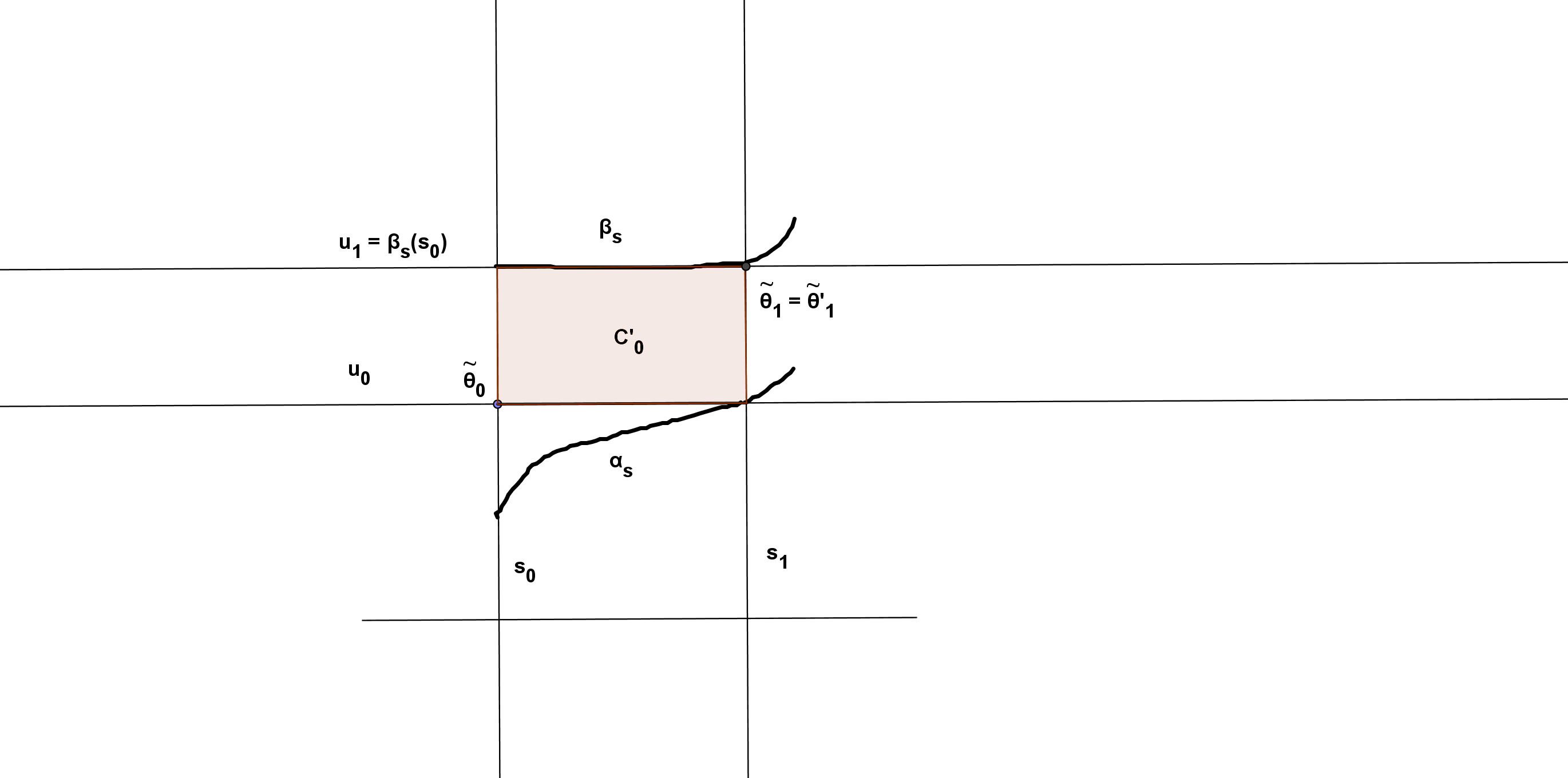}
\caption{A lozenge with a side not in the interior of
$\Omega_P$.}
\label{lozomega2}
\end{figure}

\textbf{Case 2 - $\tilde\theta_0$ is not in $\Omega_P$}

In this case, $u_0$ is not in $\mathcal A^u$.
Since $s$ intersects $\Omega_P$ in an open interval
and since it is $\gamma_0$ invariant then $s \cap \Omega_P$
is a component $D$ of $s - \tilde\theta_0$.
The projection of this component in the
{\underline {unstable}} leaf space is contained
in $\mathcal A^u$ and it is exactly the
segment $]\alpha_s(s_0), \beta_s(s_0)[$. In other words
The segment $]\alpha_s(s_0), \beta_s(s_0)[$ corresponds to the unique component of $\oo^s(\tilde\theta_0) - \{ \tilde\theta_0 \}$
intersecting $\mathcal A^u$.
One of the extremities, $\alpha_s(s_0)$ or $\beta_s(s_0)$, is a leaf which is not separated from $u_0$. Let us assume that it is $\beta_s(s_0)$, the other case
can be treated in a similar way.
Then, again because $\tilde\theta_0$ is periodic,
leaves in $\mathcal A^s$ close to $s_0$ all intersect $u_0$, which is not in $\mathcal A^u$. It follows that these
leaves are elements $s$ of $\mathcal A^s$
satisfying $\beta_s(s) = \beta_s(s_0)$. But this set of leaves is $\gamma_0$-invariant,
hence it is the entire $]s_{-1}, s_1[$.

Recall that $\tilde\theta_1$ is the periodic orbit in $s_1$.
The leaves $u_0$ and $s_1$ makes a perfect fit:
it follows that there is a component $b_1$ of
$u'_1 - \{ \tilde\theta_1 \}$ (where $u'_1 = \oo^u(\tilde\theta_1)$) such that every leaf in $]s_0, s_1[$ intersects $b_1$ (see figure \ref{lozomega33}) - observe that $u'_1$ may still be outside
$\mathcal A^u$.

\begin{figure}
  \centering
   \includegraphics[scale=0.8]{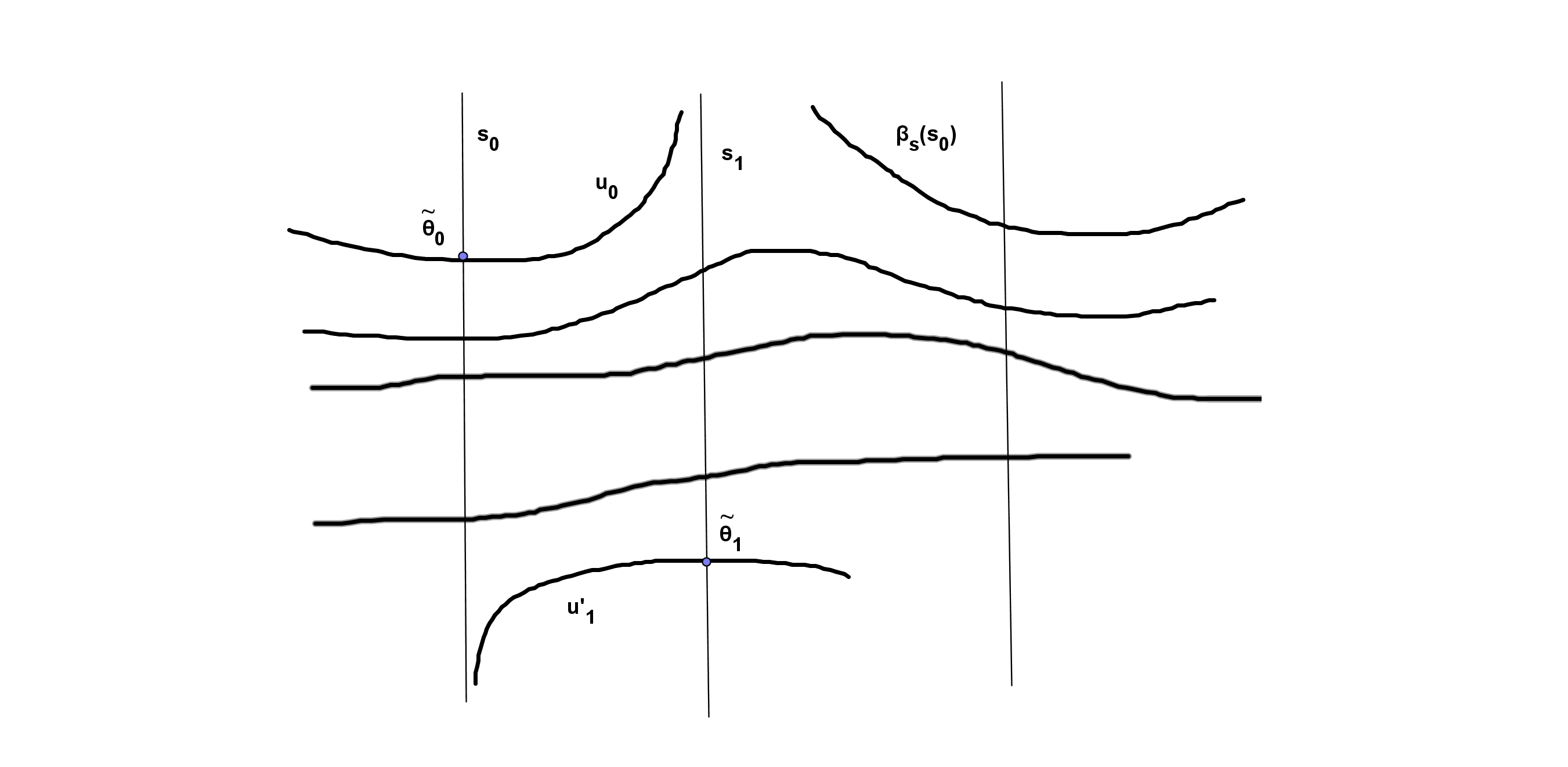}
\caption{The picture in $\oo$ for case $2$.}
\label{lozomega33}
\end{figure}

Again $u'_1$ does not intersect $s_0$ and $b_1$
makes a perfect with $s_0$.
It follows that we have $u'_1 = \alpha_s(s_0)$. The lozenge $C'_0$ with corners $\tilde\theta_0$ and $\tilde{\theta}_1$ is then
a lozenge with interior contained in $\Omega_P$, corresponding to the rectangle $]s_0, s_1[ \times ]\alpha_s(s_0), \beta_s(s_1)[$.
\end{proof}

Observe that it is still not totally
clear from this proof that $C'_0$ coincides with the first lozenge $C_0$ of the chain connecting $\tilde\theta_0$ to $h(\tilde\theta_0)$. But it is easy to infer this statement: indeed, apply Proposition \ref{pro:lozengeinomega} to $s_1$: we get another lozenge at the right
of $s_1$, that is, intersecting
$s$ with $s > s_1$ in $\mathcal A^s$. If we iterate this procedure, since there are only finitely many $\gamma_0$-fixed points in $\mathcal A^s$
between $s_0$ and $h(s_0)$, we finally reach
the stable leaf $h(s_0)$. It follows that the sequence of lozenges obtained in this way must be the unique chain of lozenges connecting $\tilde\theta_0$ to $h(\tilde\theta_0)$.
In figure \ref{toromega}, we have drawn what a possible
example of this sequence of lozenges
in $\mathcal A^s \times \mathcal A^u$. We have drawn
in red the  corners of the rectangles that cannot be corners of the lozenges (they are attracting or repelling fixed points for the action
of $\gamma_0$ on $\mathcal A^s \times \mathcal A^u$), and in blue points that \emph{may} be corners.
In other words, the blue corners correspond to actual orbits
of the flow in $\Omega'$. The action of $\gamma_0$ is attracting
on one of $\mathcal A^s, \mathcal A^u$ and repelling on the
other one.
The red corners correspond to perfect fits
between stable/unstable leaves and they do not correspond
to any orbit of $\widetilde{\Phi}$. In $\mathcal A^s$,
$\mathcal A^u$ the actions at the corresponding
points are either both attracting or both repelling.

\begin{figure}
  \centering
   \includegraphics[scale=0.8]{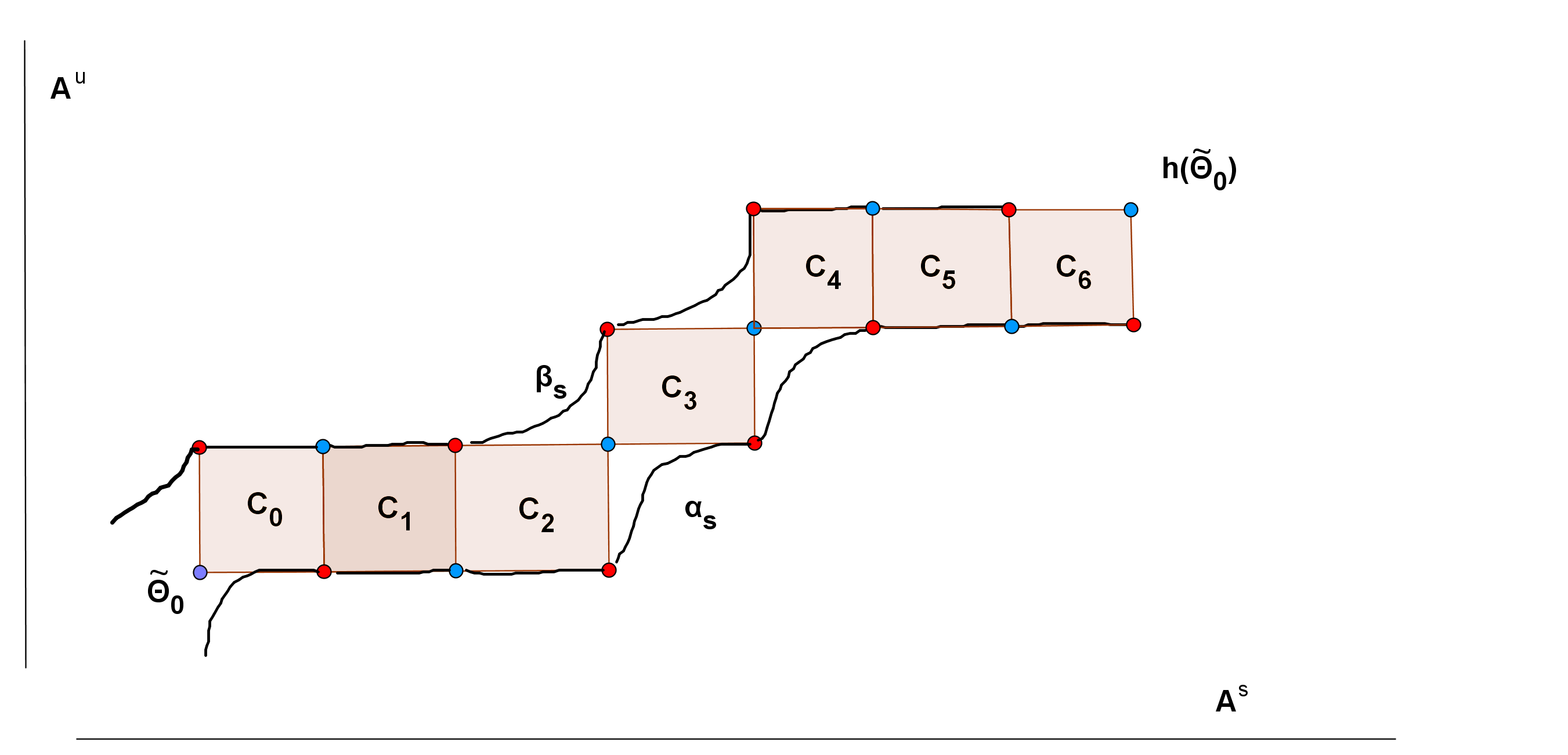}
\caption{A chain of lozenges in $\Omega'$ between the graphs of $\alpha_s$ and $\beta_s$.}
\label{toromega}
\end{figure}

As we can see in this picture, this chain of lozenges is a subdivision of a chain of rectangles, where each rectangle is a maximal
union of s-adjacent or
u-adjacent lozenges. Moreover, it follows from the analysis above that (blue) corners between two successive maximal
rectangles (for example, the corner common to
$C_2$ and $C_3$, or the one common to $C_3$ and $C_4$) cannot be as in case $2$ of the proof of Proposition \ref{pro:lozengeinomega}: they are points
in $\Omega'$ corresponding to the common corner of the lozenges in question. We call them \emph{true blue corners}. The other blue corners (for example,
the one common to $C_0$ and $C_1$, are \emph{fake blue corners}: they are precisely the ones appearing in case $2$ of the proof of Proposition \ref{pro:lozengeinomega}. They are never points in $\Omega'$.
More precisely:

\vskip .08in
\noindent
{\bf {Type S and type N for fixed points
in $\mathcal A^s$.}} \
There are two types of fixed points of $\gamma_0$ in $\mathcal A^s$: the stable leaves $s_i$ containing their blue point (which is then a true
blue corner), and the others, for which $s_i \cap \Omega'$ contains no $\gamma_0$-fixed point. We call the first type of type S, and the second type of
type N - this terminology comes from the following observation left to the reader:
let $x_i$ be the fixed point of $\gamma_0$ in $s_i$.
If $s_i$ is of type S, then $\oou(x_i)$ Separates
$x_{i_1}$ from $x_{i+1}$ in $\oo$.
in $s_{i+1}$ from the the $\gamma_0$-fixed point in $s_{i-1}$; and
If $s_i$ is  of type N, then $\oou(x_i)$ does not
separate $x_{i-1}$ from $x_{i+1}$ in $\oo$.
\vskip .04in

\begin{lemma}{}{}
If $s$ is of type N then $s$ is not in $\tilde{\mu}_s$.
\end{lemma}

\begin{proof}{}
Observe that if $s$ is of type N, then $\alpha_s$ or $\beta_s$ is constant in a neighborhood of $s$. Now the lemma follows from Lemma \ref{le:alphabetamu}.
\end{proof}

Conversely:

\begin{lemma}{}{}
If $s$ is of type S then $s$ is in $\tilde{\mu}_s$.
\end{lemma}

\begin{proof}{}
The proof is by contradiction. Suppose that $s$ is not in $\tilde{\mu}_s$. Let
$J$ be the complementary interval of $\tilde{\mu}_s$ in $\mathcal A^s$ containing $s$.
Observe that $J$ is $\gamma_0$-invariant. Its endpoints $s_-,$ $s_+$ with $s_- < s < s_+$ are also
$\gamma_0$-fixed points.
By the previous lemma $s_-$ and $s_+$ are of type S since they are in $\tilde{\mu}_s$.
Each of $s_-, s_+$ contains
 a $\gamma_0$-fixed point $x\pm = (s_\pm, u_\pm)$ in $\Omega$ which is a true blue corner, and there is
also a fixed point $(s,u)$ as well.
Since $\tilde{\mu}_s$ is perfect and since the $\pi_1(P)$-orbit of $s_-$ is dense in $\tilde{\mu}_s$ there
is an iterate $\gamma(s_-)$ that intersects the
unstable leaf $u_- = \oo^u(x_-)$.
The leaf $u_-$ is in $\mathcal A^u$, because
$\oou(u_-)$ separates two stable leaves
in $\mathcal A^s$ from each other since $s_-$ is of
type S and is associated with a true blue corner.
Again, because $x_-$ is the periodic point in $s_-$ it
follows that $\gamma(s_-)$ intersects elements
in $\mathcal A^u$ that are bigger than $u_-$.
In other words, $\beta_s(\gamma(s_-)) > u_-$.
But since $\beta_s$ is non decreasing, we have

$$u_- \ < \ \beta_s(\gamma(s_-)) \ \leq \ \beta_s(s_-) \ \leq \ u.$$

Observe that one can chose $\gamma$ so that in addition
$\gamma s_-$ is not fixed by $\gamma_0$: it is a fixed point of an
element $\gamma'_0 = \gamma\gamma_0\gamma^{-1}$ which has no common
fixed points with $\gamma_0$ (indeed, observe that
if two elements of $\pi_1(P)$  as above
admit a common fixed point in $\mathcal A^s$,
then either they are both orientation preserving
or orientation reversing on $\mathcal A^s$ and it follows that
they have exactly the same set of fixed points in $\mathcal A^s$).

The unstable leaf $\beta_s(\gamma(s_-))$ is fixed by $\gamma'_0$,
hence its image by $\beta_u$ is a $\gamma'_0$ fixed point too.
But since $\beta_u$ is non-decreasing, we have

$$\beta_u(u_-) \ \leq \ \beta_u(\beta_s(\gamma s_-)) \ \leq \ \beta_u(u).
\ \ \ {\rm But \ also} \ \ \
s_- \ \leq \ \beta_u(u_-), \ \ {\rm and} \ \
\beta_u(u) \ \leq s_+$$

\noindent
(see figure \ref{typeSgeo}). As a consequence
 we obtain that $[s_-, s_+]$ contains the
$\gamma'_0$-fixed point $\beta_u(\beta_s(\gamma(s_-)))$.

\begin{figure}
  \centering
   \includegraphics[scale=0.8]{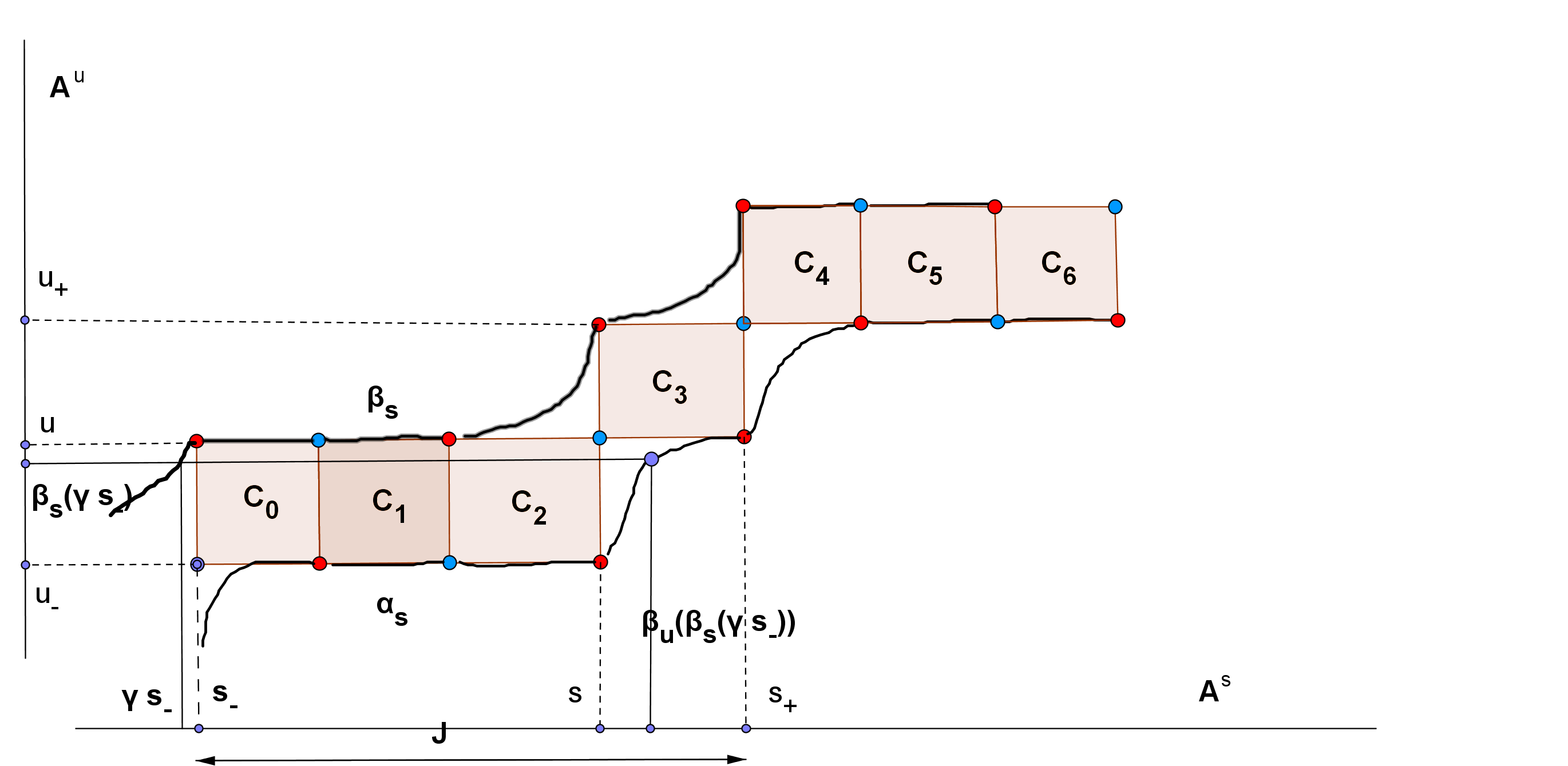}
\caption{The gap $J$ of $\tilde{\mu}_s$ contains a $\gamma'_0$-fixed point.}
\label{typeSgeo}
\end{figure}

But $[s_-, s_+]$ is the closure of the gap $J$ of $\tilde{\mu}_s$: it follows
that $J$ is preserved by $\gamma'_0$. Its endpoints $s_\pm$ are fixed by $\gamma_0$ and $\gamma'_0$: contradiction.
This proves the Lemma.
\end{proof}

We stress the following important property. The arguments above
show that $\tilde\theta$ is a true blue corner if and only
if the following happens: the two consecutive lozenges in the
chain that share $\tilde\theta$ as a corner are
{\underline {not adjacent}}. Otherwise these lozenges
either intersect common stable or unstable leaves.

\begin{corollary}{}{}
Let $\tilde\theta \in \oo$ be a fixed point of a non-trivial element
$\gamma_0$ of $\pi_1(P)$. Then, $\tilde\theta$ lies in $\Omega_P$ if and
only if $\oo^s(\tilde\theta) \in \tilde{\mu}_s$ or $\oo^u(\tilde\theta) \in \tilde{\mu}_u$. Moreover, if it happens, then the preimage
$\flat^{-1}(\tilde\theta) = (\oo^s(\tilde\theta), \oo^u(\tilde\theta))$ lies in $\tilde{\mu}_s \times \tilde{\mu}_u$. In particular,
$\tilde{\mu}_s$ is the closure of the subset of $\mathcal A^s \subset \hhs$ that consists
of $\oos(\tilde\theta) \in \mathcal A^s$, where $\tilde\theta$ in $\Omega_P$  is fixed by $\gamma$ in $\pi_1(P) - id$ and $\oos(\tilde\theta)$ is of type S.
\label{minimal}
\end{corollary}

\begin{proof}{}
Let $\tilde\theta \in \oo$ with $\gamma(\tilde\theta)
= \tilde\theta$ and $\gamma$ in $\pi_1(P) - id$.
By the previous Lemmas $\oos(\tilde\theta)$ is in
$\tilde\mu_s$ if and only if $\oos(\tilde\theta)$ is
of type S.
Hence $\tilde\theta$ is a true blue corner
and as shown in the proof of the previous Lemma,
$\oou(\tilde\theta)$ is in $\mathcal A^u$.
This implies that $\tilde\theta$ is in $\Omega_P$.
In addition $\oou(\tilde\theta)$ is invariant
under $\gamma$ and it corresponds to a true blue corner
when looking at the action on $\mathcal A^u$,
so $\oou(\tilde\theta) \in \tilde\mu_u$.

Conversely if $\tilde\theta$ lies in $\Omega_P$ then
$\oos(\tilde\theta)$ and $\oou(\tilde\theta)$
are in $\mathcal A^s$ and $\mathcal A^u$ respectively.
Since $\oos(\tilde\theta)$ separates
$\oos(h(\tilde\theta))$ from $\oos(h^{-1}(\tilde\theta))$
and similarly for $\oou(\tilde\theta)$. It follows
that $\oos(\tilde\theta)$ is in $\tilde\mu_s$ and
$\oou(\tilde\theta)$ is in $\tilde\mu_u$.
This proves the first two statements of the Corollary.

To prove the last statement all we need to do is to
show that there is $\oos(\tilde\theta)$ in $\tilde\mu_s$
fixed by $\gamma$ in $\pi_1(P) - id$.
To prove that, recall that $\pi_1(P)$ is not abelian,
and therefore there is $\gamma \in \pi_1(P) -id$
with a fixed point $\oos(\tilde\theta)$ in
$\mathcal A^s$. If $\oos(\tilde\theta)$ is in $\tilde\mu_s$
then we are done. Otherwise $\oos(\tilde\theta)$ is in a
complementary component $J$ of $\tilde\mu_s$.
In that case $\gamma^2$ fixes the endpoints of $J$
that are in $\tilde\mu_s$.

This shows that the set $B$ that is the closure of
the fixed points of $\gamma \in \pi_1(P) - id$ of
type S is a non empty set. It is closed and
$\pi_1(P)$ invariant. Since any fixed point of type
S is in $\tilde\mu_s$ then $B$ is a subset of $\tilde\mu_s$.
The minimality of $\tilde\mu_s$ shows that
$B = \tilde\mu_s$.
This finishes the proof of the Corollary.
\end{proof}

In the sequel, we will need to consider the maps $\alpha^-_s$, $\beta_s^+$, $\alpha_u^-$, $\beta^+_u$, analogous
to the maps $\alpha_i^-$, $\beta_i^+$ with $i=1,2$ introduced in section \ref{sub:consorbit}. They are defined as follows:

\begin{eqnarray*}
  \alpha_s^-(s) &=& \lim_{\epsilon \to 0} \alpha_s(s - |\epsilon|) \\
  \alpha_u^-(u) &=& \lim_{\epsilon \to 0} \alpha_u(u - |\epsilon|) \\
  \beta_s^+(s) &=& \lim_{\epsilon \to 0} \beta_s(s + |\epsilon|) \\
  \beta_u^+(u) &=& \lim_{\epsilon \to 0} \beta_u(u + |\epsilon|)
\end{eqnarray*}

The maps $\alpha_{s,u}^-$ are continuous on the left, and $\beta_{s,u}^+$ are continuous on the right. For every $s$ in $\mathcal A^s$ we clearly have,
$\alpha_s^-(s) \leq \alpha_s(s)$ and $\beta_s(s) \leq \beta_s^+(s)$. Similarly, for every $u$ in $\mathcal A^u$ we have
$\alpha_u^-(u) \leq \alpha_u(u)$ and $\beta_u(s) \leq \beta_u^+(u)$.

The proof of the following two lemmas is easy,
with the help of figure \ref{typeSgeo}):

\begin{lemma}\label{le:subeta}
Let $(s,u)$ be a fixed point in $\tilde{\mu}_s \times \tilde{\mu}_u$ of an element $\gamma_0$ of $\pi_1(P) - id$.
Then, $(s,u)$ is the corner of an open  rectangle $R$ in $\oo$,
whose other corner is $(\beta_u^+(u), \beta_s^+(s))$ .
Observe that $\beta_u^+(u)$ is the lowest element $s'$
of $\tilde{\mu}_s \cap Fix(\gamma_0)$ satisfying $s < s'$.
\end{lemma}

\begin{proof}{}
It follows that $(s,u)$ is in $\Omega_P$ and it is a
true blue corner, and so is $(\beta^+_u(u),\beta^+_s(s))$.
As in figure \ref{typeSgeo} these two corners
as connected by a chain of adjacent lozenges.
We let $R$ to be this union of lozenges, plus the
sides common to adjacent lozenges.
By the definition of $\beta^+_u, \beta^+_-$, the
rectangle $R$ is the {\underline {maximal}} open rectangle
made of adjacent lozenges (all intersecting a common
stable or unstable leaf) containing the initial lozenge
with corner $(s,u)$ and with stable leaves $> s$ in
$\mathcal A^s$.
Suppose that these are the lozenges $C_1,...,C_i$ in the
chain. Notice that the next lozenge in the chain, $C_{i+1}$
cannot be adjacent with $C_i$. If $i = 1$ this is not
possible for then $C_1 \cup C_2$ would be adjacent
and $C_1$ would not be maximal.
If $i > 1$ suppose that $C_1,...,C_i$ all intersect
common unstable leaves and $C_i, C_{i+1}$ intersect
common stable leaves. Then $C_{i-1}, C_i, C_{i+1}$
all share a common corner and $C_i$ could be eliminated
from this chain of lozenges.
This shows that both  $\beta^+_u(u)$ and
$\beta^+_s(s)$ both separate consecutive lozenges
in the chain, and therefore by the above remark,
both are in $\tilde\mu_s$ and $\tilde\mu_u$
respectively.

In addition by this description anyother stable leaf
$s'$ fixed by $\gamma_0$ and between $s$ and $\beta^+_u(u)$
is a stable leaf in between two {\underline {adjacent}}
lozenges in the chain. It follows that $s'$ is not in
$\tilde\mu_s$.
\end{proof}


\begin{lemma}\label{le:propertyalphabeta}
For every $s$ in $\tilde\mu_s$ and every $u$ in
$\tilde\mu_u$ so that: $s$ is fixed by $\gamma_0 - id$
and is of type S and $u$ is also fixed by $\gamma_0$  :
\begin{eqnarray*}
\alpha_s^-(\beta_u^+(u)) & = & u \\
\alpha_u^-(\beta_s^+(s)) & = & s
\end{eqnarray*}
\label{composition}
\end{lemma}

\begin{proof}{}
For every $(s,u)$ in $\tilde\mu_s \times \tilde\mu_u$
then $(v,w) = (\beta^+_u(u),\beta^+_s(s))$ is the other
corner of the maximal rectangle $R$ as in the previous
Lemma. Starting from the corner
$(v,w)$ in $\tilde\mu_s \times \tilde\mu_u$ and
going in the {\underline {negative}} direction
there is also a maximal rectangle (union of adjacent
lozenges) and that is exactly the same rectangle
$R$. The other corners are obtained
by the maps $\alpha^-_u$ and $\alpha^-_+$
and they recover the point $(s,u)$.
\end{proof}

We define the maps $\tilde{\tau}_s := \beta_u^+ \circ \beta_s^+$
from $\tilde\mu_s$ to $\tilde\mu_s$ and
$\tilde{\tau}_u := \beta_s^+ \circ \beta_u^+$
from $\tilde\mu_u$ to $\tilde\mu_u$. From the discussion above, when $s$ is a fixed point of $\gamma_0$ in $\tilde{\mu}_s$, then
$\tilde{\tau}_s(s)$ is the second $\gamma_0$-fixed point in $\tilde{\mu}_s$ after $s$.
Clearly $\beta^+_u$ sends a fixed point of type
S of $\gamma \in \pi_1(P) - id$ to
a fixed point of $\gamma_0$ acting on $\mathcal A^u$
of the analogous type S property.
Hence restricted to the fixed points over all
$\gamma \in \pi_1(P) - id$ of type S the map
$\beta^+_u$ is injective and clearly order preserving.
We later show that all functions $\alpha^-_s,\alpha^-_u,
\beta^+_s,\beta^+_u$ are strictly increasing
when restricted to the appropriate $\tilde\mu_s$ or
$\tilde\mu_u$. In particular they are order preserving
homeomorphisms.


\section{The lifted flow in the intermediate cover}

In this section,
we produce a good model for the free Seifert piece
by lifting to a appropriate cover,
and show that the stable
and unstable foliations restrict to non singular foliations in this piece.

\subsection{Putting the boundary tori in good position in the intermediate cover}

Let $M_P$ be the cover of $M$ associated with
$\pi_1(P)$.
Notice that $M_P$ is not compact.
Let $\hat \Phi$ be the lift of $\Phi$
to $M_P$.
Choose an embedded representative
for $P$ in $M$ bounded by finitely many incompressible tori ${\mathcal T} = \{ T'_1, ..., T'_{i_0} \}$.
Each torus $T'_i$ is homotopic to an a priori only immersed Birkhoff torus $B_i$.

\begin{proposition}{}{}
Each torus $T'_i$ is freely homotopic to a
weakly embedded Birkhoff torus $T^*_i$ so that the collection
of such tori satisfies:

$-$ $T^*_i$ lifts to an {\underline {embedded}} Birkhoff torus $T_i$ in $M_P$.

$-$ The collection $\{ T_1, ..., T_{i_0} \}$ can be chosen to be pairwise disjoint and they
bound a compact submanifold $\hat P \subset M_P$.

$-$ $\pi_1(\hatp)$ is isomorphic to $\pi_1(M_P) = \pi_1(P)$. The closure
of each component of $M_P - \hatp$ is homeomorphic to $T^2 \times [0,\infty)$.
\label{goodform}
\end{proposition}

\begin{proof}{}
In \cite{bafe1} we proved that each $T'_i$ is freely homotopic to a weakly embedded Birkhoff torus $T^*_i$.
We will adjust $T^*_i$ as needed. In \cite{bafe1}, Theorem 6.10 we proved that we can
choose $T^*_i$ to be embedded unless one of the following happens:

\begin{itemize}

\item 1) $T'_i$ is isotopic to the boundary of a regular neighborhood $V$ of an embedded
Birkhoff-Klein bottle $K$. The neighborhood $V$ is contained in a free Seifert piece $P_1$.

\item 2) $T'_i$ is homotopic to a weakly embedded Birkhoff torus contained in
a periodic Seifert fibered piece $P_1$.

\end{itemize}

Consider first possibility 1). Here $\partial V = T'_i$ and $V$ is contained in a Seifert piece
$P_1$ and $T'_i$ is a torus in the JSJ decomposition of $M$.
Since $T'_i$ is also in the boundary of a piece
of the JSJ decomposition,
it follows that
$V = P_1$. Then $\pi_1(P_1) = \pi_1(K)$. As seen in the proof of Proposition
\ref{axesfirst} this
implies that $P_1$ is not free. This contradiction shows that this case cannot happen.

Consider now possibility 2). In this case, it is quite possible that $T^*_i$ is
\underline{not} embedded in $M$. What we will show is that we can choose $T^*_i$ so that
$T_i$ is embedded in $M_P$.
Here the Seifert piece $P_1$ is periodic and hence $P, P_1$ are distinct Seifert pieces
and hence $T'_i$ is a torus in the boundary of both $P_1$ and $P$.
 As described in \cite{bafe1}, $T^*_i$ is
weakly embedded, so if $T^*_i$ is not embedded, it is because of the periodic orbits
in $T^*_i$: either a periodic orbit in $T^*_i$ is traversed multiple times as a loop
in $T^*_i$ or two periodic orbits collapse. In the first case there is an element $g \in
\pi_1(M)$
associated with the periodic orbit so that $g$ is not in $\pi_1(T'_i)$, but for some
$n > 1$, $g^n$ is in $\pi_1(T'_i)$. Here $g$ is in $\pi_1(P_1)$ because in $P_1$
the orbit is represented by a closed orbit.
The other option is that two periodic orbits collapse. Here there is an element $g$ in
$\pi_1(P_1)$  that is not in $\pi_1(T'_i)$ and which corresponds to the
identification of orbits in the Birkhoff torus.

Suppose that this problem persists when lifting $T^*_i$ to $M_P$.
By Scott's core theorem \cite{Sco1,He}, there is a compact
core in $M_P$ that carries all the homotopy of $M_P$.
The submanifold $P$ lifts homeomorphically to $M_P$.
The submanifold $P_1$ lifts to a non compact submanifold
$P'_1$
of $M_P$, but the component of the
intersection of $P_1$ and $P$ corresponding to $T'_i$
lifts homeomorphically to $M_P$.
Hence we can apply the same analysis of the last
paragraph to $M_P$.
Then there is a $g$ as
above, but now this $g$ is in $\pi_1(M_P) = \pi_1(P)$. It follows that this element
$g$ is in both $\pi_1(P)$ and $\pi_1(P'_1)$.
But this intersection is only
$\pi_1(T'_i)$. This is a contradiction and it shows that we can choose $T_i$
to be embedded in $M_P$.

\vskip .1in
The next step is to analyse whether the collection $\{ T_i \}$ can be chosen pairwise disjoint.
Suppose that $T_i \cap T_j \not = \emptyset$ for some $i \not = j$. Suppose first that
a closed orbit in $T_i$ intersects a closed orbit in $T_j$. Then they are the same closed orbit
of $\hatflo$. This produces a closed curve in $T'_i$ which is freely homotopic to a closed
curve in $T'_j$. This produces an essential annulus in $P$, and this annulus is
isotopic to a vertical annulus.
It follows that the periodic orbit in $T'_i$ is freely homotopic to a regular fiber in
$P$ up to powers. But then $P$ would be a periodic Seifert piece, contradiction
to the assumption that $P$ is free.

\vskip .05in
Suppose now that a closed orbit $\gamma$ in $T_i$ intersects the interior of a Birkhoff annulus
in $T_j$ in $M_P$. Lift this to the cover $M^*$ of $M_P$ associated to $\pi_1(T_j)$.
Notice that in this proof we use several covers $\mi \rightarrow M^* \rightarrow M_P
\rightarrow M$. The torus $T_j$ lifts
homeomorphically to an embedded torus in $M^*$, which is still denoted by $T_j$ for
simplicity.  We use the fact that if $B$ is a lift to $\mi$ of a Birkhoff
torus, then an orbit $\tilde \theta$ of $\wwp$ which intersects $B$ transversely, intersects
$B$ only once and the components of $\tilde \theta - B$ are in different components
of $\mi - B$.
This follows from the description of lifts of Birkhoff annuli and associated lozenges
in $\mi$.
The closed orbit $\gamma$ in $M_P$  lifts to a curve $\gamma_0$ in $M^*$  intersecting
$T_j$ transversely.
If $\gamma_0$ is closed then
it represents a power of $\gamma$ in $\pi_1(M)$,
hence this power of $\gamma$
is freely homotopic into $T_j$ in $M_P$. This was disallowed in the previous paragraph.
Hence $\gamma_0$ is not compact. Suppose that $\gamma_)$ intersects $T_j$ more
than once in $M^*$. Then since $\pi_1(M^*) = \pi_1(T_j)$, the
corresponding lift of $\gamma_0$ to $\mi$
is an orbit intersecting the universal cover of $T_j$ more than once, contradiction.
Hence there are two rays of $\gamma_0 - T_j$ and they are in different
components of $M^* - T_j$.
Let $\delta$ be one such ray. If $\delta$ accumulates
in a point of $M^*$, then because $\gamma$ is compact this would imply that $\delta$
is actually a closed curve, as $M^* \rightarrow M$ is a cover. This is a contradiction.
Therefore $d_{M^*}(p, T_j)$ goes to infinity as $p$ escapes in the ray $\delta$.

On the other hand, projecting to $M$ we see that
$\pi(\gamma_0)$ is homotopic to a curve disjoint
from $\pi(T_j)$, since $\pi(\gamma_0)$ is homotopic
into $T'_i$ and $\pi(T_j)$ is homotopic to $T'_j$.
Lifting this free homotopy
to $M^*$ we see that $\gamma_0$ is homotoped by a homotopy
moving points a bounded distance
to be disjoint from $T_j$. Hence there is a ray of $\gamma_0 - T_j$
which is a bounded distance from $T_j$. As seen above, this is a contradiction and it
shows this situation cannot happen.

This shows that any closed orbit in $T_i$ cannot intersect $T_j$. Finally we check what
happens if the interior of a Birkhoff annulus in $T_i$ intersects the interior of a Birkhoff
annulus in $T_j$. Put $T_i, T_j$ in general position. The analysis now follows standard
arguments. First eliminate null homotopic intersections. Using innermost arguments
one gets a disk in $T_i$ and another in $T_j$ which jointly produce an embedded
sphere bounding a ball in $M_P$ - as $M_P$ is irreducible. This intersection
can be eliminated by sliding $T_j$ across the ball. The remaining intersections
are freely homotopic to closed orbits of the flow (up to powers). By the first part of
this analysis, this cannot happen.

We conclude that we can choose the collection $\{ T_1, ..., T_{i_0} \}$ to be pairwise
disjoint. Jointly they bound a compact submanifold, which is
denoted by $\hatp$ in $M_P$.
The submanifold $\hatp$ carries all the homotopy of $M_P$.
Since $\mi \cong \rrrr^3$ and $M$ is Haken one can use the
general theory of Haken manifolds and compact cores
\cite{Sco1,He},
and it follows that
the closure of the components of $M_P - \hatp$ are homeomorphic to $T^2 \times [0,\infty)$.
This finishes the proof of Proposition \ref{goodform}.
\end{proof}

\subsection{Analysing the free Seifert piece in the intermediate cover}$\\$
\label{sub:intercover}

\vskip .01in
\noindent
{\bf {Notation}} $-$ In the cover $M_P$ we have the compact submanifold $\hatp$ with
boundary a union of pairwise disjoint Birkhoff tori $\{ T_1, ..., T_{i_0} \}$.
Let $\Lambda^s_P, \Lambda^u_P$
be the lifts of $\ls, \lu$ to $M_P$.
Let $\Phi_P$ be the lift of $\Phi$ to $M_P$.

\vskip .1in
\noindent
{\bf {Preparatory step}} $-$ We will adjust the boundary of $\hatp$ slightly near some of the
boundary periodic orbits in order to have $\Lambda^s_P, \Lambda^u_P$ to be
non singular foliations when restricted to $\hatp$.
Consider a Birkhoff torus $T$ in the boundary of $\hatp$. Consider a periodic orbit $\theta$
in $T$ and lifts $\widetilde T$ and $\widetilde \theta$ to $\mi$. Let $\mathcal C$
be the chain of lozenges invariant under $\pi_1(T)$ associated with the Birkhoff representative
$T$. If the two consecutive lozenges of $\mathcal C$ at $\widetilde \theta$ are not adjacent we do not
make any adjustment to $T$ near $\theta$. Suppose now that these lozenges $C_1, C_2$
are adjacent along $\wls(\widetilde \theta)$. We adjust $T$ as follows.
Let $L$ be  the half leaf of $\wls(\widetilde \theta)$ in the boundary of
both $C_1$ and $C_2$. Move $\widetilde T$ away from $\widetilde \theta$ and
slightly into $L$. We can make this adjustment in $M$. Now $T$ is not tangent to $\Phi$
near $\theta$, but is actually transverse to $\Phi$ near $\theta$. That is, we
eliminated one tangency of $\Phi$ and $T$. See details of this in \cite[Lemma $3.3$]{bafe2}. This can
only be done if and only if $C_1$ and $C_2$ are adjacent.
We assume that $T$ is this slightly adjusted torus.

After this modification each boundary component of $\hatp$ is transverse to both
$\Lambda^s_P$ and $\Lambda^u_P$. Since $T$ is transverse to $\hatp$ outside of the periodic
orbits of $\hat \Phi$, then we only have to check what happens at a tangent orbit
$\theta \subset T$. Lift to $\mi$ producing lifts $\tilde \theta$ and $\widetilde T$ and
lozenges $D_1, D_2$ abutting $\tilde \theta$.
Since the lozenges $D_1, D_2$ are
not adjacent at $\tilde \theta$, then both $\wls(\tilde \theta)$ and
$\wlu(\tilde \theta)$ separate $D_1$ from $D_2$. It follows
that $\Lambda^s_P, \Lambda^u_P$ are transverse to $T$ at $\theta$. If $D_1, D_2$ were adjacent
along $\wls(\theta)$ then $\Lambda^u_p$ would be tangent to
$T$ at $\theta$ and not transverse. After these modifications we prove the following:

\begin{proposition}{}{}
The foliations $\Lambda^s_P, \Lambda^u_P$ induce non singular foliations in $\hatp$, which
are transverse to $\partial \hatp$. They are denoted by
$\hatwls, \hatwlu$. These foliations are $\mathbb R$-covered.
Each leaf of $\hatwls$
intersects every component of $\hatp$ in a single component.
\label{folihat}
\end{proposition}

\begin{proof}{}
We do the proof for $\hatwls$. By the preliminary step $\Lambda^s_P$ induces a foliations in $\hatp$
which is possibly singular with $p$-prong singularities. The preliminary step shows that there
is no local leaf of $\Lambda^s_P$ restricted to $\hatp$ which intersects $\hatp$ only in a tangent
periodic orbit of $\hatp$. If that were the case then in the universal cover no stable
prongs of the lifted periodic orbit would separate the lozenges abutting that orbit.
Hence these lozenges would be adjacent along an unstable leaf.
This was dealt with in the preliminary step. It follows that the restriction of $\Lambda^s_P$ to
$\hatp$ is an actual foliation in $\hatp$ with possible $p$-prong singularities.

Now double $\hatp$ along the boundary and double the foliations as well.
This produces a manifold $2 \hatp$ with a foliation $2 \hatwls$.
We have the following
properties:

\begin{itemize}

\item $\hatp$ is Seifert fibered and hence $2 \hatp$ is also Seifert fibered.

\item $2 \hatwls$ is a (possibly singular) foliation.

\item Since every tangent orbit to $\partial \hatp$ has at least one prong of $\Lambda^s_P$
entering $\hatp$, then in $2 \hatp$ there are at least two prongs of $2 \hatwls$ for
each such orbit.

\item Therefore the singularities of $2 \hatwls$ are $p$-prong singular orbits
with $p \geq 2$.

\end{itemize}

Now we use the theory of essential laminations  \cite{Ga-Oe} to deal with this situation.
If there are singular leaves of $2 \hatwls$ blow them up to produce a lamination
$\mathcal L$ in $2 \hatp$.  We will prove that $\mathcal L$ is an essential lamination.

Suppose first there is a compact leaf $C$ of $\mathcal L$.
First we have to rule out sphere leaves of $\mathcal L$ and
tori leaves bounding solid
tori. A lot of the arguments are standard in $3$-dimensional topology
and some details are left to the reader.
We think of $\hatp$ as contained in $2 \hatp = \hatp \cup P^*$.
Suppose that $C$ is a sphere.
If $C$ is contained in $\hatp$ or in $P^*$ that
produces a compact leaf of $\wls$ in $\mi$, contradiction. It follows that
$C$ intersects $\partial \hatp$. Now using the components of $C - \partial \hatp$
and innermost arguments there is a component of $C - \partial \hatp$
which is a disk $D$. Without loss of generality assume that $D \subset \hatp$.
Consider the flow $\hat \Phi$ in $\hatp$. If $\partial D$ is not a tangent
orbit of the flow then the flow $\hat \Phi$ is transverse to $T$ in $\partial D$.
Hence it is transverse to $\partial D$. Suppose it is incoming in $\partial D$.
This is impossible as it would generate a center singularity of the flow
in $D$. If $\partial D$ is tangent to the flow a similar argument ensues.
It follows that $C$ cannot be a sphere.

Suppose now that $C$ is a torus. Consider a component $E$ of $C - \partial \hatp$.
Assume that $E$ is contained in $\hatp$. Since $C$ is compact then $C$ is the
double of $E$. So the problem can only happen if $E$ is an annulus.
The case that $C$ bounds a solid torus in $2 \hatp$
can only happen if
$E$ is an annulus that is boundary parallel in $\hatp$.
So $E$ together with an annulus in $\partial \hatp$ bound a solid torus $V$ in $\hatp$.
In particular both boundary components of $E$ intersect the same component
$T$ of $\partial \hatp$.
Lifting to the universal cover we obtain a lift $\widetilde E$ of $E$ which
has two boundary components in the same component $\widetilde T$ of $\partial
\widetilde P$ because $V$ is a solid torus. So this produces a leaf of
$\wls$ which intersects $\widetilde T$ in at least two components. This is a
contradiction as $T$ is a quasitransverse Birkhoff torus. Hence the leaves of
$\mathcal L$ cannot be tori bonding solid tori.


Let $W$ be the closure of
a  complementary component of $\mathcal L$ and $Y$ a component of
$\partial W$. We need to prove that $Y$ is incompressible in  $W$.
Suppose this is not the case. Then there is a simple
closed curve $\alpha$
in $Y$ which is not null homotopic in $Y$ but bounds a disk $D$ in $W$.
We look at the intersection of $D$ and $\partial \hatp$ $-$ the surfaces
are supposed to be in general position with respect to each other.
In $D$ look at an innermost arc intersection $\delta$ with $\partial \hatp$.
This curve
$\delta$ bounds a subdisk $D_1$ of $D$ (that is $\delta$ and
an arc in $\partial D$) with no arc intersections with $\partial
\hatp$. If there is a circle intersection in $D_1$, it is null homotopic
in $D$ and hence bounds a disk in $\partial \hatp$ as well. These two
disks can be used to create a sphere which bounds a ball and such
intersections can be isotoped away.
It follows that we may assume that the interior
of $D_1$ does not intersect  $\partial \hatp$. The complementary regions
of $\mathcal L$ restricted to $\partial \hatp$
(that is, $(2\hatp - \mathcal L) \cap \partial \hatp$, \
are either annuli $-$ when
one blows up a closed curve in $\hatp$; or infinite strips, if one
blows up an infinite curve in $\partial \hatp$.
Since the two endpoints of $\delta$
are in $D$ which is a disk and has connected boundary; it follows that
the endpoints of $\delta$ are in $Y$. It also follows that $\delta$
union an arc $\epsilon$ in $Y$  bounds
a disk $D_2 \subset \partial \hatp$.
Furthermore $\delta$ and an arc in $\partial D$
 bounds a subdisk $D_1 \subset \hatp$, where $D_1$ is a subdisk
of $D$.
Then $D_1 \cup D_2$ is a disk $D'$ which
has boundary in a leaf of $\mathcal L$.
Now $D'$ is a disk enitrely contained in $\hatp$. And in $\hatp$ it is easy to see
that this component of $Y \cap \hatp$ is incompressible in $W$. Hence
$\partial D'$ also bounds a disk in $Y$.
Since  $\partial D'$ is null homotopic in $Y$,
then the curve $\alpha$ cut up by $D_2$
and union with the arc $\epsilon$ produces
another simple closed curve in $Y$ which is not null homotopic. Proceed with
one less intersection with $\partial \hatp$. By induction we arrive finally
at a contradiction and therefore $Y$ is incompressible in $W$.

In a similar way, using that there are
no one prongs in $\Lambda^s$, then the boundary
leaves of $K$ are also end incompressible. One of the main results of \cite{Ga-Oe}
implies that $\mathcal L$ is an essential lamination in $2 \hatp$. Brittenham's
theorem \cite{Br} implies that $\mathcal L$ has a minimal sublamination
$\mathcal L_1$ which is vertical or horizontal.

Suppose that $\mathcal L_1$ is vertical. Let $F$ be a leaf of $\mathcal L_1$. Then
$F$ has a closed curve $\alpha$
contained in say $\hatp$ so that $\alpha$ is
not null homotopic in $\hatp$. Since it is vertical it is freely homotopic to
a regular fiber in $\hatp$ and projects to a curve which is not null homotopic
in $M$. So in $M$ the stable leaf which contains it is not a plane
and contains a periodic orbit $\beta$. But then the fiber in $P$ is freely
homotopic to a periodic orbit up to powers. This would imply that $P$ is
a periodic piece, contrary to assumption in this case.

\vskip .1in
We conclude that $\mathcal L_1$ is a horizontal sublamination. Since $\mathcal L_1$
is horizontal it follows that $\mathcal L$ is horizontal also, because one can do the
isotopies preserving the tori in $\partial \hatp$. In particular this also
implies that there are no singularities of $2 \hatwls$. We conclude the following:
\begin{itemize}
\item $\hatwls$ is a non singular horizontal foliation in $\hatp$.

\item Recall that each component of $\partial \hatp$ is a Birkhoff torus.
It now follows that every compact component
$\tilde \theta$ of $\hatwls \cap \partial \hatp$
is a closed orbit of $\Phi_p$ and locally the stable leaf of $\tilde \theta$
has only one prong entering $\hatp$.
Notice that there
may be singularities of $\Phi_P$ contained in $\partial \hatp$, but only
one stable prong of such an orbit enters $\hatp$.

\item Given these facts it is immediate that in the universal cover of $\hatp$ every
leaf of the lifted foliation $\widetilde \hatwls$ intersects every
component of $\partial \widetilde \hatp$. In particular $\hatwls$ is
an $\rrrr$-covered foliation in $\hatp$.
\end{itemize}

The fact about $\tilde\theta$ in the second claim is true
because $\tilde\theta$ is a closed curve in a Birkhoff
torus and it is contained in a single stable leaf.
By the structure of the induced stable foliation
in Birkhoff annuli, it follows that $\tilde\theta$
has to be a boundary component of one of the annuli
in the Birkhoff torus, and hence it is a closed
orbit of the flow.

We now prove the last statement of the proposition.
First we note the following fact: consider an orbit $\theta$ of $\Phi_P$ that intersects
a component $T^-$ of $M_P - \hatp$ and let $p$ be a point of $\theta$ in this component.
Suppose that the forward orbit of $p$ is always in this component, that is it never
enters $\hatp$. Then if the forward orbit of $p$ only limits in points of the torus
boundary $T$ of $T^-$ it follows that $p$ is in the stable leaf of a boundary tangent
orbit contained in $T$. This follows from the fact that $T$ is a Birkhoff torus and the
hyperbolic dynamics near a periodic orbit.

Let now $L$ be a leaf of $\Lambda^s_P$. We show that $L$ intersects $\hatp$ in a single
component (could be empty). Let then $p, q$ be points in $L \cap \hatp$.
Let $\theta, \beta$ be the $\Phi_P$ orbits through $p, q$ respectively.
Suppose first that the forward orbit of $p$ gets out of $\hatp$. Then it first exits
$\hatp$ through a boundary Birkhoff torus $T$. As explained before $\theta$ can intersect
$T$ at most once. Hence once $\theta$ leaves $\hatp$ through $T$ it has to stay in
the component $T^-$ of $M_P - \hatp$ bounded by $T$.
We claim that $\theta$ cannot be in the stable leaf of a boundary periodic orbit
contained in $T$. This is a key fact here. Suppose not.
Then $L$ is the stable leaf
of this periodic orbit $\theta_1$. But $L$ also intersects $T$ transversely in the
orbit through $p$. Consider a segment in $\theta$ from the intersection with
$T$ and then to a point very near the boundary orbit $\theta_1$.
Add a small arc at
the end to connect it to $\theta_1$ and produce segment $S$. Then $S$ is a segment
in the closure of $T^-$ intersecting $T$ only in the endpoints. Since $\pi_1(M_P) = \pi_1(P)$,
this segment is homotopic rel endpoints to an arc in $T$. Hence when we lift to $\mi$,
producing $\widetilde T, \widetilde \theta, \widetilde L$ we obtain that $\widetilde \theta$
is transverse to $\widetilde T$ $-$ in the lift of the interior of a Birkhoff annulus, hence
in a lozenge. In addition $\widetilde \theta$ is also in the stable leaf of one of the lifts of a  periodic orbit,
that is in a corner of one of the lozenges associated with $\widetilde T$. This is a
contradiction. This shows that $\theta$ cannot forward accumulate only in
$T$.

Since $q$ is also in $L$ it is forward asymptotic with the $p$ orbit,
 and so has to enter
the component $T^-$ as well (as $\theta$ gets sufficiently far from $T$). So we have the
following:
Points $p_0, q_0$ in the forward orbits of $p, q$ that are in $T$. From $p$ to $p_0$ the
orbit $\theta$ is contained in $\hatp$ and after $p_0$ the orbit $\theta$ is entirely contained
in $T^-$.
Similarly for $q$ and the orbit $\beta$.
Go forward enough on both orbits so they are close. Then we produce an arc in $L$ as follows:
start in $p_0$, move forward on $\theta$ until very close to $\beta$, move to $\beta$ along $L$ and
then backwards to $q_0$. As above call this arc $S$. As above $S$ is homotopic into $T$.
As above lift to $\mi$ to produce $\widetilde T, \widetilde L, \widetilde p_0, \widetilde q_0$, etc..
By the homotopic property above, the two points
$\widetilde p_0, \widetilde q_0$ are in $\widetilde T$ and also in $\widetilde L$. But the
intersection of $\widetilde L$ with $\widetilde T$ is connected $-$ just consider the intersection
of $\widetilde L$ with the chain of lozenges associated with $\widetilde T$. Therefore
$\widetilde p_0, \widetilde q_0$ can be connected along $\widetilde T$ by an arc in
$\widetilde L$.
Therefore in $\hatp$, $p$ and $q$ are connected by an arc from $p$ to $p_0$ in $\hatp$, then
an arc along $L \cap T$ to $q_0$ then back to $q$. It follows that $p,q$ are in the same
component of $L \cap \hatp$.

This deals with the case that the forward orbit of $p$ exits $\hatp$.
This was the harder case. Let us now deal with the other case. Suppose first that $\theta$ forward
accumulates in a point in the interior of $\hatp$. Then $\beta$ also does and since the intersection
of $\beta$ with $\hatp$ is connected the forward orbit of $q$ is contained in $\hatp$. Then
$p,q$ are connected in $L \cap \hatp$.
If the forward orbit of $p$ only accumulates in a component $T$ of $\partial \hatp$, then it is
in the stable leaf of a boundary periodic orbit
$\theta_1$. Similarly for $q$. By the explanation
in the beginning, $q$ cannot be in an orbit that intersects $T$ transversely, therefore
the forward orbit of $q$ has to be entirely contained in $\hatp$. So again $p, q$ are connected
in $L \cap \hatp$.

This finishes the proof of Proposition \ref{folihat}.
\end{proof}

\subsection{Gaps are periodic and consequences}\label{sub:consequences}

In the previous section, we have shown that in the intermediate cover $M_P$ there is a submanifold $\hat{P}$ whose boundary is an union
of tori transverse to the foliations ${\Lambda}^s_P$, $\Lambda^u_P$, which are regular inside $\hat{P}$.

Let $\widetilde{P}$ be the lift of $\hat{P}$ to
$\widetilde M$ which is $\pi_1(P)$-invariant.
The induced foliations $\hat{\Lambda}^s$ and $\hat{\Lambda}^u$ on $\hat P$ are $\mathbb R$-covered, the last statement of Proposition
\ref{folihat} and the arguments in the proof of
Proposition \ref{folihat}  imply that the projection of these $\mathbb R$'s into the leaf spaces
of $\wls$ and $\wlu$ respectively is injective.
In addition these projections are obviously $\pi_1(P)$ invariant.
From this it follows that
the leaf space of $\widetilde{\hat \Lambda^s}$ is $\mathcal A^s$,
and the leaf space of
$\widetilde{\hat \Lambda^u}$ is $\mathcal A^u$, and that the projection of $\widetilde{P}$
in the orbit space is contained in the domain $\Omega'$
defined in Section \ref{sub:omegainorbitspace}.
Recall that $\Gamma = \pi_1(P)$.

Actually, for every lift $\widetilde T$ of a torus boundary $T$ of $\hat P$,
the projection of $\widetilde T$ in $\oo$ (or $\Omega'$)
is a chain of lozenges.
More precisely let $\widetilde A_j$ be a Birkhoff band, connected component of $\widetilde T$ with
the tangent orbits removed. It projects in $\Omega$ in a region of the form:
$$R_j = \{(x,y) \in \Omega' \; \mid \; x_j < x < x_{j+1}, \; y_j < y < y_{j+1} \}$$
Due to our adjustment procedure
(preparatory step in section \ref{sub:intercover}),
the set $R_j$  is not necessarily a lozenge;
it could be a $s$-scalloped or $u$-scalloped chain of lozenges,
where we also include the sides between adjacent lozenges.
We call it a \textit{generalized lozenge}. In section \ref{sec:minimalfixed},
we have proved that the corners $(x_i, y_i)$ of these
rectangles all lie in $\Omega'$.
Let $\gamma_0 \in \Gamma$ be an element generating the group
of orientation preserving elements of $\pi_1(\Gamma)$ fixing every $x_i$ and every $y_i$: according to Corollary \ref{minimal},
the fixed points of $\gamma_0$ that are in $\tilde{\mu}_s$ are precisely the $x_i$'s, and  the $y_i$'s are the fixed points of $\gamma_0$
that are in $\tilde{\mu}_u$. The other fixed points of $\gamma_0$ in $\mathcal A^s$ or in $\mathcal A^u$ correspond to subdivisions of
$R_j$ in lozenges as in section \ref{leaforbit}.

\begin{lemma}
$\widetilde A_j$ is an entrance (respectively exit) transverse band if and only if
$]y_j, y_{j+1}[$ is a gap of $\tilde{\mu}_u$ (respectively $]x_j, x_{j+1}[$ is a gap of $\tilde{\mu}_s$).
\label{entrance}
\end{lemma}

\begin{proof}
This lemma was proved in
\cite[Section $3.1$]{Ba3} in the case of $\rrrr$-covered
Anosov flows.
Here we give a simpler proof, and that
applies to our much more general situation.

First of all irrespective of entering or exit annulus,
we claim the following:

\vskip .05in
\noindent
{\bf {Claim:}} If $\gamma$ is in $\Gamma = \pi_1(P)$
and $\gamma(R_j)$ intersects $R_j$, then they are
equal and $\gamma$ is a power of $\gamma_0$.

Suppose that $\gamma(R_j)$ non trivially intersects
$R_j$. Since $R_j$ is a rectangle, it follows
that either $\gamma(]x_j,x_{j+1}[)
\subset ]x_j,x_{j+1}[$ or $\gamma(]y_j,y_{j+1}[)
\subset ]y_j,y_{j+1}[$ (but not both). In either
case there is $w$ in $R_j$ with $\gamma(w) = w$.
Projecting to $M$ produces a closed orbit of $\Phi$
intersecting one of the Birkhoff tori transversely
and so that it represents an element of $\pi_1(P)$.
As we explained previously this is impossible. This
proves the claim.

\vskip .05in
Let us consider the case where $\widetilde A_j$ is an exit transverse band.
Then, an orbit starting from a point in $\widetilde A_j$ goes outside $\widetilde P$
and never go back since it crosses $\wt$ at most once.
Assume that $\tilde{\mu}_s$ intersects $]x_j, x_{j+1}[$.
Then there is an iterate $\gamma x_j$ in $]x_j, x_{j+1}[$,
because $x_j$ is in $\tilde\mu_s$ and $\tilde\mu_s$
is the minimal set of the action.
Then $\gamma(R_j)$ intersects $]x_j,x_{j+1}[ \times
\mathcal A^u$. But since $\gamma(R_j) \cap R_j = \emptyset$,
it follows that this intersection is either
``below" $R_j$ (that is contains points $(x,y)$ with
$y < y_j$) or ``above" $R_j$.  Since $\beta_s(\gamma x_j)
\geq y_{j+1}$ we have to have the second option.
The band $\widetilde A_j$ does not intersect
the unstable leaf $y_{j+1}$ but intersects
every unstable leaf $y$ with $y < y_{j+1}$ in
$\mathcal A^u$ near $y_{j+1}$. So this band escapes
{\underline {down}} as it nears the unstable leaf
$y_{j+1}$. This is because in a stable leaf,
say $\gamma x_j$, as it nears the intersection with
$y_{j+1}$ one has to escape flow backwards for it not
to intersect $y_{j+1}$, as flow forwards all orbits
are asymptotic. It follows that $\gamma R_j$ is
in the component of $\mi - \widetilde T$
that is ``flow forwards" of $\widetilde A_j$.
But this is a contradiction to the property above.

This contradiction  shows that $]x_j, x_{j+1}[$ is disjoint
from $\tilde{\mu}_s$, i.e. is a gap.

%
%

In a similar way, one proves that if  $\widetilde A_j$ is an entrance transverse band, then $]y_j, y_{j+1}[$ is a gap.

In order to conclude,
we just have to prove that when  $\widetilde A_j$ is an exit transverse band, then $]y_j, y_{j+1}[$ is \textbf{not} a gap (the proof
that $]x_j, x_{j+1}[$ is not a gap when  $\widetilde A_j$ is an entrance transverse band is similar).

Assume by a way of contradiction that $]x_j, x_{j+1}[$ and $]y_j, y_{j+1}[$ are both gaps. Let $p_n = \gamma_n x_j$ a sequence of iterates accumulating non-trivially to $x_j$.
We can assume that no $p_n$ is fixed by $\gamma_0$, i.e. that no $\gamma_n\gamma_0\gamma_n^{-1}$ is a power of $\gamma_0$.
Since $]x_j, x_{j+1}[$ is a gap, we have $p_n < x_j$. On the other hand, for $n$ sufficiently big, we have $y_j < \beta_s(p_n)  \leq \beta_s(x_j)$.
According to Lemma \ref{le:alphabetamu}, the last inequality is strict, since $]p_n, x_j[$ cannot be a gap ($x_j$ cannot be the extremity of two different gaps).
Hence, since $y_{j+1} \geq \beta_s(x_j)$, the gap $]y_j, y_{j+1}[$ contains all the points $\beta_s(p_n)$, that are fixed points of $\gamma_n\gamma_0\gamma_n^{-1}$.
It follows that $y_j$ and $y_{j+1}$ are fixed by $\gamma_n\gamma_0\gamma_n^{-1}$. Therefore, $\gamma_n\gamma_0\gamma_n^{-1}$ are all powers of
$\gamma_0$. Contradiction.

The Lemma is proved.
\end{proof}

When $]x_j, x_{j+1}[$ is a gap, we define:
$$\Delta(x_j, y_j) :=  \{ (x,y) \in \Omega' \; \mid \; x_j < x < x_{j+1} , \; \alpha_s(x) < y < y_j \}$$

When the gap is $]y_j, y_{j+1}[$, we define:
$$\Delta(x_j, y_j) :=  \{ (x,y) \in \Omega' \; \mid \; \alpha_u(y) < x < x_j , \; y_j < y < y_{j+1} \}$$


The following is one of the most important facts
in our analysis:

\begin{proposition}\label{pro:periodicgap}
Every gap of $\tilde{\mu}_s$ (respectively $\tilde{\mu}_u$) is $\pi_1(P)$-periodic. More precisely, all gaps of $\tilde{\mu}_s$ and $\tilde{\mu}_u$ are sides of one generalized lozenges $R$
corresponding to a Birkhoff annulus in $\partial\hat P$.
\label{periodic}
\end{proposition}

\begin{proof}
Let $I$ be a gap of $\tilde{\mu}_s$. Assume it is not a side of some generalized lozenge corresponding to a Birkhoff annulus in $\partial\hat P$. Let $\widetilde \Lambda^s(I)$ be the preimage of $I$ in $\widetilde P$, and let $\hat \Lambda^s(I)$ be the projection of $\widetilde \Lambda^s(I)$ in $\hat P$. By hypothesis, $I$ cannot be one side of a generalized lozenge. Hence $\hat \Lambda^s(I)$
contains no point in an exit annulus. It follows that no orbit in $\hat \Lambda^s(I)$ can escape from $\hat P$.

But since $I$ is an open segment in $\mathcal A^s$ there is
a segment $J$ in an unstable leaf, transverse to the flow
and so that $J$ is contained in $\hat \Lambda^s(I)$.
By the above, the forward orbit of $J$ is entirely
contained in $\hat P$, and similarly for any point
in its closure. Since orbits in $J$ are expanding away
from each other we obtain unstable leaves entirely contained
in $\hat P$.
This contradicts the final statement of Proposition
\ref{folihat}.

This finishes the proof of the Proposition.
\end{proof}

%

Recall that at the end of section \ref{sec:minimalfixed} we have defined maps $\alpha_{s,u}^-$, $\beta_{s,u}^+$, and $\tilde{\tau}_{s} = \beta_u^+ \circ \beta_s^+$.

\begin{corollary}\label{cor:alphabetamu2}
The restrictions to $\tilde{\mu}_s$ of the maps $\alpha_s^-$, $\beta_s^+$, and the restrictions of $\alpha_u^-$, $\beta_u^+$ to $\tilde{\mu}_u$
are all injective.
\end{corollary}

\begin{proof}
We first prove that the restriction of $\alpha_s^-$ to $\tilde{\mu}_s$ is injective. Assume not:
there are two element $a$, $b$ of $\tilde{\mu}_s$ such that $\alpha_s^-(a) = \alpha_s^-(b)$.
As in Lemma \ref{le:alphabetamu} one proves that $]a, b[$ is a gap of $\tilde{\mu}_s$. According to Proposition \ref{pro:periodicgap},
$a$ and $b$ are preserved by a non-trivial element $\gamma_0$ of $\pi_1(P)$, and $]a, b[$ is the $s$-side of a $\gamma_0$-invariant generalized lozenge
$R$ corresponding to an exit Birkhoff band $\widetilde A$. Let $(a, u_1)$ be the ``lower left" corner of $R$
and $(b, u_2)$ the
``upper right" corner of $R$ in $\Omega'$.
It follows that they are both true blue corners.
Then as in Lemma \ref{composition},
we have that $\alpha_s^-(b) = u_1$.
But since we assumed that $\alpha_s^-(a) = \alpha_s^-(b)$,
we obtain $\alpha_s^-(a) = u_1$, which is impossible.

The proof of the injectivity of $\beta^+_s$ on $\tilde{\mu}_s$, and of $\alpha^-_u$, $\beta_u^+$ on $\tilde{\mu}_u$ are similar.
\end{proof}

\begin{proposition}\label{pro:taumu}
The image of $\tilde{\mu}_s$ by $\alpha^-_s$ is $\tilde{\mu}_u$. Similarly, $\beta_s^+(\tilde{\mu}_s) = \tilde{\mu}_u$ and $\alpha^-_u(\tilde{\mu}_u) = \beta^+_u(\tilde{\mu}_u) = \tilde{\mu}_s$.
In particular, $\tilde{\tau}_s$ induces an
order preserving  homeomorphism from $\tilde{\mu}_s$ onto itself,
and $\tilde{\tau}_u$ induces an order preserving
 homeomorphism from $\tilde{\mu}_u$ onto itself.
\end{proposition}

\begin{proof}
We just deal with $\beta^+_s$, the other cases being similar.
Let $s$ be an element of $\tilde{\mu}_s$ fixed by a non-trivial
element $\gamma_0$ of $\pi_1(P)$.
According to Proposition \ref{pro:periodicgap} and Corollary \ref{minimal} there is an element $u$
of $\tilde{\mu}_u$ such that $(s,u)$ is an element of $\Omega'$ fixed by a non-trivial element $\gamma_0$ of $\pi_1(P)$.
More precisely, $(s,u)$ is the lower left corner of a $\gamma_0$-invariant
generalized lozenge whose other (upper right)
corner is $(\beta^+_u(u), \beta^+_s(s))$ (see Lemma \ref{le:subeta}).
In particular, $\beta_s^+(s)$ lies in $\tilde{\mu}_u$. Since elements with non-trivial $\pi_1(P)$-stabilizers are dense in $\tilde{\mu}_s$,
we obtain $\beta_s^+(\tilde{\mu}_s) \subset \tilde{\mu}_u$.
Moreover, by the same arguments,
$\beta_s^+(\tilde{\mu}_s)$ contains all the elements of $\tilde{\mu}_u$
with non-trivial stabilizer. In order to conclude, we just have to prove $\tilde{\mu}_u \subset \beta_s^+(\tilde{\mu}_s)$.

Assume not. Let $u$ be an element $u$ of
$\tilde{\mu}_u - \beta^+_s(\tilde{\mu}_s)$.
Let $y$ be the bigger element of $\tilde{\mu}_s$
such that $\beta^+_s(s) < u$ for every $s < y$. Then, for every $s > y$ we have $\beta^+_s(s) > u$. By hypothesis, $\beta^+_s(y)$
is different from $u$. There are two cases:

-- either $\beta^+_s(y) < u$: in this case, for every $s$ in $\tilde{\mu}_s$ we have $\beta_s^+(s) \leq \beta_s^+(y)$ if $s \leq y$
(because $\beta_s^+$ is non decreasing) and $\beta_s^+(s) > u$ if $s > y$.

-- or $\beta^+_s(y) > u$: in this case, for every $s$ in $\tilde{\mu}_s$ we have $\beta_s^+(s) < u$ if $s < y$
and $\beta_s^+(s) > u$ if $s \geq y$.

In both cases, $\beta^+_s(\tilde{\mu}_s)$ is disjoint from the open interval $I$ with extremities $\beta^+_s(y)$ and $u$.
Since $\beta^+_s(\tilde{\mu}_s)$ is dense in $\tilde{\mu}_u$, it follows that $I$ is a gap of $\tilde{\mu}_u$.
By the previous Proposition it follows that
$u$ has
non-trivial $\pi_1(P)$-stabilizer.
It then follows by Lemma \ref{composition}
that $u$ is in $\beta^+_s(\tilde{\mu}_s)$. Contradiction.

The statements about $\tilde\tau_s$ and $\tilde\tau_u$
follow immediately.
\end{proof}

Once we proved that $\tilde\tau_s$ and $\tilde\tau_u$
are order preserving homeomorphisms, the following happens.
For every element $\gamma$ of $\pi_1(P)$:

-- $\tilde{\tau}_s \circ \gamma = \gamma \circ \tilde{\tau}_s$ if $\gamma$ preserves the orientation,

-- $\tilde{\tau}_s \circ \gamma = \gamma \circ \tilde{\tau}_s^{-1}$ if $\gamma$ reverses the orientation.

\begin{corollary}{}{}
$\alpha^-_u \circ \beta^+_s(u) = u$ for all $u$ in $\tilde\mu_u$
and $\alpha^-_s \circ \beta^+_u(s) = s$ for all $s$ in $\tilde\mu_s$.
\label{double}
\end{corollary}

\begin{proof}{}
The equalities are true in dense subsets of $\tilde\mu_u$,
$\tilde\mu_s$ respectively by Lemma \ref{composition}.
The previous proposition shows that the maps $\alpha^-_s,
\alpha^-_u, \beta^+_s,\beta^+_u$ are all homeomorphisms,
so the result follows.
\end{proof}

\vskip .06in
Observe that $\alpha_s(\tilde{\mu}_s)$ in general
is not necessarily contained in $\tilde{\mu}_u$, hence Proposition \ref{pro:taumu}
is false if we replace $\alpha_s^-$ by $\alpha_s$. But in the sequel we will need to understand where $\alpha_s$ and $\alpha_s^-$ may differ.

\begin{lemma}\label{le:s<s}
Let $s$ be an element of $\tilde{\mu}_s$. We always have $\alpha_s^-(s) \leq \alpha_s(s)$, and if the strict inequality
$\alpha_s^-(s) < \alpha_s(s)$ holds, then $s$ is an element of $\tilde{\mu}_s$ with non-trivial $\pi_1(P)$-stabilizer.
Similarly, we have $\beta_s(s) = \beta^+_s(s)$ unless $s$ has a non-trivial stabilizer; and for every $u$ in $\tilde{\mu}_u$,
we have $\alpha_u^-(u) = \alpha_u(u)$ and $\beta_u(u) = \beta^+_u(u)$ unless $u$ has a non-trivial stabilizer.
\label{power}
\end{lemma}

\begin{proof}
The inequality $\alpha_s^-(s) \leq \alpha_s(s)$ is obvious. Assume $\alpha_s^-(s) < \alpha_s(s)$.
Notice that $s \cap \Omega' \not = \emptyset$.
By Lemma \ref{interval} there is $p$ in $s$ and an
open interval $I$ in $\oou(p) \cap \Omega_P$ containing $p$
in the interior.
Let $u = \oou(p)$. It is an element of
$]\alpha_s(s), \beta_s(s)[$,
and by the above there is $s_0 < s$
such that for for every
element $s'$ of $]s_0, s[$ the point $(s',u)$ also lies in $\Omega'$.

Let now $u'$ be an element of $\mathcal A^u$ in the interval $]\alpha_s^-(s), \alpha_s(s)[$.
For $s' < s$, we have $\alpha_s(s') \leq \alpha_s^-(s) < u'$.
For $s'$ in $]s_0, s[$ we obtain $\alpha_s(s') < u' <  \beta_s(s')$,
the last inequality holding since $(s', u)$  lies in $\Omega'$.
Therefore,
for every $s'$ in $]s_0, s[$ and every $u'$ in
$]\alpha_s^-(s), \alpha_s(s)[$, the point $(s',u')$
lies in $\Omega'$. But $(s, u')$ is not in $\Omega'$
(since $\alpha_s(s) > u'$). It follows that $\beta_u(u') = s$ for every $u'$ in $]\alpha_s^-(s), \alpha_s(s)[$. The map $\beta_u$ is therefore
constant on $]\alpha_s^-(s), \alpha_s(s)[$. As in Lemma \ref{le:alphabetamu}, we obtain that $]\alpha_s^-(s), \alpha_s(s)[$ is contained in
a gap of $\tilde{\mu}_u$. By Proposition \ref{periodic}
this gap is periodic, and left invariant by some
$\gamma_0$ in $\pi_1(P) - id$, hence $s$, which is the image of  $]\alpha_s^-(s), \alpha_s(s)[$ by $\beta_u$,
also is invariant by $\gamma_0$, and
has a non-trivial stabilizer.
\end{proof}

\begin{lemma}\label{le:tauk}
There is an integer $k>0$ such that $\tilde{\tau}_s^k$ and $h$ coincide on $\tilde{\mu}_s$ and
such that $\tilde{\tau}_u^k$ and $h$ coincide on $\tilde{\mu}_u$.
\label{power}
\end{lemma}

\begin{proof}{}
Recall that $\tilde\tau_s = \beta^+_u \circ \beta^+_s$
and $\tilde\tau_u = \beta^+_s \circ \beta^+_u$.
Let $(s,u)$ be an element of $(\tilde{\mu}_s \times \tilde{\mu}_u) \cap \Omega'$ fixed by a non-trivial element $\gamma_0$ of $\pi_1(P)$.
Then $(h(s), h(u))$ is also a fixed point of $\gamma_0$ in $(\tilde{\mu}_s \times \tilde{\mu}_u) \cap \Omega$. There is a sequence of generalized lozenges
connecting $(s,u)$ to $(h(s), h(u))$. Moreover, if $s$ is an attracting fixed point, $h(s)$ is attracting too. It follows that there is an integer
$k(s)$ such that $h(s) = \tilde{\tau}_s^{k(s)}(s)$ and $h(u)=\tilde{\tau}_u^{k(s)}(u)$. Moreover, since $h^{-1}\circ \tilde{\tau}^\ell_s$ is continuous on $\tilde{\mu}_s$
for every integer $\ell$, the map $s \mapsto k(s)$ is locally constant on $s \in \tilde{\mu}_s$,
hence constant since $\pi_1(P)$-invariant.
\end{proof}

Here are a couple of remarks:

1) Suppose that $\gamma_0$ in $\pi_1(P)$  is associated
with a boundary
periodic in a Birkhoff torus.
Consider the collection $\{ x_i \},
\ i \in {\mathbb Z}$
of type S fixed points of $\gamma_0$ in $\mathcal A^s$.
Notice that this set is invariant
by $h$, hence bi-infinite.
The intervals $]x_i,x_{i+1}[$ are associated with
the lifts of the Birkhoff annuli and can be entering
or exiting, alternatively. Hence, by Lemma \ref{entrance},
 up to reindexing,
the intervals $]x_{2i},x_{2i+1}[$ are all complementary
components of $\tilde\mu_s$ and the intervals
$]x_{2i-1},x_{2i}[$ all intersect $\tilde\mu_s$.

2) If $\mathcal C$ is the chain of lozenges associated
with $\gamma_0$ and $\mathcal D$ is any other chain
of lozenges invariant by some ${\mathbb Z}^2 < \pi_1(P)$,
then $\mathcal D$ and $\mathcal C$ intersect each
other a lot: for every lozenge $C_i$ in $\mathcal C$,
then it intersects one at least one lozenge in $\mathcal
D$ and vice versa.

\begin{proposition}\label{pro:top}
Let  $\top$ be the map from $\widetilde U = \{ (x,y) \in \tilde{\mu}_s \times \tilde{\mu}_s \;|\; x < y < \tilde{\tau}_s(x)\}$
into $\mathcal A^s \times \mathcal A^u$ defined by:
$$\top(x,y) := (x, \alpha^-_s(y))$$
Then, $\top$ is a homeomorphism, with image
$(\tilde{\mu}_s \times \tilde{\mu}_u) \cap \Omega'$.
\end{proposition}

\begin{proof}
In order to prove the Proposition, we just have to prove
that the image of $\top(\widetilde U)$ is precisely
$(\tilde{\mu}_s \times \tilde{\mu}_u) \cap \Omega'$.
This is because all maps are homeomorphisms.

Let $(x,y)$ be an element of $\widetilde U$. The image $\top(x,y) = (x, \alpha^-_s(y))$ lies in $\tilde{\mu}_s \times \tilde{\mu}_u$.
We have to show that this image is in $\Omega'$.

\vskip .1in
\noindent
{\bf {Case 1:}} \
Consider first the case where $x$ is fixed by a non-trivial
element $\gamma_0$ of $\pi_1(P)$.

Then, there is a unique element $u$ of $\tilde{\mu}_u$ fixed
by $\gamma_0$ and such that $(s,u)$ lies in $\Omega'$. The point $(x,u)$ is the lower left corner of some generalized lozenge $R$. The
upper right corner
of $R$ is $(\beta^+_u(u), \beta_s^+(x))$. There is another generalized lozenge $R'$ with corners

$$(\beta^+_u(u), \beta_s^+(x)) \ \ \ {\rm and} \ \ \
(\beta^+_u(\beta_s^+(x)), \beta_s^+(\beta^+_u(u)))
 \ = \ (\tilde{\tau}_s(x), \tilde{\tau}_u(u)).$$

\noindent
Finally there is a generalized lozenge $W$ with corners
$(\alpha^-_u(u),\alpha^-_s(x))$ and $(x,u)$.
We refer to
figure \ref{fig:corners} in this proof.
Let $\mathcal C$ be the chain of generalized lozenges
(each a union
of adjacent lozenges) through $(x,u)$. It contains
$R,R'$ and $W$.
The only $\gamma_0$ fixed points of type S
in $\mathcal A^s$ in
$[x,\tilde\tau_s(x)]$ are the endpoints and $\beta^+_u(u)$.
Suppose first that $y$ is in $]x,\beta^+_u(u)]$.
If $y = \beta^+_u(u)$ then $\alpha^-_s(y) = u$ by Lemma
\ref{composition}, and hence $(x,\alpha^-_s(y))$ is in
$\Omega'$. Otherwise $y$ is in $]x,\beta^+_u(u)[$.
Clearly

$$\alpha_s(x) \leq \alpha^-_s(y) \leq
\alpha^-_s(\beta^+_u(u)) = u.$$

\noindent
So to have $\alpha^-_s(y)$ intersect $x$ we only
need to rule out $\alpha^-_s(y) = \alpha_s(x)$ and
not intersecting $x$.
Since $y$ is in $]x,\beta^+_u(u)[$,
$\tilde\mu_s$ intersects this interval.
In addition $]x,y[$ cannot be a gap of $\tilde\mu_s$,
as $y$ is not a $\gamma_0$ fixed point of type S.
In particular
there is a translate $\gamma x$ in $]x,y[$,
where we can assume $\gamma$ preserves orientation
in $\mathcal A^s, \mathcal A^u$.
Hence  the generalized lozenge $\gamma(W)$ in
$\gamma(\mathcal C)$ intersects the generalized
lozenge $R$ crossing it vertically:
the interior of $\gamma(W)$ intersects parts of the unstable
sides of $R$ and the closure of $\gamma(W)$  does not intersect
the stable sides of $R$.
The lower left corner of $\gamma(W)$ is
$(\gamma(\alpha^-_u(u),\alpha^-_s(x))$.
The next generalized lozenge in $\gamma(\mathcal C)$ in
the negative direction has to intersect $W$ horizontally
and hence intersects $s$.
This implies that

$$\alpha^-_s(y) \geq \alpha^-_s(\gamma(\alpha^-_u(u)))
\ > \alpha_s(s)$$

\noindent
so $\alpha^-_s(y)$ intersects $s$,
which is what we wanted to prove in this subcase.

Suppose now that $y$ is in $]\beta^+_u(u),\tilde\tau_s(x)[$.
Here $\alpha^-_s(y) \geq \alpha^-_s(\beta^+_u(u)) = u$,
so now we need to show that $\alpha^-_s(y) < \beta_s(s)$.
By arguments similar to the subcase we just finished,
there is $\gamma x$ in the interval
$]y,\tilde\tau_s(x)[$. Then $\gamma(W)$ ($W$ as above)
intersects $S$ horizontally, so
$\alpha^-_s(\gamma(x)) < \beta_s(x)$.
Since $\alpha^-_s(y) < \alpha^-_s(\gamma(x))$,
it follows that $\top(x,y) = (x, \alpha^-_s(y))$ lies in $\Omega'$
This finishes the analysis in Case 1.

\begin{figure}
  \centering
   \includegraphics[scale=0.8]{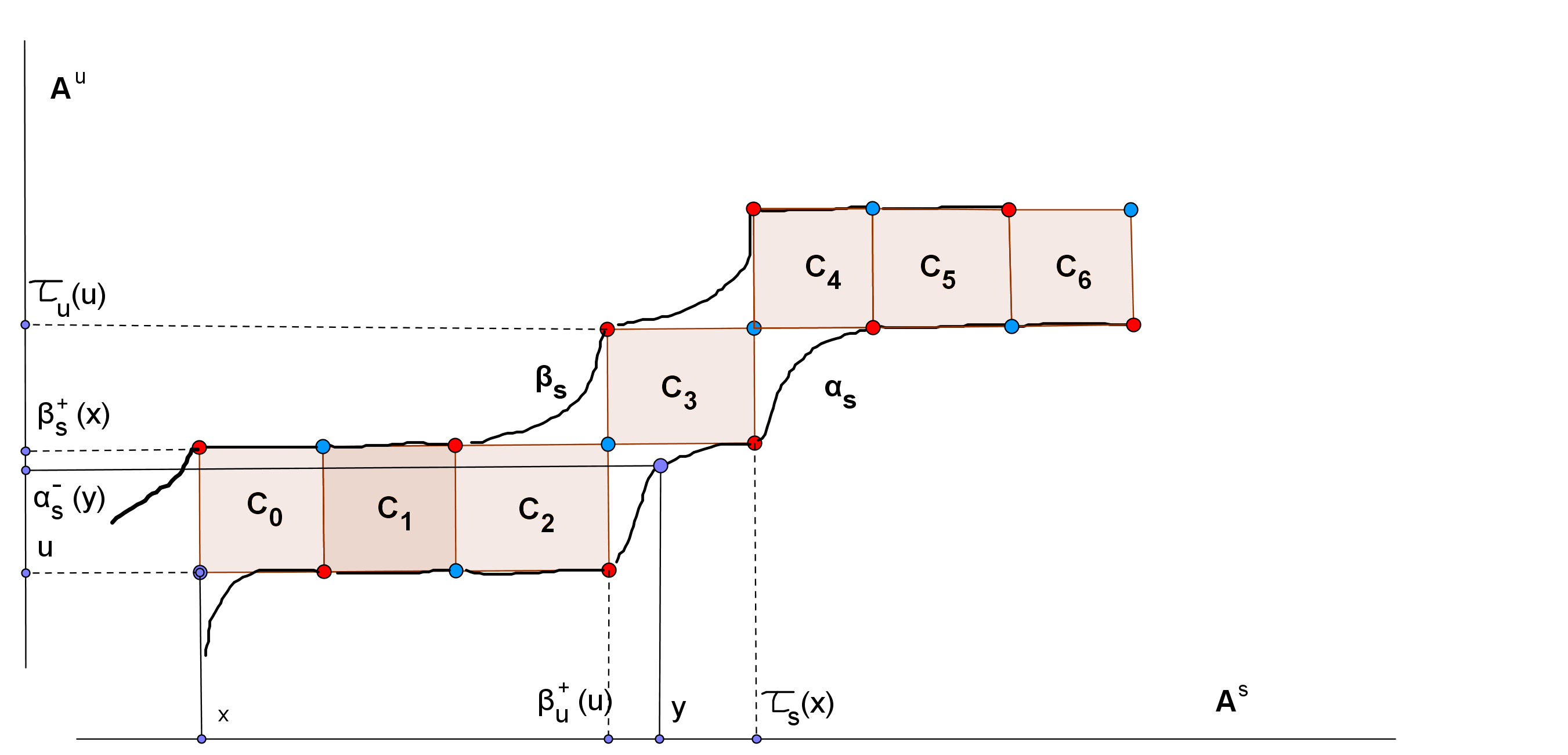}
\caption{The case where $x$ has a non-trivial stabilizer.}
\label{fig:corners}
\end{figure}

\vskip .08in
\noindent
{\bf {Case 2:}} \
From now on, we assume that $x$ has a trivial $\pi_1(P)$-stabilizer.

Since $x <y$, we have $\alpha_s(x) \leq \alpha_s^-(y)$. Moreover, if we have equality $\alpha_s(x) = \alpha_s^-(y)$, then $\alpha_s$
is constant on $]x,y[$. It follows (Lemma \ref{le:alphabetamu}) that $]x,y[$ is a gap of $\tilde{\mu}_s$,
because both endpoints are in
$\tilde\mu_s$. According to the important
Proposition \ref{pro:periodicgap},
$x$ has then a non-trivial $\pi_1(P)$-stabilizer, that we have excluded. Therefore:
$$\alpha_s(x) < \alpha_s^-(y)$$
By hypothesis, we have $y < \tilde{\tau}_s(x) = \beta^+_u(\beta^+_s(x))$. It implies $\alpha_s^-(y) \leq  \alpha_s^-(\beta^+_u(\beta^+_s(x)))$
By Lemma \ref{le:propertyalphabeta} we obtain $\alpha_s^-(y) \leq \beta^+_s(x)$.

We claim that this inequality is strict. Indeed, assume by a way of contradiction that the equality $\alpha_s^-(y) = \beta^+_s(x)$ holds.
Then $\alpha^-_s(z) = \beta^+_s(x) = \alpha^-_s(y)$
for any $z$ in $]y, \beta^+_u(\beta^+_s(x))[$. As we
have already observed several times, since $\alpha_s^-$ is then constant on $]y, \beta^+_u(\beta^+_s(x))[$, this segment is
contained in a periodic gap of $\tilde{\mu}_s$, preserved by a non-trivial element $\gamma_0$ of $\pi_1(P)$. But
$y$ and $\beta^+_u(\beta^+_s(x))$ are both elements of $\tilde{\mu}_s$,
hence they are both fixed points of $\gamma_0$.
We conclude that $(y, \beta^+_s(x))$ is a corner of a generalized lozenge $R$. But then we obtain $\alpha_s^-(y) < \beta^+_s(x)$: contradiction.

Therefore, as claimed, we have $\alpha_s^-(y) < \beta^+_s(x)$. According to Lemma \ref{le:s<s}, and since the $\pi_1(P)$-stabilizer
of $x$ is trivial, $\beta_s(x) = \beta^+_s(x)$: hence $\alpha_s(x) < \alpha^-_s(y) < \beta_s(x)$. These inequalities mean
that $\top(x,y) = (x, \alpha^-_s(y))$ lies in $\Omega'$: we have proved $\top(\widetilde U) \subset
(\tilde{\mu}_s \times \tilde{\mu}_u) \cap \Omega'$.

\vskip .05in
Conversely, let $(x,z)$ be an element of
$(\tilde{\mu}_s \times \tilde{\mu}_u) \cap \Omega'$. According to Proposition \ref{pro:taumu},
$z = \alpha_s^-(y)$ for some element $y$ of $\tilde{\mu}_s$. Clearly, since $(x, \alpha_s^-(y))$ lies in $\Omega'$, we have $y > x$.
Assume that $y \geq \tilde{\tau}_s(x)$: then by Corollary
\ref{double}

$$\alpha_s^-(y) \ \geq \ \alpha_s^-(\beta_u^+ \circ \beta_s^+(x))
\ \geq \ \beta_s^+(x).$$

\noindent
But this is impossible since $(x, \alpha^-_s(y)) = (x,z)$
is in $\Omega'$. Therefore, $y < \tilde{\tau}_s(x)$: $(x,y)$ is
an element of $\widetilde U$ such that $\top(x,y) = (x,z)$.
This finishes the proof of the proposition.
\end{proof}


\section{Proof of the Main theorem}
\label{conclusion}

We have now all the ingredients needed in the proof of the main theorem,
which we state in more precise terms:

\begin{theorem}
Let $(M, \Phi)$ be a pseudo-Anosov flow
in a closed $3$-manifold.
Let $P$ be a free Seifert piece in $M$. Assume that $P$ is not elementary, i.e. that $\pi_1(P)$ does
not contain a free abelian group of finite index. Then, in the intermediate cover $M_P$ associated to $\pi_1(P)$ there is a compact submanifold $\hat P$
bounded by embedded Birkhoff tori, such that the restriction of the lifted
flow $\hat \Phi$ to $\hat P$ is orbitally equivalent to a hyperbolic
blow up $(P_\Gamma(\Phi), \Phi)$ of a
geodesic flow associated to a convex cocompact subgroup $\Gamma \subset \widetilde{\mbox{PGL}}(2, \mathbb R)$ isomorphic to $\pi_1(P).$ More precisely,
this orbital equivalence maps the restricted foliations $\hat \Lambda^s$ and $\hat \Lambda^u$ to the restricted foliations
$\hat \Lambda^s_\Gamma$ and $\hat \Lambda^u_\Gamma$ (see Definition \ref{def:hyperbolicblowup}). Moreover, $\hat P$ is almost unique up to isotopy along the lifted flow $\hat \Phi$:
if  $\hat P'$ is another compact submanifold bounded by embedded Birkhoff tori, and if $\hat P_*$, $\hat P'_*$ are the complements in $\hat P$, $\hat P'$ of the (finitely many) periodic orbits contained in $\partial \hat P$,
there is a map $t: \hat P_* \to \mathbb R$ such that the map from $\hat P_*$ into $M_P$ mapping $x$ on $\hat{\Phi}^{t(x)}(x)$ is a homeomorphism, with image $\hat P'_*$.
\label{main}
\end{theorem}

The manifold $\hat{P}$ has boundary and the flow
$\hat{\Phi}$ is neither tangent nor transverse to the
boundary everywhere, so the situation is way more subtle
than a semi-flow in $\hat{P}$ transverse to the boundary.
The semiconjugacy sends orbits to orbits, where
an orbit may be defined only in an interval of the
parameter that includes any possible boundary points.
One added subtlety or difficulty
in the proof involves the tangent orbits in $\partial \hat{P}$.

\begin{proof}{}
The proof of this theorem will take all of this section and it will use the previous constructions
in the article.

\vskip .05in

We summarize what has been done in the
previous sections: the action $\rho_s$ of $\pi_1(P)$
on the stable leaf space $\hhs$ has an invariant axis $\mathcal A^s$,
that is homeomorphic to the reals and is properly
embedded in $\hhs$.
We identify $\mathcal A^s$ with the reals $\mathbb R$.
Let $\Gamma = \pi_1(P)$.
Let $h$ represent the regular fiber in $\pi_1(P)$. Then $h$ acts freely on $\mathcal A^s \cong \mathbb R$.
Since $< h >$ is a normal subgroup of $\pi_1(P)$ this induces an action
of $\bar{\rho}_s$ of $\bar{\Gamma} = \pi_1(P)/<h>$  on

$$\mathbb S^1 \ = \ \mathbb R/h \ \cong \ \mathcal A^s/h$$

The goal is to show that the induced representation $\bar{\rho}_s$ is
a hyperbolic blow up of a Fuchsian representation and then use the
results and constructions of section \ref{sec:blowing}.

The unique $\Gamma$-invariant minimal set $\tilde{\mu}_s$ projects to an invariant set $\mu_s$ in $\mathbb S^1$, which is the unique $\bar{\rho}_s(\bar{\Gamma})$-invariant minimal invariant set.

We also have a homeomorphism $\tilde{\tau}_s: \tilde{\mu}_s \to \tilde{\mu}_s$ almost commuting with $\rho_s(\Gamma)$ and an integer $k>0$
such that $\tilde{\tau}_s^k = h|_{\tilde\mu_s}$
(Proposition \ref{pro:taumu}, Lemma \ref{le:tauk}).

Then, exactly as is done in Proposition 3.18 of \cite{Ba3},
one can extend $\tilde{\tau}_s$ to $\mathcal A^s$ so that it commutes with $h$
and satisfies $\tilde{\tau}^k_s = h$.
Since it commutes with $h$, it follows that $\tilde{\tau}_s$ induces a homeomorphism $\tau_s$ of $\mathbb S^1$.
Let $\mathfrak I$ be the set of gaps of $\mu_s$ in $\mathbb S^1$.
Let $\sigma_s: \bar{\Gamma} \to S(\mathfrak I)$ be the representation describing the action of $\bar{\Gamma}$ on $\mathfrak I$.
It is very easy to see that $\bar{\rho}_s$ is
a $(\mu_s,\tau_s,\sigma_s)$-representation
on $\mathbb S^1$, that lifts to the
$(\tilde \mu_s,\tilde \tau_s, \tilde\sigma_s)$-representation $\rho_s$
on $\mathcal A^s$.

We want to show that $\bar{\rho}_s$ is a hyperbolic blow up of a Fuchsian representation.
We know that all gaps of ${\mu}_s$ are periodic (Proposition \ref{pro:periodicgap}).
First we modify the actions on gaps by ``hyperbolic'' modifications.
Instead of adding new fixed points as we did previously, we replace
the action of the stabilizer on the periodic gap under consideration
by a map $f_0$ which has no fixed points in the gap, and only fixes the
extremities.
Hence it is more accurate to call this process a hyperbolic \textit{blow down.}
This produces another representation
$\bar{\rho}'_s: \bar{\Gamma} \to$ Homeo$(\mathbb S^1)$.
Let $\pi_s: \mathcal A^s \rightarrow \mathbb S^1$ be the projection map.

\begin{lemma}{}{} The representation $\bar{\rho}'_s$ is topologically conjugate
to a Fuchsian representation.
\end{lemma}

\begin{proof}{}
We need to verify the conditions of Theorem \ref{thm:caracteriseconvergencegroup} that we reproduce here
for the readers's convenience (we recall that $\bar{\Gamma}_0$ is the index $2$ subgroup made of elements preserving the orientation):
\begin{enumerate}
  \item every gap of $\mu_s$ is periodic,
  \item for every $x$ in $\mathbb S^1$, the stabilizer of $x$ is trivial or cyclic,
  \item for every non-trivial element $\bar{\gamma}$ of $\bar{\Gamma}_0$ the fixed point set of $\bar{\rho}'_s(\bar{\gamma})$ is either trivial,
or one orbit of ${\tau}_s$, or the union of two orbits by $\tau_s$, one made of attractive fixed points and the other made of repellent fixed points,
  \item  if $(x_0, y_0)$ is a pair of fixed points of some element of $\bar{\Gamma}$ with $x_0 < y_0 < \tau_s(x_0)$,  then the $\bar{\rho}'_s(\bar{\Gamma}_0)$-orbit by the diagonal action of $(x_0, y_0)$ in the space $U = \{ (x,y) \in \mu_s \times \mu_s \;|\; x < y < \tau_s(x)\}$ is closed and discrete.
\end{enumerate}

Condition (1) is the content of Proposition \ref{pro:periodicgap}.

Condition (2) is also easy to deduce. The stabilizer of a point
by $\rho_s$ is at most cyclic, as $\Phi$ is a pseudo-Anosov flow.
Hence condition (2) was already satisfied by the representation
$\bar{\rho}_s$.
During the hyperbolic blow down we have
eliminated fixed points, so condition (2) follows.

Now consider condition (3). Let $\bar{\rho}'_s(\bar{\gamma})$ be an element of the
representation. Look back at $\bar{\rho}_s(\bar{\gamma})$ before the blow down.
Either it acts freely or has fixed points. If it has fixed points, then
it has an even number of points, that are consecutively attracting and repelling.
This is because when we pull back $\bar{\gamma}$ to an element
$\gamma$ of $\Gamma$ so that it has fixed points, then the fixed
points of $\gamma$ are discrete and are alternatively attracting
and repelling.
So the same holds for $\bar{\rho}_s(\bar\gamma)$.
When we perform the hyperbolic blow down, the fixed points that are not in
$\mu_s$ are removed. Hence the only remaining the fixed points
are those corresponding to elements in $\tilde\mu_s$.
This proves (3).

Finally consider condition (4). In Proposition \ref{pro:top}
we proved that the map $\top$ from

$$\widetilde U \ = \  \{ (x,y) \in
\tilde\mu_s \times \tilde\mu_s \;|\; x < y < \tilde\tau_s(x)\}$$

\noindent
into $\mathcal A^s \times A^u$ defined by $\top(x,y) = (x, \alpha_s^-(y))$ is a homeomorphism between $\widetilde U$ and
$(\tilde{\mu}_s \times \tilde{\mu}_u) \cap \Omega$. Let $(x_0,y_0)$
be an element of $\widetilde U$ fixed by some element $\gamma_0$:
its image by $\top$ is an element of $\Omega' \approx \Omega_P \subset \oo$ corresponding to a periodic orbit of $\widetilde\Phi$.
Since every periodic orbit in $M$ is compact, and
since every periodic orbit of $\Phi$ admits a neighborhood in $M$ in which there is no other periodic orbit freely homotopic
to it, it follows that the $\bar{\rho}_s(\bar{\Gamma}_0)$-orbit
of $(\pi_s(x_0), \pi_s(y_0))$ in $U = \{ (x,y) \in \mu_s \times \mu_s \;|\; x < y < \tau_s(x)\}$ is closed and discrete (for more details, see for example Lemma $3.20$ in \cite{Ba3}).
Now the $\bar{\rho}'_s(\bar{\Gamma}_0)$-orbit
of $(\pi_s(x_0), \pi_s(y_0))$ coincides with its
$\bar{\rho}_s(\bar{\Gamma}_0)$-orbit since the hyperbolic blow down
does not modify nothing in $\mu_s \times \mu_s$. Condition (4) is proved.

By Theorem \ref{thm:caracteriseconvergencegroup} it follows that
$\bar{\rho}'_s$ has the $(k)$-convergence property, and is topologically
conjugate to a Fuchsian group.
\end{proof}

The same happens with the representation $\bar{\rho}_u$ coming from the unstable axis $\mathcal A^u$: it is a hyperbolic blow up of a Fuchsian representation. Observe that
$\alpha_s^+$ induces a conjugacy between $\bar{\rho}_s$ and $\bar{\rho}_u$, at least on the minimal sets $\mu_s$ and
$\mu_u$. It follows that the Fuchsian representations conjugated to $\bar{\rho}'_s$ and $\bar{\rho}'_u$ are the same.
This is because a Fuchsian action is determined by
its action on the minimal set. It follows that
$\bar{\rho}_s$ and $\bar{\rho}_u$ are hyperbolic blow ups of
the \underline{same} Fuchsian representation $\bar{\rho}_0: \bar{\Gamma} \to$ PGL$(2,\mathbb R)$.

Since every fixed point in $\mathcal A^s$ or $\mathcal A^u$ is attracting or repelling,
the Fuchsian representation has no parabolic element
and no non-trivial element of finite order: it is a convex cocompact Fuchsian representation as defined in Definition \ref{defi:convexcocompact}.

\vskip .1in
Therefore, we are precisely in the situation described in section
\ref{sec:blowing}:
There is a hyperbolic blow up of a piece of geodesic flow
$(P_\Gamma(\Omega), \Psi, \hat\Lambda^s_\Gamma, \hat\Lambda^u_\Gamma)$ associated to the hyperbolic blow ups $\bar{\rho}_s$ and $\bar{\rho}_u$ of the convex cocompact
representation $\bar{\rho}_0$. As mentioned in Remark \ref{rk:uniqueblowup}, this piece does not depend on the choice of the maps $\alpha_1$ and $\beta_1$ involved in
the construction. The orbit space $D(\Psi)$ of $(P_\Gamma(\Omega), \Psi)$ is identified with the domain $D$ in
$\mathcal A^s \times \mathcal A^u \approx \mathbb R \times \mathbb R$ obtained
by removing from $\Omega'$ the triangles $\Delta(\tilde\theta)$ associated to
orbits $\tilde\theta$ tangent to lifted Birkhoff tori, and this identification is $\Gamma$-equivariant.  In other words, the Seifert piece $\hat P$, equipped with the restriction of $\hat\Phi$, and
the foliations $\hat\Lambda^s$, $\hat\Lambda^u$ has exactly the same
transverse structure as $(P_\Gamma(\Omega), \Psi, \Lambda^s(\Psi),
\Lambda^u(\Psi))$, meaning that there is a
$\Gamma$-equivariant homeomorphism
$$\Upsilon: \ D(\Psi) \
\to \ D$$
\noindent
between their orbit spaces.
The usual key idea is that $\Upsilon$ should lift to an orbital equivalence
between $(\hat P, \hat \Phi)$ and
$(P_\Gamma(\Omega), \Psi)$, which moreover maps $\hat\Lambda^{s,u}$ on
$\hat \Lambda^{s,u}_{\Gamma}(\Psi)$ since $\Upsilon$ respects
the horizontal/vertical foliations.

Actually this is not exactly true, since here we consider restrictions to Seifert pieces admitting toroidal
components with
{\underline {tangent}} periodic orbits: $(P_\Gamma(\Omega), \Psi)$ depends
on the choice on the boundary Birkhoff tori.

Moreover, there is an additional difficulty: the periodic
orbits contained in the boundary Birkhoff tori of $\hat P$
might be singular periodic orbits, with $p$-prongs ($p \geq 3$)
hence cannot be orbitally equivalent in their neighborhood to the
corresponding periodic orbit in a Birkhoff torus
$T \subset \partial P_{\Gamma}(\Omega)$.
Such a periodic orbit
lifts in $\widetilde M$
to an orbit $\tilde\theta$ contained in $\Omega_P$ but not in the interior of $\Omega_P$: it is the situation described at the end of section \ref{sub:omegainorbitspace}. The associated
triangle $\Delta(\tilde\theta)$ is then empty.


We solve this difficulty as follows: let $\hat\theta$ be a periodic
orbit of $\hat\Phi$, contained in a Birkhoff torus $T' \subset \partial \hat P$. Since the orbit space of the restriction of
$\hat\Phi$ to $\hat P$ is $D \subset \Omega'$, there is
only one component $F^s$ of $\Lambda^s_P(\hat\theta) - \hat\theta$
intersecting $\hat P$, the other components are contained in $M_P - \hat P$.
None of them can be in the interior
of $\Omega_P$ (the triangle $\Delta(\tilde\theta)$ is empty).
What we do is to artificially modify the flow $\hat{\Phi}$ in the
other side of $T'$ in $M_P$ in the neighborhood of
$\hat\theta$ so that $\hat{\theta}$ is still a periodic orbit of the modified flow $\hat{\Phi}'$, but now regular.

More precisely, we consider pairwise disjoint
small tubular open neighborhoods $U(T')$ in $M_P$ of each Birkhoff
torus $T' \subset \partial {\hat{P}}$, and the union
$\mathcal U$ of $\hat P$ with all these open domains $U(T')$.
We replace $\hat{\Phi}$ in $\mathcal U$ by a semi flow $\hat{\Phi}'$ such that:

-- $\hat{\Phi}$ and $\hat{\Phi}'$ coincide in $\hat{P}$,

-- every periodic orbit contained in $\partial \hat{P}$ is a $2$-prong periodic orbit for $\hat{\Phi}'$,

-- the restriction of $\hat{\Phi}'$ to every $U(T')$ is orbitally equivalent to an \underline{Anosov} flow (not pseudo-Anosov)  in the
neighborhood of an embedded Birkhoff torus.
By an abuse of notation we still denote by
$\hat{\Lambda}^{s,u}$ the stable and unstable foliations
in $U(T')$.

Recall that we proved in Remark \ref{rk:isotopybirkhofftori} that embedded elementary Birkhoff annuli are unique up to orbital equivalence in
their neighborhoods, meaning that any two of them admit tubular neighborhoods in which the restrictions of the flow are orbitally equivalent.
This is important because we will now show that the
arguments used in the proof in Remark \ref{rk:isotopybirkhofftori}
lead to the fact that $(\mathcal U, \hat{\Phi}')$ is orbitally
equivalent to the restriction of $\Psi$ to a neighborhood of
$P_\Gamma(\Omega)$ in $M_\Gamma(\Omega)$.

Indeed: first, select pairwise disjoint tubular neighborhoods
$U(T)$ for every $T$ in $\partial P_\Gamma(\Omega)$, so that
there are orbital equivalences $f_{T}$ between the restriction of $\Psi$
to $U(T)$ and the restriction of $\hat{\Phi}'$ to $U(T')$
(the torus $T'$ corresponding to $T$ being the one lifting to a
Birkhoff plane having the
same projection in $\Omega' \approx \Omega$
as a lift of $T$). Futhermore, $f_{T}$ sends
the restrictions of $\Lambda^{s,u}_\Gamma$ to $U(T)$ to the
restrictions of $\hat\Lambda^{s,u}$ to $U(T')$ - more precisely, the lift of
$f_{T}$ induces on the orbit space the restriction of $\Upsilon$
to the chain of lozenges in $\Omega$
associated with (a lift of) $T$ and the chain of lozenges in
$\Omega'$ associated to (a lift of) $T'$.
Observe that in general, $f_{T}(T)$ is \underline{not} $T'$,
but a small deformation of it, which is isotopic to
 $T'$ along the flow, outside the tangent periodic orbits.
Since this deformation is along the flow,  the foliation
$\hat{\Phi}'$ still coincides with the foliation $\hat{\Phi}$
on one side of $f_{T}(T)$.
We
{\underline {modify}} $\hat{P}$ so that we have $T= f_{T'}(T')$ for every $T'$.

Let $\mathcal U^*$ be the union in $M_\Gamma(\Omega)$ of $P_\Gamma(\Omega)$
and the $U(T')$'s, and equip $\mathcal U$, $\mathcal U^*$ with complete metrics, so that the orbits of the restrictions
of $\Psi$ to $\mathcal U^*$ and of $\hat{\Phi}'$ to $\mathcal U$, parametrized by unit length, are complete.
Notice that $\mathcal U^*, \mathcal U$ are {\underline {open}}.
From now on let $\Psi$ and $\hat\Phi'$ denote these restrictions, which  after these reparametrizations
are complete flows, not merely semi-flows.

By construction, the collection of
maps $f_{T}$ for $T$ boundary components
of $\partial P_{\Gamma}(\Omega)$,
 defines an orbital equivalence
between $\Psi$ and $\hat{\Phi}'$ in
the neighborhoods of $\partial P_\Gamma(\Omega)$ and $\partial \hat P$
that we want to extend in the interior of $P_\Gamma(\Omega)$.
Select a collection $(Z_i)_{i \in I}$ of open $2$-dimensional
disks in $P_\Gamma(\Omega)$ transverse
to $\Psi$ such that for some $\epsilon >0$, the flow boxes $U_i$
obtained by pushing $Z_i$ along $\Psi$ a time of absolute value
$<\epsilon$, together with all the $U(T)$'s,
define a locally finite covering of $\mathcal U^*$. We moreover can require that the flow boxes $U_i$ for $i \in I$ are all contained in the
interior of $P_\Gamma(\Omega)$.

For every $i \in I$, the projections
in $\Omega \approx \Omega'$ of every lift of $Z_i$ are open disks;
take the image of each of them in $\Omega'$ by the map $\Upsilon$,
lift it to $\widetilde{M}$. We can choose this last lift so
that it projects to a disk $Z'_i$ contained in $\hat{P}$,
and transverse to
$\hat\Phi' = \hat\Phi$ (in $\hat{P}$).
We then have a naturally defined homeomorphism $f_i: Z_i \to Z'_i$.

We  incorporate the collection of indices $I$ and the collection of boundary Birkhoff tori into a collection $J$. For every $j \in J$ is either an element of $I$, or a Birkhoff torus
$T$. In order to simplify the redaction,
in the last case, we define $Z_j$ and $Z'_j$ as $T$, $T' =f_{T}(T)$ respectively, and $U_j$ as the tubular neighborhood $U(T)$.

Now we do as in Remark \ref{rk:isotopybirkhofftori}: select a partition of unity $(\mu_j)_{i \in J}$ for the covering $(U_j)_{j \in J}$, and for every $x$ in $\mathcal U^*$ define $f(x)$ as the barycenter
of the various $\hat{\Phi'}^{t_j}(f_j(x_j))$, balanced with the
weight $\mu_j(x)$, where $x_j$ is an element of $Z_j$ and $t_j$ a real number in absolute value $<\epsilon$ such that $\Psi^{t_j}(x_j)=x$.

In this way, we obtain a continuous homotopy equivalence $f: \mathcal U^* \to \mathcal U$, mapping orbits of $\Psi$ into orbits of $\hat{\Phi}'$.
Observe that since the flow boxes $U_i$ for $i \in I$ are contained in the
interior of $P_\Gamma(\Omega)$, if $x \in \mathcal U^*$
is outside $P_\Gamma(\Omega)$, or even in
$\partial P_\Gamma(\Omega)$, then $f(x) = f_{T}(x)$ for some torus $T$:
$f$ restricted to the complement of  $P_\Gamma(\Omega)$ is a homeomorphism, with image the complement of $\hat{P}$ in $\mathcal U$.

Moreover, $f$ lifts to a map $\tilde{f}$ between the universal coverings
that, restricted to $P_{\Gamma}(\Omega)$
 induces the map $\Psi: \Omega \to \Omega'$ between the orbit spaces.
In particular, it follows that
$f$ maps leaves of $\Lambda^{s,u}(\Psi)$ into leaves of
$\hat\Lambda^{s,u}$, and it may fail to be injective only along
orbits of $\Psi$: if $f(x) = f(x')$, then there is some real number $t$ such that $x' = \Psi^t(x)$. There is a continuous map $(x,t) \mapsto u(x,t)$ such that for every $x$ in $\mathcal U^*$ we have
$f(\Psi^t(x)) = (\hat{\Phi}')^{u(x,t)}(f(x))$ (this continuous map is well-defined and unique despite of periodic orbits, it is due to the fact that the orbits lifted in the universal covering are non-periodic).

As mentioned in Remark \ref{rk:isotopybirkhofftori}, if we can find a real number $t_0 > 0$ such that $u(x,t) \neq 0$ for every $t > t_0$ for every $x$ in $\mathcal U^*$, then
one can average $f$ along $\Psi$-orbits and obtain a map $f_0: \mathcal U^* \to \mathcal U$ with the same properties than $f$, but which is moreover a homeomorphism (once more, for details,
we refer to Proposition $3.25$ of \cite{Ba3}).

The existence of such a $t_0$ will follow, by compactness
of $P_\Gamma(\Omega)$, if we are able to prove that for every $x$ in $\mathcal U^*$, there is an upper bound on the set of real numbers
$t$ such that $u(x,t)=0$. Assume by a way of contradiction that there is a sequence of elements $x_n$ of $\mathcal U^*$ and an increasing sequence $0 < t_1 < t_2 < ...$ of positive real numbers with no upper bound
such that $u(x_n, t_n)=0$ for every integer $n >0$.

If $\Psi^t(x)$ lies outside $P_\Gamma(\Omega)$ for some $t$, then,
since an orbit can exit $P_\Gamma(\Omega)$ at most once, we have
$\Psi^s(x) \notin P_\Gamma(\Omega)$ for all $s >t$ and therefore
$f(\Psi^s(x)) \neq f(x)$. Since $f(\Psi^{t_n}(x_n)) = (\hat{\Phi}')^{u(x_n, t_n)}(f(x_n)) = f(x_n) \in P_\Gamma(\Omega)$, it follows that for every $t$ in $[0, t_n]$ the iterate $\Psi^t(x)$ lies
in $P_\Gamma(\Omega)$. Since $P_\Gamma(\Omega)$ is compact, up to a subsequence, we can assume that $x_n$ and $y_n = \Psi^{t_n}(x_n)$ have limits $x_\infty$ and $y_\infty$. Moreover, $f(y_n)=f(\Psi^{t_n}(x_n)) = f(x_n)$, therefore $x_\infty$ and $y_\infty$ have the same image by $f$, and since $f$ is transversely injective, there is some real number $t_1$ such that $y_\infty = \Psi^{t_1}(x_\infty)$.

For every integer $n$, consider the loop $\gamma_n$ in $P_{\Gamma}(\Omega)$
starting from $x_\infty$, going to $x_n$, then following the $\Psi$-orbit during the time $t_n$ until reaching $y_n$, when going to $y_\infty = \Psi^{t_1}(x_\infty)$, and then going back to $x_\infty$ along the $\Psi$-orbit backward. The bigger $n$, the smaller the intermediate steps outside
the $\Psi$-orbits, and the bigger the time $t_n$ of travel along the first orbit. Due to the pseudo-Anosov character of $\Psi$, the action of $\gamma_n$
on the orbit space $D(\Psi)$ near the orbit of a fixed lift of $x_\infty$ is contracting on the vertical, and expanding on the horizontal, and the bigger $n$ is, the bigger are these expansions and contractions.

On the other hand, the image of $\gamma_n$ by $f$ is a loop homotopic to the loop $\gamma'_n$
in $\hat{P}'$,
going from $f(x_\infty)$ to $f(x_n)$, then going quickly to $f(y_\infty) = f(x_\infty)$. Indeed, we can replace the portion on the $\hat\Phi'$-orbits by constant maps, since $u(x_n,t_n)= 0 = u(x_\infty, t_1)$. Therefore, this loop is homotopically trivial and acts trivially on the orbit space $D$. It is a contradiction, since $\Upsilon$ is a conjugacy between the action of $\gamma_n$ on $D(\Psi)$ and the action of $\gamma'_n$ on $D$.

Therefore, we have shown that $(\mathcal U^*, \Psi)$ and $(\mathcal U, \hat{\Phi}')$ are orbitaly equivalent, in a way preserving the stable and unstable foliations It follows that
the same is true for the restrictions of $\Psi$ and $\hat\Phi$ to
respectively $P_\Gamma(\Omega)$ and $\hat{P}$, since $\hat\Phi$ and $\hat\Phi'$ coincide on $\hat{P}$.

This finishes the proof of Theorem \ref{main}, except that we still have to show that $\hat{P} \subset M_P$ is unique up to isotopy along the flow
outside the periodic orbits contained in the boundary. This
essentially follows from the ``uniqueness of Birkhoff tori up to flow isotopy'', established in \cite{bafe1}.
However, in \cite{bafe1} this statement was imprecise, since as pointed out in Remark \ref{rk:isotopybirkhofftori}, the flow isotopy does not in general
extend to the tangent periodic orbits. We clarify the situation here.

The key fact is that for any other compact core $\hat{P}'$ of $M_P$, admitting as boundary components embedded Birkhoff tori, every component $T'$ of $\partial \hat{P}'$ is homotopic
to a boundary torus of $\hat{P}$, hence with the same fundamental group (as subgroups of $\pi_1(M)$). Therefore, the chains of lozenges in $\Omega_P$ preserved by $\pi_1(T)$ and $\pi_1(T')$
are the same. Therefore, we can push along the flow every Birkhoff torus in $\hat{P}'$
onto the corresponding torus in $\partial\hat P$ (except the tangent periodic orbits that they already have in common).
Moreover, the order in which entrance/exit
tori in the boundary are crossed is encoded in the orbit space $D$: we can therefore push every boundary torus one-by-one, without perturbing at any
moment what has be done previously. It then follows that this procedure lead to an isotopy along the flow between the convex cores $\hat P'$ and $\hat P$
outside their boundary components.

Theorem \ref{main} is proved.
\end{proof}

\section{Examples of pseudo-Anosov flows and tori decompositions}
\label{examples}

Here we produce a variety of examples showing how two adjoining pieces $P,P'$ of the
JSJ decomposition can behave with respect to a pseudo-Anosov flow.
The examples we mention here were constructed in one way or another previously
in other references. Therefore we only describe them briefly with an emphasis on the
properties we want to analyze and refer to the appropriate references.
First a preliminary lemma:

\begin{lemma}{}{}
Let $\Phi$ be a pseudo-Anosov flow in $M^3$ and $P,P'$ adjoining pieces
along a torus $T$. Suppose that $P$ is a free Seifert fibered piece and
$\Phi$ is not orbitally equivalent to a suspension Anosov flow.
Then $T$ is not homotopic to a torus transverse to $\Phi$, that
is, it can only be homotopic to a quasi-transverse torus.
\end{lemma}

\begin{proof}{}
Let $h$ represent the fiber in $P$. According to Prop. \ref{axesfirst}
the axes $\aas, \aau$ are properly embedded lines
in the respective leaf spaces.
Since $\Phi$ is not orbitally equivalent to a suspension
Anosov flow, then
$\pi_1(T)$ leaves invariant a chain of lozenges $\mathcal C$
\cite{Fe5}. The lozenges
intersect the stable and unstable leaves in $\aas,\aau$. If two
lozenges $C_1, C_2$ are adjacent, then they either intersect the
same stable leaves or the same unstable leaves. Since $h$ acts monotonically on both
$\aas,\aau$ there are consecutive lozenges that are not adjacent.
According to \cite{bafe2} this implies that $T$ cannot be homotopic to a torus
transverse to $\Phi$.
\end{proof}

Note that there may be many consecutive lozenges in $\mathcal C$ that are adjacent.
The proof shows that {\em not all} such pairs can be adjacent.
Now we begin the examples.

\vskip .15in
\noindent
{\bf {Free adjoining free}}

Examples of these are the Handel-Thurston flows \cite{Ha-Th}.
For example suppose that $\Psi_0$ is the geodesic flow in a hyperbolic
surface $S$. Cut the manifold along the torus above a (say) separating, simple geodesic $c$ of $S$
and reglue with an appropriate Dehn twist. Handel and Thurston showed
 that the
resulting flow is Anosov. If $S_1, S_2$ are the closures of the complementary
components of $c$ and $M_1, M_2$ are the $\mathbb S^1$ bundles over these,
then they survive in the surgered manifold $M$. In addition they form the JSJ decomposition
of $M$. Finally both pieces are free.

\vskip .15in
\noindent
{\bf {Free adjoining periodic}}

Start with say the geodesic flow $\Phi_0$ in a hyperbolic surface $S$ of high
enough genus. Let $c$ be (say) a separating, simple geodesic of $S$. Let $A_0$ be a closed annulus neighborhood of $c$ in $S$,
and let $S_-, S_+$ be the  closures of the complementary components of $A_0$ in
$S$. Let $V_0, V_-, V_+$ be the unit tangent bundles of $A_0, S_-, S_+$ in $T^1S$.
The unit tangent bundle $T^1S$ is the union of $V_0$, $V_-$ and $V_+$.

Now do a blow up of one periodic flow line $\theta_+$ of $\Phi_0$ over $c$. In doing so, we create
3 periodic orbits, $\theta_1, \theta_2, \theta_3$, two of which are hyperbolic
and the last one (say) $\theta_3$ is an expanding orbit.
Remove a solid torus neighborhood of
$\theta_3$ to get a manifold $M_1$ with a semi-flow transverse to the boundary.
In the terminology of \cite{BBY}, $M_1$ is a
\emph{hyperbolic attracting plug}: the maximal invariant subset is hyperbolic, and there is a unique
entrance torus. The entrance lamination (see \cite{BBY}) is then a foliation. It is easy to check that
the stable leaves of $\theta_1$, $\theta_2$ intersects the entrance torus along closed leaves. Moreover,
since $\theta_1$ and $\theta_2$ are freely homotopic, the closed leaves in the entrance torus have the same contracting direction
(once more, in the terminology of \cite{BBY}). It follows that the entrance foliation admits at least one Reeb annulus - actually, a more
careful analysis as in \cite{Fr-Wi} would show that it is the union of two Reeb annuli.

Now consider a copy $M_2$ of $M_1$ but with the reversed flow.
In \cite{BBY}, B\'eguin, Bonatti and Yu showed that if we glue $M_1$ to $M_2$
by a map $h$ sending the exit foliation of $\partial M_2$ to a foliation of $\partial M_1$
transverse to the entrance lamination, then the resulting flow $\Phi$
is Anosov. Since the entrance and exit foliations have a Reeb annulus,
such a gluing map must send the isotopy class
of closed exit leaves to the isotopy class of closed entrance  leaves.
This is because
the only possible way for $h$ to map the boundary of a Reeb
annulus transversely to the entrance foliation (which has
a Reeb annulus) is
for $h$ to map  this curve to the isotopy class of closed
curves in the entrance foliation.

Each submanifold $V_-$, $V_+$ survives in $M_1$ and hence
in $M$. The same is true for their copies $V'_-$, $V'_+$ in $M_2$.
By construction these are Seifert fibered manifolds and they are free.
The remaining part of $M$ is the gluing of $W_0 \subset M_1$ with a copy of $W'_0 \subset M_2$, where
$W_0$ is $V_0$ with a solid torus removed. Observe that $W_0$
and $W'_0$ are both  diffeomorphic to the product of a punctured annulus with the circle.
In addition $\theta_1$ is a representative of the Seifert fiber
in $M_1$ and so is any closed curve in the entrance foliation
of $\partial M_1$. Similarly for the exiting foliation in
$\partial M_2$.
Therefore  condition on the gluing maps we pointed out in the previous
paragraph implies that the gluing map $h$ sends the Seifert
fibers of $W_0$ to the Seifert fibers of $W'_0$. Hence
the union $W = W_0 \cup W'_0$ is also a Seifert piece, diffeomorphic
to the product of the circle with the sphere minus fours disks. The fiber is freely homotopic to the periodic orbits $\theta_1$, $\theta_2$,
hence \textit{not} freely homotopic in $M$ with the fiber of $M_1$.  This shows that the pieces of the JSJ decomposition of $M$ are $V_-$, $V_+$, $V'_-$, $V'_+$ and $W$.
These pieces, except $W$, are free pieces for the Anosov flow in $M$. The piece $W$, adjoining all the others, is periodic.


This produces examples of Anosov flows with free pieces adjoing periodic pieces.

\vskip .15in
\noindent

Other examples were constructed in \cite{bafe1}: start with the geodesic flow
$\Phi_0$ in a closed hyperbolic surface $S$ with a symmetry over a (say) simple
geodesic $c$. Let $S_1, S_2$ be the closures of the complementary
components of $c$ in $S$ and let $M_1, M_2$ be the unit tangent bundles of $S_1, S_2$ respectively.
Do branched covers of $S$ along $c$ that generate branched covers of
the unit tangent bundle. In \cite{bafe1} we explain that the resulting
branched flow $\Phi$ is a pseudo-Anosov flow in a manifold $M$. Unlike in the previous examples, this
cannot be an Anosov flow.
The pieces $M_i$ lift to free Seifert pieces of the final flow $\Phi$.
The torus over the geodesic $c$ lifts to an $\mathbb S^1$ bundle over
a graph $E$ that has singularities.  This bundle is Seifert fibered. As in the previous example
one can easily show that this is a piece of the JSJ decomposition of $M$ and it is
a periodic piece. It is adjoining to the free pieces.
One can also do this construction with a disjoint union of geodesics in $S$.

\vskip .15in
\noindent
{\bf {Hyperbolic blow up}}

\begin{figure}
  \centering
   \includegraphics[scale=0.5]{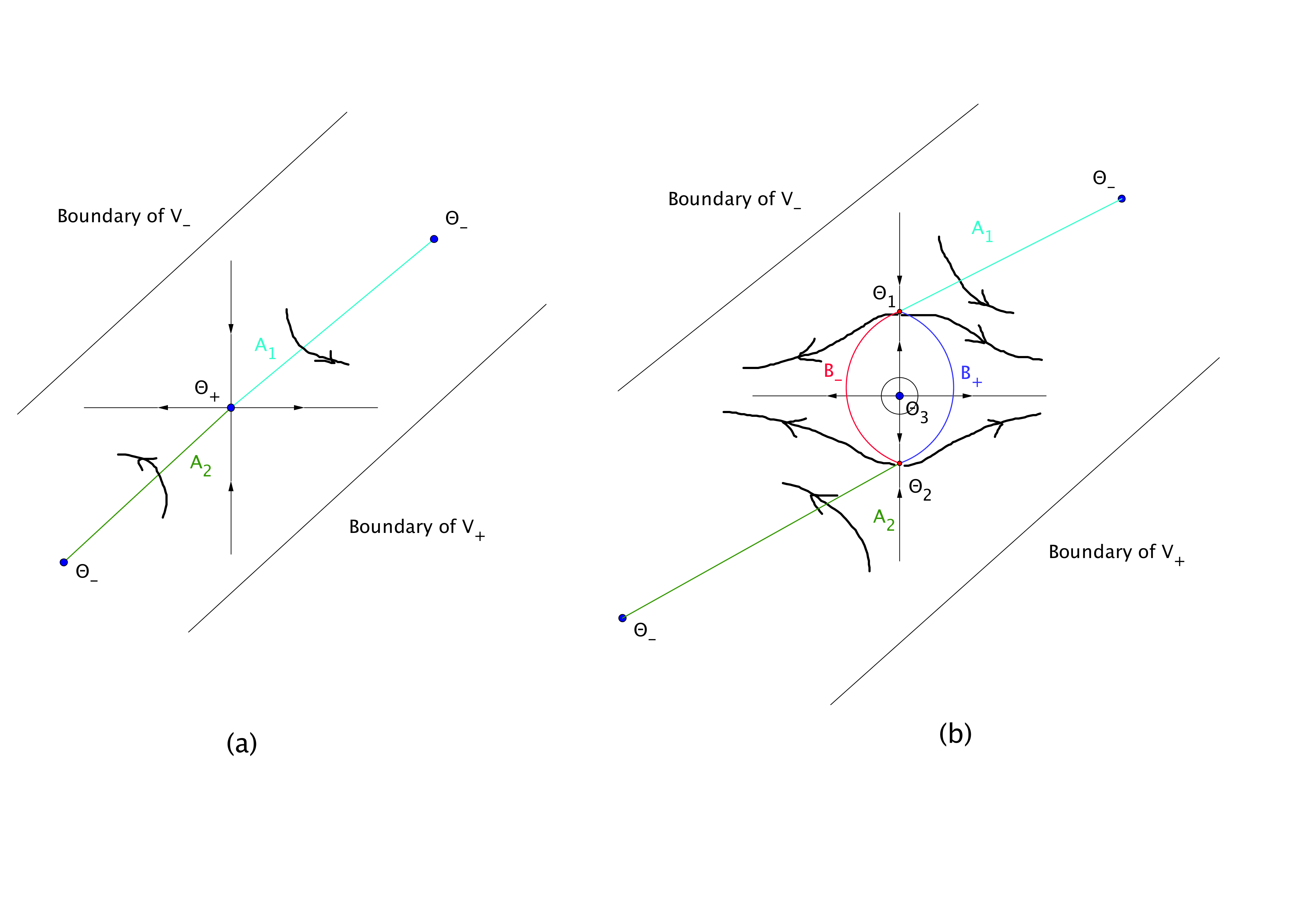}
\caption{Birkhoff annuli: (a) Before the DA blow
up fo orbit $\theta_+$,
(b) After the DA blow up
(the periodic orbit $\theta_3$ is the repelling one created during the DA operation, it is removed during the contruction.}
\label{fig:da}
\end{figure}

We further explain the first class of examples in the previous subsection
Free adjoining periodic. See Figure \ref{fig:da}.
Recall that $W = W_0 \cup W'_0$ is a
(periodic) Seifert piece. The original flow $\Phi_0$ had a Birkhoff torus associated with
the geodesic $c$. This torus was made up from two Birkhoff annuli between
the orbit $\theta_+$ and another orbit $\theta_-$ corresponding to the geodesic
traversed in the opposite direction.
After the DA blow up of the orbit $\theta_+$ creating periodic
orbits $\theta_1, \theta_2$ and $\theta_3$,
these two Birkhoff annuli are modified to two Birkhoff annuli
$A_1, A_2$, with
$A_1$ from $\theta_1$ to $\theta_-$ and $A_2$ from $\theta_2$ to $\theta_-$.
These annuli are in the submanifold $W_0 \subset M_1$.
There is an embedded annulus $B_-$ from $\theta_1$ to
$\theta_2$ contained in $W_0$ and so
that the torus $A_1 \cup A_2 \cup B_-$ is isotopic to $\partial V_-$
in $V_- \cup W_0$.  The annulus $B_-$ can be chosen transverse
to the blow up flow in $T_1 S$, and so is a Birkhoff annulus
for this flow, even though this flow is {\underline {not}}
Anosov or pseudo-Anosov.
There is also a Birkhoff annulus $B_+$ in $W_0$ so
that $A_1 \cup A_2 \cup B_+$ is isotopic to $\partial V_+$
in $W_0 \cup V_+$.

The union $A_1 \cup B_- \cup A_2$ is a Birkhoff torus for $\Phi$ in $M$
as well ($\Phi$ is Anosov), again isotopic to $\partial V_-$ in $M$.

Note that the Birkhoff annuli $B_-$ and $B_+$ cannot be elementary
for the Anosov flow $\Phi$.
This is because these Birkhoff annuli realize an oriented isotopy between
the orbits $\theta_1$ and $\theta_2$ oriented by the direction of the flow.
Therefore each of $B_-, B_+$ must be isotopic to a union of
an even number elementary annuli (cf. \cite{bafe1}).
We analize this situation in much more detail.
There is a symetric picture in $W'_0$, with three periodic
orbits $\theta'_1, \theta'_2, \theta'_-$ in $W'_0$ that are freely homotopic
to Seifert fibers and 4 Birkhoff annuli $B'_-$, $B'_+$, $A'_1$, $A'_2$.
Here $A'_1, A'_2$ are elementary and $B'_-,B'_+$
are not elementary.  All the periodic orbits in the boundary
of these Birkhoff annuli are fibers, but taking into account their orientation
given by the direction of the flow, we see that
$\theta_1$, $\theta_2$ are freely homotopic to $\theta'_-$, and to the \underline{inverse} of $\theta'_1$, $\theta'_2$, $\theta_-$.

Consider an elementary Birkhoff annulus containing $\theta_1$
in its boundary. There are at most four such elementary Birkhoff annuli, because the flow is Anosov and therefore
a given point in $\oo$ is the corner of at most four lozenges.
Recall that a quadrant is
a connected component of a small tubular neighborhood of
$\theta_1$ with the stable/unstable local
leaves removed. It is clear from the picture that one of the four
quadrants of $\theta_1$ cannot contain an elementary Birkhoff annulus:
the quadrant between the local stable and unstable
half-leaves crossing the boundary of $V_-$. The three others quadrants
contain $A_1$, $B_-$ and $B_+$. We conclude that each of them contains
one elementary Birkhoff annulus:, and
we already know one of them, which is $A_1$.
The other boundary component of any one of the two other Birkhoff annuli must be an oriented periodic orbit freely homotopic to the inverse of $\theta_1$:
it only can be $\theta'_1$ or $\theta'_2$.
We explain this in more detail. Any closed orbit $\alpha$ of $\Phi$
what is freely homotopic to $\theta_1$ cannot intersect $\partial M_1$
because $\partial M_1$ is transverse to the flow and separates $M$.
In addition if $\alpha \subset M_1$ then since $M_1, M_2$ are union
of Seifert spaces it follows that one can do cut and paste and produce
a free homotopy from $\alpha$ to $\theta_1$ entirely contained
in $M_1$. But the only periodic orbit of $\Phi$ in $M_1$ that is
freely homotopic to the inverse of $\theta_1$ in $M_1$ is $\theta_-$,
because of the structure of the geodesic flow in $T^1 S$.
The same reasoning applies to $M_2$, hence the other orbit
in the boundary of these Birkhoff annuli must be either
$\theta'_1$ or $\theta'_2$.

It follows that there is an elementary Birkhoff annulus $C_1$ between
$\theta_1$ and $\theta'_1$, and another elementary Birkhoff annulus $D_1$
between $\theta_1$ and $\theta'_2$.
In a symmetric way, there is an elementary Birkhoff annulus $C_2$
between $\theta_2$ and $\theta'_1$, and another elementary Birkhoff annulus
$D_2$ between $\theta_2$
and $\theta'_2$.
The union

$$T_1 = C_1 \cup D_1 \cup C_2 \cup D_2$$

\noindent
can be chosen to be embedded and it is isotopic to $\partial M_1
= \partial M_2$. Notice that $T_1$ has two periodic orbits
$\theta_1$ and $\theta_2$ in $M_1$ and two other periodic
orbits $\theta'_1, \theta'_2$ in $M_2$. The four Birkhoff annuli
above are transverse to $\partial M_1$ in their interiors.
Modulo changing the indices we can assume that
$\theta'_1$ is the orbit in the intersection of $C_1$ and $C_2$,
and $\theta'_2$ is the orbit in the intersection of $D_1$ and $D_2$.

Now we can describe the free Seifert pieces of $M$ with respect
to the flow. We can choose a representative of the free
Seifert piece $V_-$ in $M$ bounded by
the embedded Birkhoff torus $A_1 \cup A_2 \cup C_2 \cup C_1$.
Similarly $V_+$ has a representative with boundary
$A_1 \cup A_2 \cup D_2 \cup D_1$,
also $V'_-$ has a representative with boundary $A'_1 \cup A'_2 \cup D_1 \cup C_1$,
and finally  $V'_+$ has a representative with boundary
$A'_1 \cup A'_2 \cup D_2 \cup C_2$.

\vskip .1in
Finally we discuss the periodic Seifert piece $W$.
The periodic Seifert piece $W$ has a two dimensional spine that
can be chosen to be

$$Z \ = \ A_1 \cup A_2 \cup A'_1 \cup A'_2 \cup C_1 \cup C_2
\cup D_1 \cup D_2.$$

\noindent
The associated fat graph to this Seifert piece
(see explanation of this in \cite{bafe1}) has 8 edges
corresponding exactly to the 8 Birkhoff annuli in $Z$.
In addition it has 6 vertices corresponding to $\theta_1,
\theta_2, \theta'_1, \theta'_2, \theta_-, \theta_+$.
The first 4 vertices have valence 3 in the fat graph
and the last 2 have valence 2. For example
$\theta_1$ is a boundary of orbit of $A_1,C_1$ and $D_1$
only and $\theta_-$ is a boundary orbit of
$A_1$ and $A_2$ only.

\vskip .1in

In terms of actions on the circle, the flow $\Phi$ in
a Seifert fibered piece of $M$ can be achieved by blowing up the action on
one
of the intervals associated with the action in the original Birkhoff annuli
in $T_1 c$ to a homeomorphism with {\underline {two new hyperbolic fixed
points}} in the interior. This splits a lozenge into three adjacent lozenges.

\begin{remark}{\em
Let us denote by $\Phi_1$ the
blow up flow of $\Phi_0$ in $N = T^1S$.
The two Birkhoff annuli $A_1, A_2$ of $\Phi_1$ and the
annulus $B_+$ bound a submanifold $Y$ in $N$. The
annulus $B_+$ is transverse to $\Phi_1$ in the interior.
The flow $\Phi_1$ restricted to $Z$ is already exactly
in the format prescribed by the Main theorem (notice that
$Z$ is {\underline {not}} a Seifert piece of the
JSJ decomposition of $N$ as $N$ is Seifert fibered).
What we mean is that it is orbitally equivalent
to one obtained by the Hyperbolic blow up of geodesic
flow operation in section \ref{sec:blowing}.
{\underline {However}} the flow $\Phi_1$ in $N$ is
{\underline {not}} an Anosov flow, since it has
a repelling orbit. That is why one does the operation
of removing a solid torus neighborhood of the repelling
orbit and gluing a time reversal flow to
{\underline {obtain}} an Anosov flow in the final
manifold $M$. But the final flow in $M$ contains a copy
of the flow
$\Phi_1 |_{Y}$ in it and $Y$ is associated with
a Seifert fibered piece of the JSJ decomposition of $M$.
}
\end{remark}

\vskip .15in
\noindent
{\bf {Free adjoining atoroidal}}

We give some examples of Anosov flows. Start with a geodesic flow $\Phi_0$ in
a closed hyperbolic surface $S$ and let $c$ be a {\em non simple} geodesic
that fills a subsurface $S_1$ so that the complement is also not a union of
annuli. For simplicity assume that the complement $S_2$ is connected.
Let $M_i$ be the bundles over $S_i$.
Now do a high enough Dehn surgery on an orbit of $\Phi_0$ over $c$.
In \cite{Fe6} the second author proves that the resulting manifold
$M$ has two pieces in the JSJ decomposition: one corresponds to $M_1$ that
survives the surgery intact. The other, call it $M'$, is obtained from
the surgery in $M_2$. Here $M_1$ is Seifert fibered and free and $M'$ is
atoroidal \cite{Fe6}. This produces the examples.

\vskip .15in
\noindent
{\bf {Atoroidal adjoining atoroidal}}

One can get examples very similar to the ones in the previous case.
Use the same notation as in the previous setup, but now also do Dehn
surgery on an orbit over a geodesic $c'$ that fills $S_2$.
Then the resulting manifold has two atoroidal pieces in its JSJ decomposition.

\vskip .15in
\noindent
{\bf {Periodic adjoining atoroidal}}

Examples of Anosov flows were constructed by Franks and Williams \cite{Fr-Wi} in
their seminal paper.
These are not the original set of examples in that paper $-$ the famous intransitive
ones, but are slightly more complicated. Start with a suspension $\Phi_0$ do a
blow up of an orbit and remove a solid torus neighborhood to produce an atoroidal manifold
$M_1$ with a semi-flow. Now glue to a more complicated manifold: for example
it could be an $\mathbb S^1$ bundle $P$ over a twice punctured sphere.
The flow is the suspension of a simple diffeomorphism
 with Morse-Smale singularities.
This produces a Seifert piece $P$ in the resulting manifold and the piece
is periodic.

Examples of pseudo-Anosov flows that have singularities can be obtained by the
branched construction as explained in the free adjoining free examples as follows.
Before doing the branching, do Dehn surgery to have the original
manifold with two symmetric atoroidal pieces $M_1, M_2$ glued along a torus.
Then proceed as before. The atoroidal pieces lift to atoroidal pieces
in the resulting manifold. The torus lifts to a Seifert piece $P$ that is
periodic.

\vskip .15in
\noindent
{\bf {Periodic adjoining periodic}}

A very large class of examples were constructed in \cite{bafe1}.
The construction is not immediate and we refer the reader to \cite{bafe1}.
Then if
$P,P'$ are adjoining periodic pieces, we showed in \cite{bafe2} that the
adjoining torus $T$ is isotopic to a torus {\em transverse} to the
pseudo-Anosov flow. We showed that there are many examples that are Anosov.

\vskip .15in
\noindent
{\bf {Adjoining tori that are transverse}}

The original Franks-Williams examples have two atoroidal pieces glued
along a transverse torus. As for atoroidal adjoining a periodic piece, we explained above that
Franks-Williams produced examples where the adjoining torus is again transverse to the
flow.

So unless the piece is free, there are examples where the adjoining torus
is transverse to the flow.

{\footnotesize
{
\setlength{\baselineskip}{0.01cm}

}
}

\end{document}